\documentclass[12pt]{amsart}

\usepackage{hyperref}
\usepackage{amsmath,amssymb,latexsym,amsthm}
\usepackage{url}
\usepackage{graphicx,color}

\newcommand{\N}{\mathbb{N}} 
\DeclareMathOperator{\supp}{supp}
\DeclareMathOperator{\sign}{sign}

\newtheorem{theorem}{Theorem}[section]
\newtheorem{lemma}[theorem]{Lemma}
\newtheorem{proposition}[theorem]{Proposition}
\newtheorem {definition}[theorem]{Definition}
\newtheorem {example}[theorem]{Example}
\newtheorem {corollary}[theorem]{Corollary}
\newtheorem {remark}[theorem]{Remark}

\newcommand{\HOX}[1]{}

\def\R{{\mathbb R}}
\def\Z{{\mathbb Z}}
\def\expec{{\mathbb E}}

\def\vj{{\vec{j}}}

\newcommand{\norm}[1]{\left\|#1\right\|}

\def\hat{\widehat}
\def\tilde{\widetilde}

\def\id{{\rm I}}

\title[Inverse scattering for a random potential]{Inverse scattering for a random potential}
\author[Pedro Caro, Tapio Helin and Matti Lassas] {Pedro Caro$^*$, Tapio Helin$^\circ$ and Matti Lassas$^\circ$}

\date{\noindent \today}
\thanks{\noindent $^*$ Ikerbasque \& BCAM - Basque Center for Applied Mathematics. \\
\indent $^\circ$ Department of Mathematics and Statistics, University of Helsinki. }

\begin{document}

\begin{abstract}
In this paper we consider an inverse problem for the $n$-dimensional random Schr\"{o}dinger equation $(\Delta-q+k^2)u = 0$. 
We study the scattering of plane waves in the presence of a potential $q$ which is assumed to be a Gaussian random function such that its covariance is described by a pseudodifferential operator.
Our main result is as follows: given the backscattered far field, obtained from a single realization of the random potential $q$, we uniquely determine the principal symbol of the covariance operator of $q$. Especially, for $n=3$ this result is obtained for the full non-linear inverse backscattering problem. 
Finally, we present a physical scaling regime where the method is of practical importance.
\end{abstract}


\maketitle

\tableofcontents
\section{Introduction}

In inverse scattering theory one aims at determining an electric potential $q$ in $\R^n$ with $n\geq 2$ from measurements describing how this scatters certain incoming waves. {In many applications, the scatterer is so rough and vastly complicated that there is an apparent lack of systematic patterns in its micro-scale structure. In these situations, the potential is assumed to be created by a physical random process and the goal is not any more to recover the full potential but to determine some parameters or functions describing properties of its micro-structure.}
In this paper, we are interested in reconstructing statistical properties, more precisely, \textit{the local strength} of a potential $q$.

In the usual mathematical approach to the inverse scattering theory, one considers the scattering problem
\begin{equation}
\label{eq:intro_direct}
\left\{
\begin{aligned}
& (\Delta - q(x) + k^2) u(x) = 0 \quad\quad\quad {\rm in}\; \R^n\\
&u (x) = e^{ik\theta \cdot x} + u_{sc} (x)\\
&u_{sc}(x) \;\;  \textrm{satisfies the Sommerfeld radiation condition,}
\end{aligned}
\right.
\end{equation}
where the \textit{incident wave} is assumed to be the plane wave $e^{ik\theta \cdot x}$ and the \textit{scattered} and \textit{total} waves are denoted by $u_{sc}$ and $u$, respectively. The scattered wave satisfies the following asymptotic expansion
\begin{equation*}
	u_{sc}(x)  = c_n k^{\frac{n-1}2} |x|^{-\frac{n-1}{2}}e^{ik|x|} u^\infty\left(k, \theta, \frac{x}{|x|}\right) + o\left(|x|^{-\frac{n-1}{2}}\right),
\end{equation*}
where $u^\infty$ is known as the {\it far-field pattern} of $u_{sc}$. In this context, the inverse backscattering problem aims at answering the question:
\begin{itemize}
	\item[(Q)]\sl{Given the backscattered far-field pattern $u^\infty(k,\theta,-\theta)$ for multiple values of $k>0$ and $\theta \in {\mathbb S}^{n-1}$, what kind of information of $q$ can be recovered?}
\end{itemize}

The deterministic inverse back-scattering problem---which asks whether a potential $q(x)$ can be uniquely 
determined from its  backscattered far-field pattern $u^\infty(k,\theta,-\theta)$---is a longstanding
open problem. At this moment, the problem has been solved only under assumptions on controlled angular regularity  of the potential (see
\cite{MR3224125}). We discuss below the literature about this problem.
In this paper we consider 
a related stochastic   inverse problem where the statistical parameter functions of the potential $q$
are determined from the observations.
 
Since we are interested in the situations where the scatterer presents a random behaviour, we need to rephrase our approach to the inverse scattering theory. In order to do so, we assume the potential $\omega \in \Omega \mapsto q(x,\omega)$ to be a generalized random function in a probability space $(\Omega,\mathcal{H}, \mathbb{P})$. This makes the far-field pattern be random as well, which means that it changes with each realization $q(x,\omega)$. However, our approach consists of assuming the backscattering data $u^\infty(k,\theta,-\theta)$ with $k>0$ and $\theta \in {\mathbb S}^{n-1}$ to be generated by a single realization $q(x,\omega_0)$ for certain $\omega_0 \in \Omega$. Then, the inverse backscattering problem in this context asks to determine the parameters characterizing the probability law of $q$ from the backscattering data.

As we previously advanced, we reconstruct in this paper the local strength of the potential, which is one of the parameters describing the probability law of $q$. In order to provide an interpretation of this parameter, we will need to make some general assumptions on $q$. Firstly, we assume $q$ to be a generalized Gaussian field supported in a bounded domain $D$ and its expected potential $\mathbb{E} q$ to be a smooth function. Additionally, we assume the covariance function $K_q(x,y)$ to be smooth out of the diagonal, which means that the long distance interactions depends smoothly on their locations; we also assume the average roughness (or smoothness) of $q$ to remain unchanged for every sub-domain of $D$. However, we allow the size of this roughness to change in different sub-domains of $D$. The local strength of the potential measures or controls these different sizes. These assumptions\footnote{The random model is discussed in detail in the section \ref{sec:isotropic_rf}.} can be rigorously introduced, assuming that the covariance operator $C_q$ is a classical pseudodifferential operator (see for example \cite{Hor3}) of order $-m$ with $m > n-1$,
and such that, $C_q$ has
\begin{equation}
	\label{eq:intro_principal_symbol}
	\sigma(x, \xi) = \frac{\mu(x)}{|\xi|^m}
\end{equation}
as a principal symbol, with $\mu$ a smooth non-negative function supported on $D$---called the \textit{the local strength} of the potential. As we will see in Definition \ref{def:ml_iso}, this is to say that $q$ is a Gaussian \textit{microlocally isotropic} random field. Eventually, assuming $q$ as above with $\mu$ unknown, our goal will be to reconstruct $\mu$ from the backscattering data. As we will see in the section \ref{sec:regularity}, $\mu$ yields valuable control on the oscillations of $q$: where $\mu$ is large, the rough oscillations of $q$ are most likely large as well. 

In order to cover a broad spectrum of well-known random field models, we also include the possibility of realizations of $q$ being generalized functions almost surely (a.s. for short). In fact, whenever $n-1 < m \leq n$, we can only ensure that $q$ belongs to a Sobolev space with negative smoothness index almost surely. This consideration requires a carefully analysis\footnote{The readers who are expert on uniqueness for the Calder\'on problem for non-regular conductivities will notice the connection with the work \cite{MR2026763} of Brown and Torres.} of the forward problem with compactly supported potentials in the Sobolev spaces $L^p_{-s}(\R^n)=W^{-s,p}(\R^n)$ 
 with $0 < s \leq 1/2$ and $n/s \leq p < \infty$. Inspired by the works \cite{MR0466902, MR1230709}, we provide new insights to the classical scattering theory for rough potentials.
 
{A {microlocally isotropic} Gaussian random field $q$  of order $-m$  in $D\subset \R^n$ can be written
in the form $q=(C_q)^{1/2}W$,  where $W$  is a white noise. We will later see that $q \in L^p_{-s} (\R^n)$ almost surely for any $1 < p < \infty$ and $-s < (m - n)/2$.  The local strength $\mu$  
determines the roughest component of $q$  in the sense that if
$C_{\tilde q}$  is a properly classical pseudodifferential operator of order $-m$ having the
same principal symbol (\ref{eq:intro_principal_symbol}) as the operator $C_q$, then $\tilde q=(C_{\tilde q})^{1/2}W$
is also a {microlocally isotropic} Gaussian random field  of order $-m$ such that
$\tilde q-q\in L^p_{-s+1} (\R^n)$ almost surely for any $1 < p < \infty$ and $-s < (m - n)/2$,
that is, $\tilde q-q$  is one degree smoother than $q$.}

In applications the measurement data is often obtained as an average of signals at multiple frequencies. Also, in many standard references in the literature in applied sciences,
see e.g. \cite{ishimaru1978wave,Kong,PB,SH}, one  considers the effective equations for the expectations and covariances of the scattered waves. This means that one considers the averages of waves that are 
generated by many independent samples of the scatterers. This approach can be poorly justified if the scatterer changes slowly during the measurements or is independent of time.
In this paper, the data is assumed to be a weighted average of far-field patterns at a given separation $\tau>0$: 
\begin{equation}
	\label{eq:intro_measurement_data}
	M(\tau, \theta) = \lim_{K\to\infty} \frac{1}{K}\int_K^{2K} k^m u^\infty(k,\theta,-\theta) \overline{u^\infty(k+\tau,\theta,-\theta)} dk,
\end{equation}
where $\theta \in \mathbb{S}^{n-1}$. As we pointed out before, the far-field pattern is random and consequently our data is random as well. Since we want to show that the data generated by a single realization of $q$ allows us to reconstruct {the local strength }$\mu$---which is non-random, we will need to prove that the randomness averages out at the limit. Actually, we will prove that, for $n=m=3$, there exists a known constant $c>0$ such that
\begin{equation}
M(\tau, \theta) = c\, \widehat{\mu} (2\tau\theta)\quad \text{almost surely}.
\label{id:linear_relation}
\end{equation}
Formula (\ref{id:linear_relation}) means that the measurement function $M(\tau, \theta)$, computed from the measured far-field patterns, does not depend (with probability one) on the realization of the random potential $q$. Such measurement functions that are independent of the realization of the random media are said to be \emph{statistically stable}, see 
\cite{borcea2002,papa_book}. The study of statistically stable measurement functions have turned to be very useful in particular in the study of inverse source problems in random medium background \cite{bal2002,borcea2011,borcea2015,borcea2003,borcea2016,borcea2002,papa_book}.

Despite the non-linear terms generated by the 2nd order and multiple scattering, it is interesting to note the linear relation between the data and the local strength $\mu$ in \eqref{eq:intro_principal_symbol}. This suggests that, whenever \eqref{id:linear_relation} holds, the local strength of the Born approximation of $q$ equals the local strength of the full potential $q$.

Let us now formulate the main theorem of this paper.
\begin{theorem}
\label{thm:main_theorem_rigorous} \sl
Let $q$ be a Gaussian microlocally isotropic random field of order $-3$ in $D \subset \R^3$. Then, the measurement data $\{ M(\tau_j,\theta_j) : j\in \N \}$, with $\{ (\tau_j,\theta_j) : j\in\N \}$ any dense subset of $\R_+ \times \mathbb{S}^2 $, determines the local strength $\mu$ almost surely.
\end{theorem}

{Theorem \ref{thm:main_theorem_rigorous} can be interpreted as follows: We consider a complicated potential $q(x,\omega_0)$ that is assumed to be created, before the measurements are made, by a random process, that is, the potential $q(x,\omega_0)$ is a single realization of the Gaussian microlocally isotropic random
 $q(x,\omega)$. We show 
the  measurement data $\{ M(\tau_j,\theta_j) : j\in \N \}$, obtained from this single realization of the process, determine with probability one the principal symbol of the covariance operator of the random process $q(x,\omega)$.}

In order to prove Theorem \ref{thm:main_theorem_rigorous}, we explore the limit in \eqref{eq:intro_measurement_data} by separating the effects of different orders of scattering in the Born series $$u^\infty(k,\theta,-\theta) = \sum_{j=1}^\infty u_j^\infty(k,\theta,-\theta),$$ where $u_j^\infty$ describes the far field of the $j$-th order scattering. Under the assumptions of  Theorem \ref{thm:main_theorem_rigorous}, it can be rigorously shown that the interactions coming from the simple backscattering, i.e.,
\begin{equation*}
	M_1(\tau, \theta) = \lim_{K\to\infty} \frac{1}{K}\int_K^{2K} k^m u_1^\infty(k,\theta,-\theta) \overline{u_1^\infty(k+\tau,\theta,-\theta)} dk
\end{equation*}
coincides with $M(\tau, \theta)$ almost surely, whereas the contribution from 2nd order and multiple backscattering becomes negligible at the limit:
\[\lim_{K\to\infty} \frac{1}{K}\int_K^{2K} k^m u_j^\infty(k,\theta,-\theta) \overline{u_l^\infty(k+\tau,\theta,-\theta)} dk = 0 \quad \text{for } j+l \geq 3.\]
In particular, there exists a known constant $c>0$ such that
\begin{equation*}
M_1(\tau, \theta) = c\, \widehat{\mu} (2\tau\theta)\quad \text{almost surely}.
\end{equation*}

Our work also has direct implications to the problem in any dimension if the interaction of 2nd order and multiple backscattering can be a priori neglected.
\begin{theorem}
\label{thm:main_theorem_born_approx} \sl
Let $q$ be a Gaussian microlocally isotropic random field of order $-m$ in $D \subset \R^n$ with $m > n - 1$ and $n \geq 2$. Then, the data $\{ M_1(\tau_j,\theta_j) : j\in \N \}$, with $\{ (\tau_j,\theta_j) : j\in\N \}$ any dense subset of $\R_+ \times \mathbb{S}^{n - 1}$, determines the local strength $\mu$ almost surely.
\end{theorem}
This second theorem suggests that for general dimensions $n \geq 2$, and $m > n - 1$, the measurement data $\{ M(\tau_j,\theta_j) : j\in \N \}$, with $\{ (\tau_j,\theta_j) : j\in\N \}$ any dense subset of $\R_+ \times \mathbb{S}^{n - 1}$, determines the local strength of the Born approximation of $q$  almost surely.

In the literature on scattering, one often makes use of different physical scaling regimes in order to estimate the size of relevant mathematical objects and to design effective reconstruction methods. Our hope is that the theoretical framework we set up here can produce interesting stable algorithms for the inversion of the probabilistic backscattering problem.
In this spirit, we have included a brief analysis in the appendix \ref{sec:scaling_regimes}, where we consider scaling regimes such that the analogue of Theorem \ref{thm:main_theorem_rigorous} holds for $n=m$ with $n \geq 2$.

Our work follows the line of the previous papers \cite{LPS} by Lassas, P\"aiv\"arinta and Sakasman, and \cite{HLP} by Helin, Lassas and P\"aiv\"arinta.
In \cite{LPS}, a similar problem was solved in $\R^2$ for a backscattering problem with point sources in an open and bounded set, but assuming the knowledge of the full scattered wave. The present paper improves this setting by studying scattering of plane waves and assuming only knowledge of the far-field patter of the backscattered wave. Moreover, the results are generalized to arbitrary dimension. Although, our work draws inspiration of this paper, these two aspects require more sophisticated techniques in several parts. Later, in \cite{HLP} Helin \textit{et al} considered backscattering from random Robin boundary condition in half-space geometry of $\R^3$.

The literature on the deterministic inverse backscattering problem is considerably wide.
For uniqueness results in  generic (i.e., dense and open sets) class of potentials, see \cite{MR1012864, MR1110451}. 
Uniqueness of the problem for potentials with controlled angular regularity has been proved in  
\cite{MR3224125}.
Earlier partial results for the inverse backscattering problem  for Schr\"odinger equation has been obtained in \cite{Ike,LL,MU,Rakesh1,Ste1, MR1082237,W1}. Approximative or numerical reconstructions have been studied in \cite{Bei,Ike}.
The recovery of singularities of the potential from backscattering data is analyzed in
\cite{B1,MR1243710,OPS,MR2309667,R1, RV,Ser1,Ser3,Ser2}. 
Other references on inverse backscattering for a time-harmonic Schr\"{o}dinger equation are \cite{MR2512860, MR2781141,U1}.
The backscattering problem has also been studied in the framework of acoustic scattering (see \cite{MR1466676,MR1614940,MR1607660}) and Maxwell equations (see \cite{MR1648523}). For a concise treatment of classical inverse scattering, we refer to \cite{ColtonKress}.

The wave and particle propagation in heterogeneous media has been extensively studied. Often, heterogeneous media is not known precisely and is modelled as a realization of random media with known statistics. Mathematical theory being developed typically relies on multi-scale analysis or homogenization with the aim of capturing the effective properties of the propagation. We refer to the works in \cite{bal2010kinetic, ishimaru1978wave} for various perspectives on wave propagation (whether classical or quantum) in random media.
Let us also mention the papers on random Schr\"{o}dinger models \cite{bal2011asymptotics,GuRyzhik15}, where the potential model involving slowly decaying correlations corresponds closely to the random potential model in the present paper. Notice that our work does not involve assumptions on scaling regimes nor any approximations. However, as mentioned above, we have included { Appendix A} discussing our method from the perspective of multi-scale analysis.

Recently, inverse problems related to imaging of random media have received wide attention \cite{bal2002, bal2005time,bal2007,bal2003, borcea2011, borcea2015, borcea2003, borcea2016,dehoop2012, dehoop09, fouque2007wave}. The key feature of time reversal in a randomly inhomogeneous media 
is that it leads to focusing resolution that is much better than in a homogeneous media. This phenomenon is called super-resolution and appears due to multipathing caused by the random media \cite{borcea2002}. Similar to our work, the back-propagated fields are self-averaging and the imaging method is {statistically stable}, i.e., independent of the realization of the random media

This paper is organized as follows. In the section \ref{sec:random_potential}, we describe in detail our stochastic model for a random potential $q$ and the implications it has for the regularity of $q$. As discussed above, these regularity considerations require to develop the theory of the forward problem for non-regular potentials. This is studied in the section \ref{sec:direct-scattering}. The inverse problem is then covered in the section \ref{sec:inverse_problem}. The effects of first, second and higher order scattering for zero-mean potentials are studied separately in the sections \ref{subsec:single}, \ref{subsec:2nd order} and \ref{subsec:higher_order}, respectively. In the section \ref{sec:non-zero-mean} we proof Theorem \ref{thm:main_theorem_rigorous} for non-zero-mean potentials. Finally, in Appendix \ref{sec:scaling_regimes} we consider the physical scaling regimes where our method could be numerically effective and, afterwards, in Appendix \ref{sec:gaussians} give some basic results regarding Gaussian distributions.

\section{Random potential}
\label{sec:random_potential}

\subsection{Microlocally isotropic random field}
\label{sec:isotropic_rf}

In order to provide a precise mathematical description of the random potential to be considered, let $(\Omega, \mathcal{H}, \mathbb{P})$ be a complete probability space. Since we are interested in the properties of an object with a complicated micro-structure, we start by assuming that $q$ is a \emph{generalized random function}\footnote{For properties of generalized random functions, see \cite{MR676644}.}. 
{Below, the generalized function $u$ defines a linear and continuous function $u:C^\infty_0(\R^n; \R)\to \R$. For $\phi\in C^\infty_0(\R^n; \R)$  we denote $u(\phi)=\langle u, \phi \rangle$
and for the set of generalized functions (or distributions) we use the notation $\mathcal{D}'(\R^n; \R)$. We also recall that any
function $u\in L^1_{loc}(\R^n)$ defines a generalized function given by $\langle u, \phi \rangle=\int_{\R^n} u(x)\phi(x)\,dx$. 

The assumption that $q$  is a generalized random function} means that $q$ is a mapping defined on $\Omega$ such that, for every $\omega \in \Omega$, the realization $q (\omega)$ is a linear real valued functional on $C^\infty_0(\R^n; \R)$---the space of smooth real-valued functions with compact support in $\R^n$---with $n \geq 2$ and the function
\[\omega \in \Omega \longmapsto \langle q(\omega), \phi \rangle \in \R\]
is a random variable for all $\phi \in C^\infty_0(\R^n; \R)$.
Moreover we assume that, for every compact $K \subset \R^n$, there exists a non-negative random variable $C : (\Omega, \mathcal{H})\to \R_+$ with $\mathbb{E} C^2 < \infty$ and $N \in \N$ such that
for $\mathbb{P}$-a.e. $\omega\in\Omega$ we have
\begin{equation}
|\langle q(\omega), \phi \rangle| \leq C(\omega) \sum_{|\alpha| \leq N} \sup_{x \in \R^n} |\partial^\alpha \phi (x)| 
\label{in:def-distribution}
\end{equation}
for all $\phi \in C^\infty_0(\R^n; \R)$ with compact support $\supp\,\phi \subset K$.
Note that $q(\omega) \in \mathcal{D}'(\R^n; \R)$, which denotes the space of real distributions in $\R^n$. Obviously, $q(\omega)$ can be extended to the space of smooth (complex-valued) functions with compact support as
\[\langle q(\omega), \phi \rangle = \langle q(\omega), \mathrm{Re}\, \phi \rangle + i \langle q(\omega), \mathrm{Im}\, \phi \rangle, \qquad \phi \in C^\infty_0(\R^n).\]
A generalized random function is said to be \emph{Gaussian} if the random variable
\begin{equation}
\label{term:gaussianVECTOR}
r_1 \langle q, \phi_1 \rangle + \dots + r_l \langle q, \phi_l \rangle
\end{equation}
has a Gaussian distribution for every $r_1, \dots, r_l \in \R$, $\phi_1, \dots, \phi_l \in C^\infty_0(\R^n; \R)$ and $l \in \mathbb{N} \setminus \{ 0 \}$. We say it is compactly supported if there exists a bounded domain $D$ in $\R^n$ such that
$\supp q \subset D$
almost surely.
Note that the probability law of a generalized Gaussian field $q$ is determined by
\[\mathbb{E} q : \phi \in C^\infty_0(\R^n; \R) \longmapsto \mathbb{E}\langle q, \phi \rangle \in \R\]
\[\mathrm{Cov}\, q : (\phi_1, \phi_2) \in C^\infty_0(\R^n; \R)^2 \longmapsto \mathrm{Cov}(\langle q, \phi_1 \rangle, \langle q, \phi_2 \rangle) \in \R,\]
where $\mathbb{E}\langle q, \phi \rangle$ denotes the expected value of $\langle q, \phi \rangle$ and
\[\mathrm{Cov}(\langle q, \phi_1 \rangle, \langle q, \phi_2 \rangle) = \mathbb{E}\big((\langle q, \phi_1 \rangle - \mathbb{E}\langle q, \phi_1 \rangle) (\langle q, \phi_2 \rangle - \mathbb{E}\langle q, \phi_2 \rangle)\big)\]
denotes the covariance of $\langle q, \phi_1 \rangle$ and $\langle q, \phi_2 \rangle$. Note that $\mathbb{E}q \in \mathcal{D}'(\R^n; \R)$. The \emph{covariance operator} $C_q : \phi \in C^\infty_0(\R^n; \R) \longmapsto C_q \phi \in \mathcal{D}'(\R^n; \R) $ is defined as
\[\langle C_q \phi, \psi \rangle = \mathrm{Cov}(\langle q, \phi \rangle, \langle q, \psi \rangle). \]
Since $C_q$ is continuous, by the Schwartz kernel theorem, there exists a unique $K_q \in \mathcal{D}'(\R^n \times \R^n; \R)$, usually called the \emph{covariance function}, such that
\[ \langle K_q, \psi \otimes \phi \rangle = \langle C_q \phi, \psi \rangle = (\mathrm{Cov}\, q) (\phi, \psi) \]
for all $\phi, \psi \in C^\infty_0(\R^n; \R)$ (or more generally, for all $\phi, \psi \in C^\infty_0(\R^n)$). In particular,
\begin{equation}
K_q = \mathbb{E}\big((q - \mathbb{E}q) \otimes (q - \mathbb{E}q)\big).
\label{id:kernel_cov}
\end{equation}
It is often convenient to write
\[K_q(x, y) = \mathbb{E}\big((q(x) - \mathbb{E}q(x)) (q(y) - \mathbb{E}q(y))\big) \]
and
\[\langle K_q, \psi \otimes \phi \rangle = \int_{\R^n} \int_{\R^n} K_q(x,y) \psi(x) \phi(y) \,dx \, dy.\]
\begin{definition}\label{def:ml_iso} \sl
A generalized function $q$ on $\R^n$ is called \emph{microlocally isotropic of order $-m$ in $D$}, if the following conditions hold:
\begin{enumerate}
\item $\mathbb{E} q$ is smooth,
\item $q$ is supported in $D$ a.s.,
\item the covariance operator $C_q$ is a classical pseudo differential operator of order $-m$ with $n - 1 < m \leq n + 1$  and
\item $C_q$ has a principal symbol of the form $\mu(x) |\xi|^{-m}$ with $\mu \in C^\infty_0 (\R^n; \R)$, $\supp \mu \subset D$ and $\mu(x) \geq 0$ for all $x \in \R^n$.
\end{enumerate}
\end{definition}
In consequence, if $q$ is microlocally isotropic, we have $\supp (\mathbb{E} q) \subset D$ and
\[C_q \phi (x) = \frac{1}{(2\pi)^{n}} \int_{\R^n} \int_{\R^n} e^{i(x - y) \cdot \xi} c_q(x, \xi) \phi(y) \,dy \,d\xi \]
for a classical symbol $c_q \in {\mathcal S}^{-m}(\R^n\times \R^n)$. Moreover, there exists a classical symbol $a \in {\mathcal S}^{-m-1}(\R^n \times \R^n)$ such that
\begin{equation}
a(x, \xi) = c_q(x, \xi) - \mu(x) |\xi|^{-m} \text{ for } x \in \R^n \text{ and } |\xi| \geq 1.
\label{id:isotropic}
\end{equation}
The covariance function and the symbol of $C_q$ are connected via the following identity
\begin{equation}
K_q (x, y) = \frac{1}{(2 \pi)^{n/2}} \mathcal{F}^{-1}\big( c_q(x, \centerdot) \big)(x - y),
\label{id:kernel}
\end{equation}
where $\mathcal{F}^{-1}$ denotes the inverse Fourier transform. Here, for an integrable function $f$, we define $(\mathcal{F}f)(\xi) = (2\pi)^{-n/2} \int_{\R^n} e^{-i\xi\cdot x} f(x) \, dx$ and frequently abbreviate $\widehat{f} = (\mathcal{F}f)$.
In our particular case, $\supp K_q \subset D \times D$ and
\begin{align}
\int_{\R^n} \Big( \int_{\R^n} K_q(x,y) e^{-i\xi \cdot (x - y)} \,dy \Big)  \phi(x) \, dx &= \int_D c_q(x, \xi) \phi(x) \,dx
\label{id:pre-expansion} \\
& = \int_{\R^n} \mu(x) |\xi|^{-m} \phi(x) \,dx + \int_D a(x, \xi) \phi(x) \,dx
\label{id:expansion}
\end{align}
for every $\phi \in C^\infty(\R^n)$ and all $|\xi| \geq 1$.

Let us illustrate this definition with a brief example.
\begin{example} \sl
\label{example:model}
Let $W$ stand for the generalized Gaussian white noise. That is, $W$ is a generalized Gaussian field with $\expec W = 0$ and its covariance operator satisfying
\begin{equation*}
	\expec \big( \langle W, \phi\rangle \langle W, \psi\rangle \big) = \int_{\R^n} \phi(x) \psi(x) \, dx
\end{equation*}
for every $\phi,\psi\in C^\infty_0(\R^n; \R)$. It is well known that
$W \in H_{loc}^{-n/2-\epsilon}(\R^n)$ a.s. for any $\epsilon>0$.
Our example of a microlocally isotropic random potential of order $-m$ in $D$ is given by
\begin{equation*}
	q = \sqrt{\mu} ({\rm I}-\Delta)^{-m/4} W + q_0
\end{equation*}
with $\mu$ and $m$ as in Definition \ref{def:ml_iso}, and $q_0$ a smooth real-valued function in $\R^n$ with support in $D$. Thus, $q$ is a generalized Gaussian field with covariance operator
\begin{equation*}
	C_q = M_{\sqrt \mu} ({\rm I}-\Delta)^{-m/2} M_{\sqrt \mu}
\end{equation*}
where $M_{\sqrt \mu} \phi(x) = \sqrt{\mu(x)} \phi(x)$. Its covariance function is
\begin{equation*}
	K_q(x,y) =  \sqrt{\mu(x)} \sqrt{\mu(y)} G_m(x-y) 
\end{equation*}
with $G_m \in L^1(\R^n)$ (see page 132 in \cite{MR0290095}) such that
\[\widehat{G_m}(\xi) = \frac{1}{(2\pi)^{n/2}} \frac{1}{(1 + |\xi|^2)^{-m/2}}.\]
Finally, since $C_q$ has a principal symbol of the form $\mu(x)|\xi|^{-m}$, we see that $q$ is a microlocally isotropic random potential of order $-m$ in $D$.
\end{example}

\begin{example}\sl
\label{ex:frac_brownian}
Following the Example 1 in \cite{LPS} let us define the multidimensional fractional Brownian motion in $\R^n$ for the Hurst index $H$ as the centered Gaussian process $X_H(z)$ indexed by $z \in \R^n$ with following properties:
\begin{align*}
	& \expec | X_H(z_1) - X_H(z_2)|^2 = |z_1 - z_2|^{2H} \quad \textrm{for all } z_1, z_2 \in \R^n \\
	& X(z_0) = 0 \quad {\rm and} \\
	& \textrm{the paths } z \mapsto X_H(z) \textrm{ are a.s. continuous.}
\end{align*}
The existence and basic properties of $X_H$ are well-known \cite{Kahane}. Let us define the potential $q$ by setting
\begin{equation*}
	q(z,\omega) = \sqrt{\mu}(z) X_H(z,\omega)
\end{equation*}
for some $\mu\in C^\infty_0(\R^n)$ and index $H>0$. It follows that the principal symbol of $C_q$ is of the form $\mu(z) |\xi|^{-n-2H}$.
\end{example}

\begin{figure}[h]
\begin{picture}(400,100)(40,10)
\put(0,0){\includegraphics[width=0.33\textwidth]{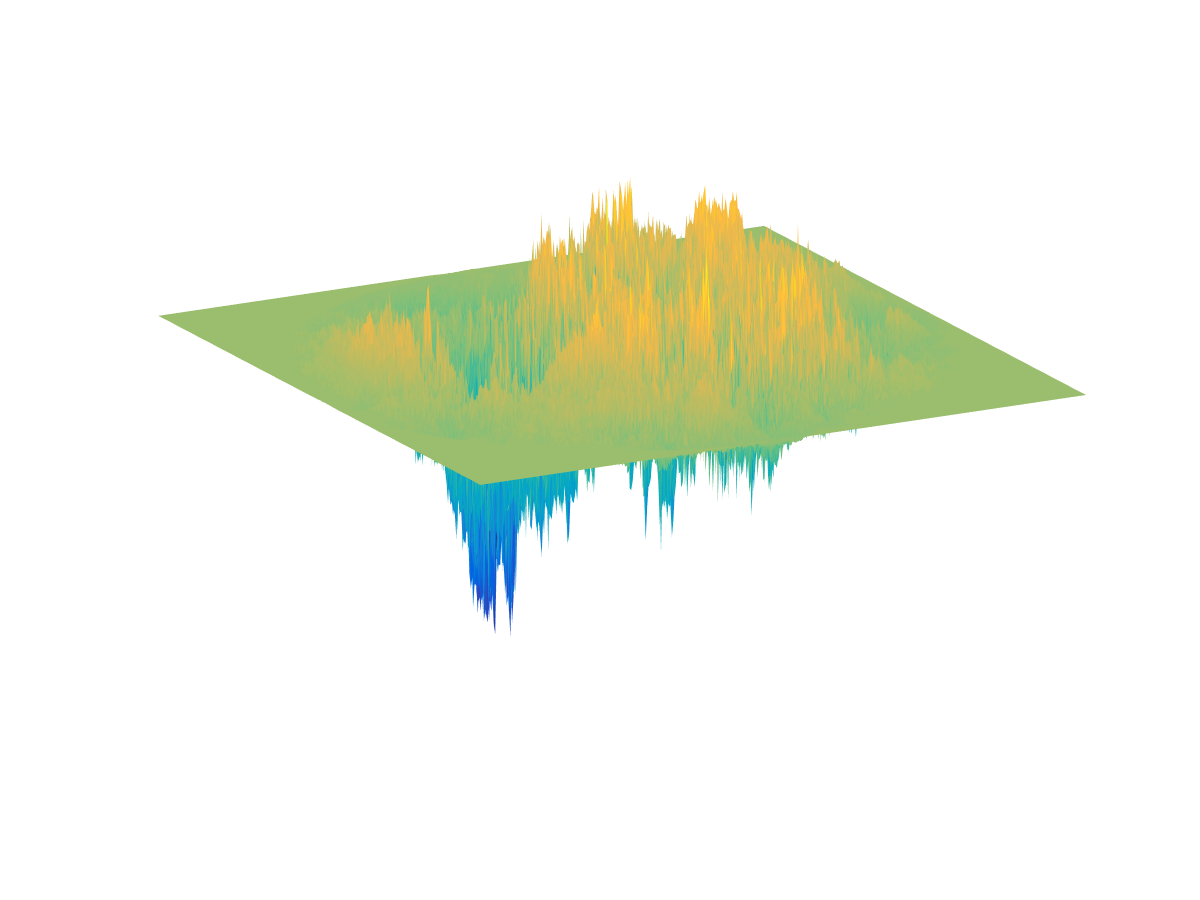}}
\put(160,0){\includegraphics[width=0.33\textwidth]{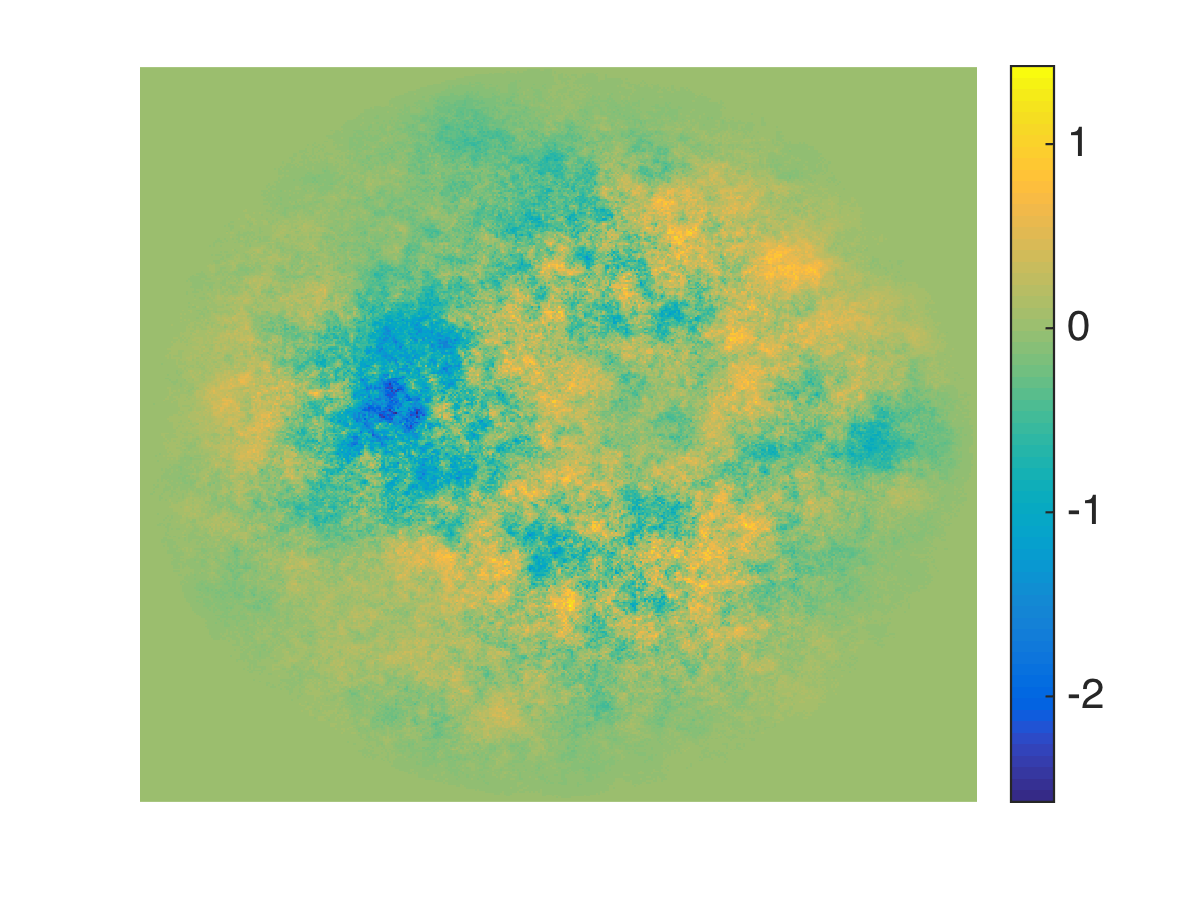}}
\put(330,0){\includegraphics[width=0.33\textwidth]{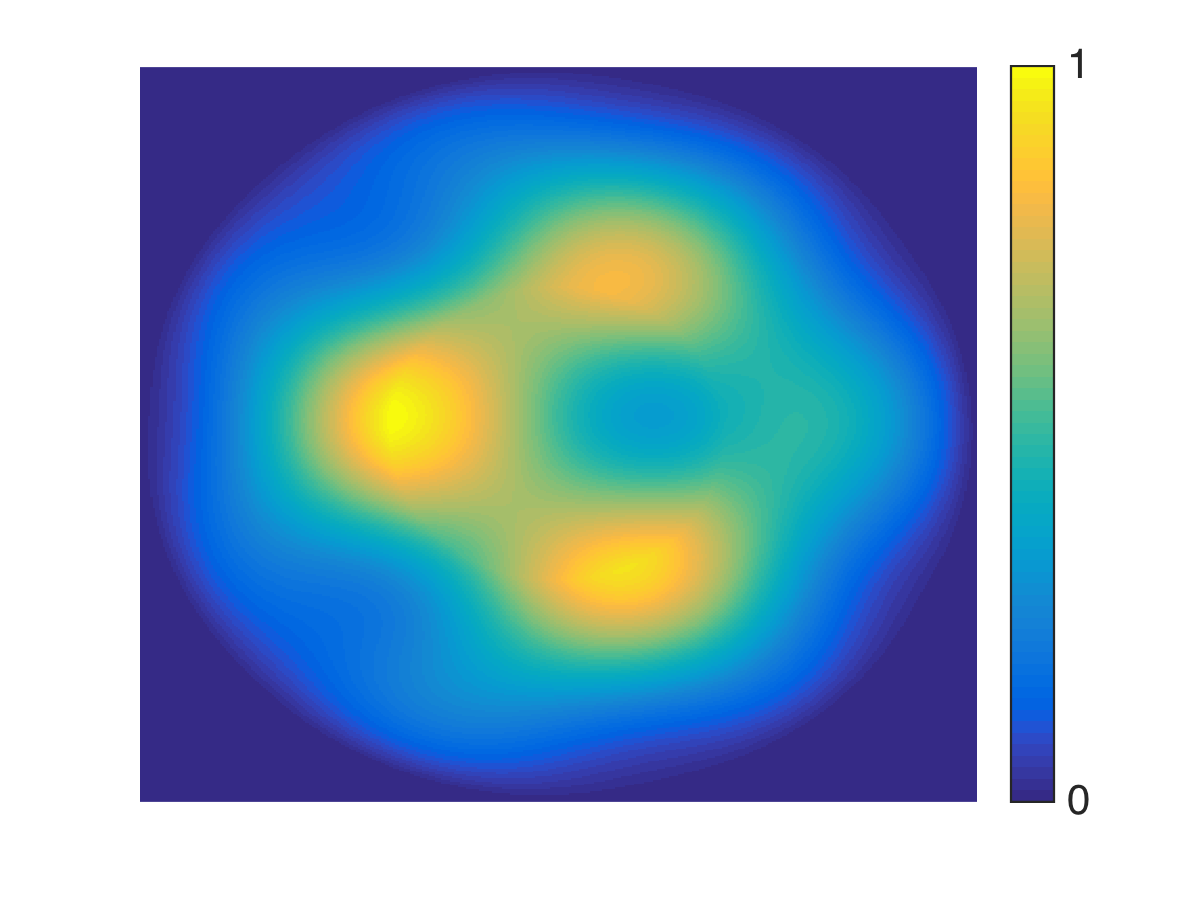}} 
\end{picture}
\caption{(Color in the online version of the paper) A realization of the random potential $q(x)$ in Example \ref{ex:frac_brownian} was generated on $\R^2$ with the Hurst index $H=0.25$ { and the local strength $\mu$. A realization $q(x) = q(x,\omega_0)$ is shown in the figure on the left and in the middle. The figure on the right is the local strength $\mu(x)$  of the random field $q$(x), i.e.,\ the principal symbol of the covariance operator $C_q$. Notice that the function $\mu$ is large on the areas where the realization of the random field $q(x,\omega_0)$ has large local fluctuations.}}
\label{fig_recon}
\end{figure}

\subsection{Regularity of the potential}\label{sec:regularity}

We will use the potential Sobolev spaces $L^p_s(\R^n)$ to determine the regularity of  the realizations of $q$. Let us recall the definition and a basic property of these spaces.

Let $\mathcal{J}_s$ denote the Bessel potential $\mathcal{J}_s = (\rm I - \Delta)^{- s/2}$, i.e., for any Schwartz function $f\in \mathcal{S}(\R^n)$ it holds that
\[\widehat{\mathcal{J}_s f} (\xi) = (1 + |\xi|^2)^{-s/2} \widehat{f}(\xi)\]
for all $\xi \in \R^n$. The Bessel potential extends to temperate distributions and, in particular, can be applied to almost every realization of $q$. The potential Sobolev space $L^p_s(\R^n)$ with $1 \leq p \leq \infty$ and $s \in \R$ is defined by the set of $ f = \mathcal{J}_s g$ such that $g \in L^p(\R^n)$ and it is endowed with the norm
\[\| f \|_{L^p_s} = \| g \|_{L^p}.\]
When $p = 2$, the space $L^2_s(\R^n)$ is commonly denoted by $H^s(\R^n)$. When $k \in \N$ and $1 < p < \infty$, the space $L^p_k(\R^n)$ can be identified with usual Sobolev spaces $W^{k,p}(\R^n)$.
{We also use the notation $W^{k,2}(\R^n)=H^{k}(\R^n)$.}
Note that, since $\mathcal{J}_t : L^p_s (\R^n) \longrightarrow L^p_{s+t} (\R^n) $ is an isometric isomorphism, $f \in L^p_s (\R^n) $ if and only if $\mathcal{J}_t f \in L^p_{s + t} (\R^n) $.

Now we can stablish the regularity of the realizations of $q$.
\begin{proposition} \label{prop:SobolevRegularity} \sl Let $q$ be a microlocally isotropic random potential of order $-m$ in $D$. Then, $q \in L^p_s (\R^n)$ almost surely for any $1 < p < \infty$ and $s < (m - n)/2$.
\end{proposition}

\begin{corollary}\sl Let $q$ be a microlocally isotropic random potential of order $-m$ in $D$ with $n < m \leq n + 1$. Then, $q \in C^{0,\alpha} (\R^n)$ almost surely for any $0 < \alpha < (m - n)/2$.
\end{corollary}

The corollary follows from the proposition using the Sobolev embeddings and the fact that, if $ m > n$ and $0 < \alpha < (m - n) / 2$, there exist $s < (m - n)/2 $ close to $(m - n)/2$ and $ p > n$ very large such that $\alpha = s - n/p < (m - n)/2$.

Let us prove Proposition \ref{prop:SobolevRegularity} in the case $\mathbb{E} q = 0$. The general case follows by our assumption of  $\mathbb{E} q$ being smooth. For $\varepsilon \in (0, 1]$, define
\begin{equation}
f_\varepsilon (\omega, x) = \langle \mathcal{J}_{-s} q(\omega), \varphi_\varepsilon (x - \centerdot) \rangle \label{id:def_of_fepsilon}
\end{equation}
with $\varphi_\varepsilon (x) = \varepsilon^{-n} \varphi (x/\varepsilon)$, $\varphi \in C^\infty_0(\R^n; [0, 1])$ and $\int_{\R^n} \varphi(x) \, dx = 1$. Note that $f_\varepsilon (\omega)$ tends to $\mathcal{J}_{-s} q(\omega)$ in $\mathcal{D}'(\R^n)$ as $\varepsilon$ goes to $0$ for $\mathbb{P}$-almost every $\omega \in \Omega$.

\begin{lemma} \sl For every $p \in [1, \infty)$ there exists $C = C(p)$ such that
\[(\mathbb{E} |f_\varepsilon (x)|^p )^{1/p} \le C (\mathbb{E} |f_\varepsilon (x)|^2 )^{1/2},\]
where $f_\varepsilon (x) = f_\varepsilon (\centerdot, x)$.
\end{lemma}
\begin{proof}
If $p \in [1, 2)$, the lemma follows applying H\"older inequality. The case $p = 2$ is obvious. So it only remains the case $p \in (2, \infty)$. For every $p \in (2, \infty)$, there is an only $j \in \N\setminus \{ 0 \}$ such that $2j < p \leq 2(j + 1)$, and, by H\"older inequality, we see that
\begin{equation}
(\mathbb{E} |f_\varepsilon (x)|^p )^{1/p} \le (\mathbb{E} |f_\varepsilon (x)|^{2(j+1)} )^\frac{1}{2(j+1)}.
\label{es:fromLp}
\end{equation}
By the triangle inequality, we see that
\[ \big(\mathbb{E} |f_\varepsilon (x)|^{2(j+1)} \big)^{1/(j+1)} \leq \Big(\mathbb{E} \big(\mathrm{Re}\,f_\varepsilon (x)\big)^{2(j+1)} \Big)^{1/(j+1)} + \Big(\mathbb{E} \big(\mathrm{Im}\,f_\varepsilon (x)\big)^{2(j+1)} \Big)^{1/(j+1)}. \]
Note that
\[\mathrm{Re}\,f_\varepsilon (\omega, x) = \langle q(\omega), \mathrm{Re}\,\mathcal{J}_{-s} \varphi_\varepsilon (x - \centerdot) \rangle, \qquad \mathrm{Im}\,f_\varepsilon (\omega, x) = \langle q(\omega), \mathrm{Im}\,\mathcal{J}_{-s} \varphi_\varepsilon (x - \centerdot) \rangle;\]
and
\begin{equation}
\mathbb{E}\mathrm{Re}\,f_\varepsilon (x) = \mathbb{E}\mathrm{Im}\,f_\varepsilon (x) = 0
\label{id:vanishing_mean}
\end{equation}
since $\mathbb{E} q = 0$. 
Using that $\mathrm{Re}\,f_\varepsilon (x)$ and $\mathrm{Im}\,f_\varepsilon (x)$ are Gaussian for every $x \in \R^n$, the identity \eqref{id:vanishing_mean} and Lemma \ref{lem:evenMOMENTUM}, we can check that
\[ \big(\mathbb{E} |f_\varepsilon (x)|^{2(j+1)} \big)^{1/(j+1)} \leq C \Big( \mathbb{E} \big(\mathrm{Re}\,f_\varepsilon (x)\big)^2 + \mathbb{E} \big(\mathrm{Im}\,f_\varepsilon (x)\big)^2 \Big) = C \mathbb{E} |f_\varepsilon (x)|^2 \]
for $C=C(j)$. Finally, plugging the previous inequality into \eqref{es:fromLp}, we get the estimate in the lemma.
\end{proof}

It is an immediate consequence of the previous lemma that
\begin{equation}
\mathbb{E} \| f_\varepsilon \|^p_{L^p(K)} = \int_K \mathbb{E} | f_\varepsilon (x) |^p \, dx \leq C \int_K \big( \mathbb{E} | f_\varepsilon (x) |^2 \big)^{p/2} \, dx \label{es:post_lemma}
\end{equation}
for every compact $K \subset \R^n$. On the other hand, we have the identity
\begin{align*}
\mathbb{E} | f_\varepsilon (x) |^2 &= \mathbb{E}\big(\langle q, \mathcal{J}_{-s}(\varphi_\varepsilon (x - \centerdot)) \rangle \langle q, \overline{\mathcal{J}_{-s}(\varphi_\varepsilon (x - \centerdot))} \rangle\big) \\
&= \langle C_q \mathcal{J}_{-s}(\varphi_\varepsilon (x - \centerdot)), \overline{\mathcal{J}_{-s}(\varphi_\varepsilon (x - \centerdot))} \rangle\\
&= \langle \mathcal{J}_{-s} C_q \mathcal{J}_{-s}(\varphi_\varepsilon (x - \centerdot)), \varphi_\varepsilon (x - \centerdot) \rangle.
\end{align*}
Note that $\mathcal{J}_{-s} C_q \mathcal{J}_{-s}$ is pseudo-differential operator of order $-m+2s$ with a symbol $\tilde{c} \in {\mathcal S}^{-m+2s} (\R^n \times \R^n)$ and Schwartz kernel
\[\tilde{K}(y, z) = \frac{1}{(2\pi)^{n/2}} \mathcal{F}^{-1} \big( \tilde{c}(y, \centerdot) \big)(y - z).\]
For $s < (m - n)/2 $, we can check that
\[\mathbb{E} | f_\varepsilon (x) |^2 = \frac{1}{(2\pi)^n} \int_{\R^n} \int_{\R^n} \Big(\int_{\R^n} e^{i\varepsilon (z - y) \cdot \xi} \tilde{c} (x - \varepsilon y, \xi) \,d\xi \Big) \varphi(y) \varphi(z)  \, dy \, dz\]
and consequently that
\begin{equation}
\mathbb{E} | f_\varepsilon (x) |^2 \leq C \int_{\R^n} (1 + |\xi|)^{-m+2s} \,d\xi.
\label{es:uniformBOUNDforEPS}
\end{equation}
Therefore, by \eqref{es:uniformBOUNDforEPS} and \eqref{es:post_lemma} we have that, for every compact $K$ in $\R^n$, there exists a constant $C$ independent of $\varepsilon$ such that
\begin{equation}
\mathbb{E} \| f_\varepsilon \|^p_{L^p(K)} \leq C,
\label{es:unformLpbound}
\end{equation}
for all $\varepsilon \in (0, 1]$.

\begin{lemma}\sl For $1\leq p < \infty $ and $s < (m - n)/2$ with $n - 1 < m \leq n + 1$, we have that
\begin{equation*}
	\mathcal{J}_{-s} q \in L^p_\mathrm{loc} (\R^n)
\end{equation*}
almost surely.
\end{lemma}
\begin{proof}
The case $p = 1$ follows from the case $1 < p < \infty$ by H\"older's inequality. The case $1 < p < \infty$ is a consequence of \eqref{es:unformLpbound}. Indeed, let $K$ be an arbitrary compact set in $\R^n$. The bound \eqref{es:unformLpbound} implies, by the Banach-Alaoglu theorem and since $L^p (\Omega \times K)$ with $1 < p < \infty$ is reflexive, that there exist $f \in L^p (\Omega \times K)$ and a vanishing sequence $\{ \varepsilon_j \}_{j=1}^\infty$, such that, $f_{\varepsilon_j}$ converges weakly to $f$ as $j$ goes to infinity. In particular, we have
\begin{equation}
\lim_{j\to\infty} \int_\Omega \int_{\R^n} f_{\varepsilon_j}(\omega, x) \phi(x) \psi(\omega) \,dx \,d\mathbb{P}(\omega) = \int_\Omega \int_{\R^n} f(\omega, x)  \phi(x) \psi(\omega) \,dx \,d\mathbb{P}(\omega)
\label{lim:weak}
\end{equation}
for every $\phi \in C^\infty_0(\R^n; \R)$ such that  $\supp (\phi) \subset K$ and every simple function $\psi : \Omega \to \R$.

Next, we show that
\begin{equation}
\lim_{\varepsilon \to 0} \int_\Omega \langle f_\varepsilon (\omega), \phi \rangle \psi(\omega) \,d\mathbb{P}(\omega) = \int_\Omega \langle \mathcal{J}_{-s} q(\omega), \phi \rangle \psi(\omega) \,d\mathbb{P}(\omega)
\label{lim:distributional}
\end{equation}
for every $\phi \in C^\infty_0(\R^n; \R)$ such that $\supp (\phi) \subset K$ and every simple function $\psi : \Omega \to \R$. The limit \eqref{lim:distributional} holds by the dominate convergence theorem since $f_\varepsilon (\omega)$ tends to $\mathcal{J}_{-s} q(\omega)$ in $\mathcal{D}'(\R^n)$ as $\varepsilon$ goes to $0$ almost surely and $|\langle f_\varepsilon (\omega), \phi \rangle |$ is bounded, for all $\varepsilon$, by an integrable function in $\Omega$. To check this last point, note that
\begin{equation*}
\langle f_\varepsilon (\omega), \phi \rangle = \Big\langle q (\omega), \chi \int_{\R^n} \varphi_\varepsilon(x) \mathcal{J}_{-s} \phi(x + \centerdot) \, dx \Big\rangle
\end{equation*}
with $\chi \in C^\infty_0(\R^n)$ such that $\chi(x) = 1$ for all $x \in \overline{D}$, and by \eqref{in:def-distribution},
\begin{equation*}
|\langle f_\varepsilon (\omega), \phi \rangle | \leq C(\omega) \sum_{|\alpha| \leq N} \sup_{y \in \R^n} \bigg| \partial^\alpha \big( \chi(y) \int_{\R^n} \varphi_\varepsilon(x) \mathcal{J}_{-s} \phi(x + y) \, dx \big) \bigg|,
\end{equation*}
where the second term on the right-hand side is bounded by a constant independent of $\varepsilon$ and $C=C(\omega)$ is integrable in $\Omega$.

Therefore, by the density of the simple functions and the limits \eqref{lim:weak} and \eqref{lim:distributional}, we have that
for $\mathbb{P}$-almost every $\omega\in\Omega$ we have
\begin{equation*}
\langle \mathcal{J}_{-s} q(\omega), \phi \rangle = \int_{\R^n} f(\omega, x)  \phi(x) \,dx
\end{equation*}
for any $\phi \in C^\infty_0(\R^n; \R)$ such that $\supp(\phi) \subset K$.
This concludes the proof of the lemma.
\end{proof}

In order to conclude the regularity of the realizations of $q$, choose $\chi \in C^\infty_0 (\R^n)$ such that $\chi (x) = 1$ for all $x \in \overline{D}$ and write
\[q(\omega) = \mathcal{J}_s (\chi \mathcal{J}_{-s} q(\omega)) + \mathcal{J}_s \big((1 - \chi) \mathcal{J}_{-s} q(\omega)\big). \]
The first term belongs to $L^p_s (\R^n)$ a.s. by the previous lemma while the second one belongs to $L^p_t (\R^n)$ a.s. for every $t \in \R$ and every $1 < p < \infty$ because the supports of $1 - \chi$ and $q(\omega)$ are disjoints. This ends the proof of Proposition \ref{prop:SobolevRegularity}.

\subsection{Covariance function}
We conclude this section providing a more detailed description of the covariance function. This will be proved in the next proposition and used at the end of the section \ref{subsec:2nd order}.

\begin{proposition}
\label{prop:correlation_function} \sl
Let $q$ be a microlocally isotropic random potential of order $-m $ in $D$. The covariance function $K_q$ has the following form:
\begin{itemize}
\item[(a)] If $n < m \leq n + 1$, there exists a compactly supported function $F_{1 + \alpha}\in C^{1,\alpha} (\R^n \times \R^n)$ for $ 0< \alpha < m - n$ such that
\[K_q (x, y) = c_{n,m}\, \mu(x) |x - y|^{m - n} + F_{1+\alpha}(x,y).\]
where $c_{n,m} $ is a constant depending on $n$ and $m$.

\item[(b)] If $n - 1 < m \leq n $, there exists a compactly supported function $F_\alpha \in C^{0,\alpha} (\R^n \times \R^n)$ for $0< \alpha < m - (n - 1)$ such that
\begin{equation*}
K_q (x, y) = 
\begin{cases}
c_{n,m}\, \mu(x) |x - y|^{-(n - m)} + F_\alpha(x,y),  & \text{if } n - 1 < m < n, \\
c_{n,m}\, \mu(x) \log |x - y| + F_\alpha(x,y),  & \text{if } m = n.
\end{cases}
\end{equation*}
where $c_{n,m} $ is a constant depending on $n$ and $m$.
\end{itemize}
\end{proposition}

\begin{proof}
As a consequence of identities \eqref{id:isotropic} and \eqref{id:kernel}, for a radially symmetric $\psi \in C^\infty_0(\R^n)$ such that $\psi(\xi) = 1$ for $|\xi| \leq 1$, we have
\begin{equation}
\begin{aligned}
(2 \pi)^{n/2} \langle &K_q (x, x - \centerdot), \phi \rangle = \int_{\R^n} \mu(x) (1 - \psi(\xi)) |\xi|^{-m}  \big(\mathcal{F}^{-1}\phi\big) (\xi) \, d\xi \\
& + \int_{\R^n} \mathcal{F}^{-1} \big( (1 - \psi) a(x, \centerdot) \big)(y)  \phi (y) \, dy + \int_{\R^n} \mathcal{F}^{-1} \big( \psi c_q(x, \centerdot) \big) (y)  \phi (y) \, dy.
\end{aligned}
\label{id:kernel,phi}
\end{equation}
Note that the function $v(x,y)=\mathcal{F}^{-1} \big( \psi c_q(x, \centerdot) \big) (y)$ satisfies $v\in \mathcal{S}(\R^n \times \R^n)$ with $\supp(v) \subset D \times \R^n$. On the other hand, $w(x,y) = \mathcal{F}^{-1} \big( (1 - \psi) a(x, \centerdot) \big) (y)$ is smooth and compactly supported in $D$ with respect to the variable $x$. Moreover, using the Hausdorff--Young inequality, we see that for fixed $x\in D$ we have $w(x,\centerdot)\in L^p_s(\R^n)$ for $2 \leq p \leq \infty$ and $s > 0$ such that $s - n/p < m - (n - 1)$. Furthermore, for every $0 < \alpha \leq 1$ and $k = 0, 1$ such that $k + \alpha < m - (n - 1)$ there exists $s > 0$ and $p > n$ with $k + \alpha = s - n/p$. Therefore, by the Sobolev embedding theorem, we have $w(x,\centerdot) \in C^{k, \alpha}(\R^n)$.

If $n < m \leq n + 1$, the first term on the right-hand side of \eqref{id:kernel,phi} coincides with
\[ \mu(x) \int_{\R^n} \mathcal{F}^{-1}\big( (1 - \psi) |\centerdot|^{-m} \big) (y)  \phi (y) \, dy \]
where $\mathcal{F}^{-1}\big( (1 - \psi) |\centerdot|^{-m} \big) \in L^p_s(\R^n)$ with $2 \leq p \leq \infty$ and $s > 0$ such that $s - n/p < m - n$. Again, by the Sobolev embedding theorem, we can ensure that $\mathcal{F}^{-1}\big( (1 - \psi) |\centerdot|^{-m} \big) \in C^{0, \alpha}(\R^n)$ for $ 0 < \alpha < m - n $. 

Let us now prove that
\[ \mathcal{F}^{-1}\big( (1 - \psi) |\centerdot|^{-m} \big)(y) = c_{n,m} \, |y|^{m - n} + g(y) \]
for some $g \in C^\infty(\R^n)$. The previous identity follows from the identity
\begin{equation}
\begin{aligned}
\mathcal{F}^{-1}\big( (1 - \psi) |\centerdot|^{-m} \big)(\lambda z) =& \lambda^{m - n} \mathcal{F}^{-1}\big( (1 - \psi) |\centerdot|^{-m} \big)(z)\\
& + \frac{1}{(2\pi)^{n/2}} \int_{\R^n} e^{i \lambda z \cdot \xi} |\xi|^{-m} (\psi(\lambda \xi) - \psi(\xi) ) \,d\xi
\end{aligned} \label{id:rrrr}
\end{equation}
with $z=y/|y|$, $\lambda = |y|$ and
\[c_{n,m} = \mathcal{F}^{-1}\big( (1 - \psi) |\centerdot|^{-m} \big)(y/|y|). \]
Recall that $\psi$ is radially symmetric and hence $c_{n,m}$ is independent of $y$.
The identity \eqref{id:rrrr} is, in turn, an immediate consequence of
\[(1 - \psi(\xi)) |\xi|^{-m} = (1 - \psi(\lambda \xi)) |\xi|^{-m} + (\psi(\lambda\xi)- \psi(\xi)) |\xi|^{-m}.\]
Thus, in the case $n < m \leq n + 1$ the identity \eqref{id:kernel,phi} becomes
\[(2 \pi)^{n/2} \langle K_q (x, x - y), \phi(y) \rangle = \int_{\R^n} (c_{n,m}\, \mu(x) |y|^{m - n} + g(x,y)) \phi(y) \, dy\]
for some $g \in C^{1,\alpha}(\R^n \times \R^n)$ with compact support with respect to the variable $x$. Since $K_q$ is supported on $D \times D$, we can just introduce an appropriate compactly supported smooth function in the previous identity and write the identity
\[K_q (x, y) = c_{n,m}\, \mu(x) |x - y|^{m - n} + F_{1+\alpha}(x,y)\]
with $c_{n,m} $ and $F_{1+\alpha}$ as in the (a). This concludes the proof of statement (a).

Consider the case (b) with $n - 1 < m < n $. Let us first record two useful identities. First, we can write
\begin{equation}
|\xi|^{-m} = \frac{2^{-m/2}}{\Gamma(m/2)} \int_0^\infty t^{m/2} e^{-t |\xi|^2/ 2} \, \frac{dt}{t},
\label{id:gamma}
\end{equation}
where $\Gamma$ stands for
\[\Gamma (\beta) = \int_0^\infty t^\beta e^{-t} \, \frac{dt}{t}. \]
Second, the basic properties of the Fourier transform yield the identity
\begin{equation}
\int_{\R^n} e^{-t |\xi|^2/2} (1 - \psi(\xi)) \mathcal{F}^{-1}\phi (\xi) \, d\xi = t^{-n/2} \int_{\R^n} e^{- |y|^2/(2t)} \mathcal{F}\big((1 - \psi) \mathcal{F}^{-1}\phi \big)(y) \, dy.
\label{id:gaussianfourier}
\end{equation}
Using \eqref{id:gamma} and \eqref{id:gaussianfourier} on the first term on the right-hand side of \eqref{id:kernel,phi} we have
\begin{equation}
\begin{aligned}
\int_{\R^n} \mu(x) (1 &- \psi(\xi)) |\xi|^{-m}  \mathcal{F}^{-1}\phi (\xi) \, d\xi \\
& = \frac{2^{-m/2}}{\Gamma(m/2)}  \mu(x) \int_0^\infty t^{m/2 - n/2} \int_{\R^n} e^{- |y|^2/(2t)} \mathcal{F}\big((1 - \psi) \mathcal{F}^{-1}\phi \big)(y) \, dy \, \frac{dt}{t} \\
& = \frac{2^{-m + n/2}}{\Gamma(m/2)} \Gamma(n/2 - m/2)  \mu(x) \int_{\R^n} |y|^{m - n} \mathcal{F}\big((1 - \psi) \mathcal{F}^{-1}\phi \big)(y) \, dy.
\end{aligned}
\label{id:related2HLS}
\end{equation}
In the last identity we have used the definition of $\Gamma$. Since
\[\int_{\R^n} |y|^{m - n} \mathcal{F}\big(\psi \mathcal{F}^{-1}\phi \big)(y) \, dy = \int_{\R^n} \left( \frac{1}{(2 \pi)^{n/2}} \int_{\R^n} |y|^{m - n} \mathcal{F}^{-1}\psi (z - y) \, dy \right) \phi(z) \, dz,\]
we can deduce from \eqref{id:related2HLS} and \eqref{id:kernel,phi}, proceeding as before, that
\[K_q (x, y) = c_{n,m}\, \mu(x) |x - y|^{m - n} + F_\alpha(x,y)\]
with $c_{n,m} $ and $F_\alpha $ as in the (b).

The case $m=n$ follows from a limit argument, making $m < n$ goes to $n$. To do so, note that
\[\Gamma(n/2 - m/2) = \frac{\Gamma(n/2 - m/2 + 1)}{n/2 - m/2},\]
which is a consequence of the definition of $\Gamma$, and rewrite \eqref{id:related2HLS} as
\begin{align*}
\int_{\R^n} \mu(x) (1 &- \psi(\xi)) |\xi|^{-m}  \mathcal{F}^{-1}\phi (\xi) \, d\xi \\
& = \frac{2^{-m + n/2}}{\Gamma(m/2)} \Gamma(n/2 - m/2 + 1)  \mu(x) \int_{\R^n} \frac{|y|^{m - n} - 1}{n/2 -m/2} \mathcal{F}\big((1 - \psi) \mathcal{F}^{-1}\phi \big)(y) \, dy \\
& \quad + \frac{2^{-m + n/2}}{\Gamma(m/2)} \Gamma(n/2 - m/2)  \mu(x) \int_{\R^n} \mathcal{F}\big((1 - \psi) \mathcal{F}^{-1}\phi \big)(y) \, dy.
\end{align*}

On the other hand, since $(|y|^{-(n-m)} - 1)/(n-m) \longrightarrow \log |y|^{-1}$ as $m \rightarrow n$, we can use the dominate convergence theorem to pass to the limit and obtain
\begin{align*}
\int_{\R^n} \mu(x) (1 - \psi(\xi)) |\xi|^{-n}  \mathcal{F}^{-1}\phi (\xi) \, d\xi =& c_{n,m}\, \mu(x) \int_{\R^n} \log |y|^{-1} \mathcal{F}\big((1 - \psi) \mathcal{F}^{-1}\phi \big)(y) \, dy \\
& + \tilde{c}_{n,m}\, \mu(x) \int_{\R^n} \mathcal{F}\big((1 - \psi) \mathcal{F}^{-1}\phi \big)(y) \, dy
\end{align*}
for $c_{n,m}$ and $\tilde{c}_{n,m}$ positive constants. Checking that the identities
\begin{align*}
\int_{\R^n} \log |y| \mathcal{F}\big(\psi \mathcal{F}^{-1}\phi \big)(y) \, dy &= \int_{\R^n} \left( \frac{1}{(2 \pi)^{n/2}} \int_{\R^n} \log |y| \mathcal{F}^{-1}\psi (z - y) \, dy \right) \phi(z) \, dz, \\
\int_{\R^n} \mathcal{F} \big((1 - \psi) \mathcal{F}^{-1}\phi \big)(y) \, dy &= \int_{\R^n} \left( 1 - \frac{1}{(2 \pi)^{n/2}} \int_{\R^n} \mathcal{F}^{-1}\psi (z - y) \, dy \right) \phi(z) \, dz
\end{align*}
hold and proceeding as before we have
\[K_q (x, y) = c_{n,m}\, \mu(x) \log |x - y| + F_\alpha(x,y)\]
with $c_{n,m} $ and $F_\alpha $ as in the (c), which ends the proof of this proposition.
\end{proof}

\section{Direct scattering for a rough potential}\label{sec:direct-scattering}

In the previous section we established that the realizations of our random potential field model can be rough (Proposition \ref{prop:SobolevRegularity}). Therefore, we need to show that the scattering problem in \eqref{eq:intro_direct}
is well-defined. To achieve this, we leave the randomness aside for a while and consider the deterministic scattering problem
\begin{equation}
\label{eq:direct_problem}
\left\{
\begin{aligned}
& (\Delta + k^2 - V) u = 0\\
&u (x) = e^{ik\theta \cdot x} + u_{\rm sc} (x)\\
& \frac{x}{|x|} \cdot \nabla u_{\rm sc}(x) - ik u_{\rm sc}(x) = o(|x|^{-\frac{n - 1}{2}}) \; \textrm{as} \, |x|\rightarrow \infty
\end{aligned}
\right.
\end{equation}
for a rough potential $V$. Later, our plan is to apply these results to the random scattering scenario for pointwise values $q=q(\omega)$.
The condition satisfied by $u_{\rm sc}$ at infinity is usually referred as the \textit{outgoing Sommerfeld radiation condition} (SRC for short). Here we assume $V$ to be in $L^p_{-s}(\R^n)$ with $0 < s \leq 1/2$ and $n/s \leq p < \infty$ and to have support in a bounded domain $D$ in $\R^n$. Note that the limiting case $V \in L^\infty(\R^n)$ (corresponding to $s\to 0$) is classical.

We will construct the scattered wave to be $u_{\rm sc} = \sum_{j = 1}^\infty u_j$ with
\[
\left\{
\begin{aligned}
& (\Delta + k^2) u_j = V u_{j - 1}\quad j \in \N \setminus \{ 0 \},\\
& u_0 (x) = e^{ik\theta \cdot x}.
\end{aligned}
\right.
\]
Since the function $\xi \mapsto (-|\xi|^2 + k^2)^{-1}$ does not define a temperate distribution, we will introduce a temperate distribution $\Psi^\pm_k$ as the limit of
\[ \xi \longmapsto \frac{1}{-|\xi|^2 + k^2 \pm i \varepsilon} \]
as $\varepsilon > 0$ vanishes. This limit exists and it is given by
\[\langle \Psi^\pm_k, \phi \rangle = \mathrm{p.v.} \int_{\R^n} \frac{\phi(\xi)}{-|\xi|^2 + k^2} \, d\xi \mp i \frac{\pi}{2k} \int_{ |\xi| = k } \phi(\xi) \, d\sigma_k(\xi) \]
where the principal value $\mathrm{p.v.} \int_{\R^n}$ stands for $\lim_{\epsilon \rightarrow 0} \int_{0<\epsilon\leq|k^2 - |\xi|^2|}$ and $d\sigma_k$ denotes the volume form on $\{ \xi\in \R^n : |\xi| = k \}$. Now, we let $\Phi^\pm_k$  denote the inverse Fourier transform of $\Psi^\pm_k$ and $\mathcal{R}^\pm_k$ be defined by
\begin{equation}
\widehat{\mathcal{R}^\pm_k f} = \Psi^\pm_k \widehat{f}
\label{id:multiplier}
\end{equation}
for $f$ in the Schwartz class, hence,
\[\mathcal{R}^\pm_k f = \frac{1}{(2\pi)^{n/2}} \Phi^\pm_k \ast f.\]
Thus, we set $u_j = \mathcal{R}^\pm_k(V u_{j - 1}) = (\mathcal{R}^\pm_k \circ V)^j u_0$, where $\mathcal{R}^\pm_k \circ V$ denotes the \textit{resolvent operator} $\mathcal{R}^\pm_k$ composed with the operator \textit{multiplication-by} $V$.

On the other hand, it is known that $\Phi_k^\pm$ satisfies
\begin{equation}
\frac{x}{|x|} \cdot \nabla \Phi_k^\pm(x) \mp ik \Phi_k^\pm(x) = o(|x|^{-\frac{n - 1}{2}}) \; \textrm{as} \, |x|\rightarrow \infty.
\label{es:SRCoutin}
\end{equation}
Therefore, the scattered wave, which has to satisfy the SRC, will be constructed with $\mathcal{R}^+_k$ as
\begin{equation}
u_{\rm sc} = \sum_{j = 1}^\infty (\mathcal{R}^+_k \circ V)^j u_0
\label{id:neumann_series}
\end{equation}
provided this infinite sum makes sense. 

In the following, we prove that the construction \eqref{id:neumann_series} is well-defined,
by proving boundedness properties for the operators $\mathcal{R}^+_k $ and {multiplication-by} $V$ in the weighted spaces $H^{s,\delta}_k(\R^n)$. The space $H^{s,\delta}_k(\R^n)$ is constructed in the following manner. Let $L^{2, \delta} (\R^n)$, with $\delta \in \R$, denote the equivalence class of measurable functions $f$ in $\R^n$ such that
\[\int_{\R^n} (1 + |x|^2)^\delta |f(x)|^2 \, dx < \infty,\]
and be endowed with the norm
\[ \| f \|_{L^{2, \delta}} = \Big( \int_{\R^n} (1 + |x|^2)^\delta |f(x)|^2 \, dx \Big)^{1/2}. \]
If $\delta = 0$, the space $L^{2,0}(\R^n)$ is just $L^2(\R^n)$.
Let $H^{s,\delta}_k(\R^n)$, with $\delta $ and $s $ in $ \R$, be the space of $f = (k^2 - \Delta)^{-s/2} g$ such that $g \in L^{2,\delta} (\R^n)$, where $(k^2 - \Delta)^{-s/2}$ is defined as the multiplier with symbol $ (k^2 + |\xi|^2)^{-s/2}$. Let this space be endowed with the norm
\[\| f \|_{H^{s, \delta}_k} = \| g \|_{L^{2,\delta}}.\]
When $\delta = 0$, the space $H^{s,0}_k(\R^n)$ will be denoted by $H^s_k(\R^n)$.

The boundedness properties of the operators $\mathcal{R}^\pm_k $ and {multiplication-by} $V$ needed for our purposes, will be studied in the sections \ref{sec:resolvent} and \ref{sec:multiplication} respectively, and stated here as follows:
\begin{theorem}\label{thm:relosvent} \sl The operator $\mathcal{R}^\pm_k $ is bounded from $H^{-s, \delta}_k (\R^n)$ to $H^{s, -\delta}_k (\R^n)$ with $0 \leq s \leq 1/2$ and $\delta > 1/2$ and satisfies the inequality
\[\| \mathcal{R}^\pm_k f \|_{H^{s, -\delta}_k} \lesssim k^{-(1 - 2s)} \| f \|_{H^{-s, \delta}_k}. \]
\end{theorem}

\begin{proposition} \label{prop:multiplication} \sl If $V $ belong either $ L^p_{-s}(\R^n)$ with $0 < s \leq 1/2$ and $ n/s \leq p < \infty$, then the operator {multiplication-by} $V$ is bounded from $H^{s, -\delta}_k (\R^n)$ to $H^{-s, \delta}_k (\R^n)$ and satisfies
\[\| V f \|_{H^{-s, \delta}_k} = o\Big( \| f \|_{H^{s, -\delta}_k}\Big), \]
expressed with the \emph{little o} of Landau.
\end{proposition}

As a consequence of these statements, there exists $k_0 > 0$, depending only on $\| V \|_{L^p_{-s}}$, $n$, $s$ and $\delta$, such that the operator $\mathcal{R}^\pm_k \circ V$ maps $H^{s, -\delta}_k(\R^n)$ into itself with a norm strictly less than $1$ for $k \geq k_0$. Then, the sum \eqref{id:neumann_series} converges in $H^{s, -\delta}_k(\R^n)$, and $u_{\rm sc}$ can be constructed by the infinite sum.

\begin{remark}
\label{re:k0}
Notice carefully that in the probabilistic problem setting the threshold wavelength $k_0 = k_0(\omega)$ becomes random due to the dependence on the term $\norm{q(\omega)}_{L^p_{-s}}$. However, we clearly have $k_0(\omega)<\infty$ almost surely and as our reconstruction method is based on a single realization of the potential, the randomness poses no problems.
\end{remark}

Finally, by construction, the scattered wave satisfies the Lippmann--Schwinger equation
\begin{equation*}
	u_{\rm sc}(x) = \int_{\R^n} \Phi_k^+(x-y) V(y) (e^{ik\theta\cdot y} + u_{\rm sc}(y)) dy \quad \text{in} \;
	H^{s,-\delta}_k(\R^n).
\end{equation*}
By the asymptotic behaviour of $\Phi_k^+$ and the fact that $u_{\rm sc}$ solves $(\Delta + k^2) u_{\rm sc} = 0$ in the exterior of a ball containing $D$, we have  that $u_{\rm sc} (x)$ is asymptotically equivalent, as $|x|$ grows, to 
\[k^{(n - 1)/2} \frac{e^{ik|x|}}{|x|^{(n - 1)/2}} u^\infty (k, \theta, x/|x|), \]
where $u^\infty (k, \theta, x/|x|)$ is the \textit{far-field pattern} and can be expressed as
\[u^\infty (k, \theta, x/|x|) = c_n\int_{\R^n} e^{-ik\frac{x}{|x|}\cdot y} V(y) (e^{ik\theta \cdot y} + u_{\rm sc}(y)) \, dy\]
with $c_n$ a constant only depending on $n$. Furthermore, we can conclude that the $u_{\rm sc}$ satisfies the SRC.

\subsection{Resolvent estimates} \label{sec:resolvent}
Start by noting that whenever $f \in \mathcal{S}(\R^n)$, the identity
\[\langle \Psi^\pm_k \widehat{f}, \phi \rangle = \mathrm{p.v.} \int_{\R^n} \frac{\widehat{f}(\xi)\phi(\xi)}{-|\xi|^2 + k^2} \, d\xi \mp i \frac{\pi}{2k} \int_{ |\xi| = k } \widehat{f}(\xi) \phi(\xi) \, d\sigma_k(\xi) \]
holds for every bounded smooth function $\phi$. Thus, according to \eqref{id:multiplier}, we have that
\[\mathcal{R}^\pm_k f (x) = \frac{1}{(2\pi)^{n/2}} \Big( \mathrm{p.v.} \int_{\R^n} \frac{e^{i x\cdot \xi}\widehat{f}(\xi)}{-|\xi|^2 + k^2} \, d\xi \mp i \frac{\pi}{2k} \int_{ |\xi| = k } e^{i x\cdot \xi} \widehat{f}(\xi) \, d\sigma_k(\xi) \Big) \]
for every $f \in \mathcal{S}(\R^n)$. For convenience, let us write
\[\mathcal{P}_k f(x) = \mathrm{p.v.} \int_{\R^n} \frac{e^{i x\cdot \xi}\widehat{f}(\xi)}{-|\xi|^2 + k^2} \, d\xi, \qquad \mathcal{Q}_k f(x) = \frac{1}{k} \int_{ |\xi| = k } e^{i x\cdot \xi} \widehat{f}(\xi) \, d\sigma_k(\xi). \]

The main goal of this section is to prove the inequalities
\begin{align}
\| \mathcal{P}_k f \|_{H^{1, -\delta}_k} &\lesssim \| f \|_{L^{2, \delta}} \label{es:re_part}, \\
\| \mathcal{Q}_k f \|_{H^{1, -\delta}_k} &\lesssim \| f \|_{L^{2, \delta}} \label{es:im_part}
\end{align}
for every $f \in \mathcal{S}(\R^n)$, since Theorem \ref{thm:relosvent} is a consequence of them. Indeed, these inequalities imply that
\begin{align*}
\| \mathcal{R}^\pm_k f \|_{H^{1 - s, -\delta}_k} &= \| \mathcal{R}^\pm_k (k^2 - \Delta)^{-s/2} f \|_{H^{1, -\delta}_k} \\
& \lesssim \| (k^2 - \Delta)^{-s/2} f \|_{L^{2, \delta}} = \| f \|_{H^{-s, \delta}_k}.
\end{align*}
On the other hand, checking that
\begin{equation}
\| \mathcal{R}^\pm_k f \|_{H^{s, -\delta}_k} \lesssim k^{-(1 - 2s)} \| \mathcal{R}^\pm_k f \|_{H^{1 - s, -\delta}_k},
\label{es:gainINk}
\end{equation}
we conclude the estimate
\[\| \mathcal{R}^\pm_k f \|_{H^{s, -\delta}_k} \lesssim k^{-(1 - 2s)} \| f \|_{H^{-s, \delta}_k}, \]
and consequently the theorem. Note that \eqref{es:gainINk} is a simple consequence of the following proposition.

\begin{proposition}\label{prop:st} \sl Let $s,t, \delta$ and $k$ be real numbers such that $s \leq t$. Then,
\[\| f \|_{H^{s,\delta}_k} \lesssim k^{-(t-s)} \| f \|_{H^{t,\delta}_k} \]
for all $f \in \mathcal{S}(\R^n)$.
\end{proposition}
\begin{proof}
Firstly note that it is enough to prove that for $s, \delta, k \in \R$ with $s \geq 0$ the estimate
\begin{equation}
\| f \|_{H^{-s,\delta}_k} \lesssim k^{-s} \| f \|_{L^{2,\delta}}
\label{es:-s0}
\end{equation}
holds for all $f \in \mathcal{S}(\R^n)$.

In order to prove \eqref{es:-s0}, it will convenient to introduce some notation. Set
\begin{align*}
D_0 &= \{ x \in \R^n : |x|\leq 1 \},\\
D_j &= \{ x \in \R^n : 2^{j - 1} < |x| \leq 2^j \} \qquad j \in \N \setminus \{ 0 \}
\end{align*}
and consider $\chi_0$ a smooth function with values in $[0, 1]$ such that $\supp \chi_0 \subset D_0$ and $\chi_0(x) = 1$ for all $|x| \leq 1/2$. Let $\chi_j$ with $j \in \N \setminus \{ 0 \}$ denote $\chi_j (x) = \chi_0(x/2^{j+1}) - \chi_0(x/2^{j})$ and note that $\supp \chi_j \subset \{  2^{j - 1} \leq |x| \leq 2^{j + 1}\}$ for $j \geq 1$. By construction, $\sum_{j \in \N} \chi_j(x) = 1$ for every $x \in \R^n$.

Using this notation, we can write
\[\| f \|_{H^{-s,\delta}_k}^2 \lesssim \sum_{j \in \N} \Big( \sum_{l \in \N} 2^{j\delta} \|(k^2 - \Delta)^{-s/2}(\chi_l f)\|_{L^2(D_j)} \Big)^2.\]
The sums on the right hand side will be studied separately according to \[\sum_{j \in \N} \Big( \sum_{l \in \N} \cdots \Big)^2 \sim \sum_{j \in \N} \sum_{|l - j|\leq 2} \big(  \cdots \big)^2 + \sum_{j \in \N} \Big( \sum_{|l - j|> 2}  \cdots \Big)^2.\]
Start by the case where, the support of $\chi_l$ and $D_j$ intersect each other:
\begin{equation*}
\sum_{j \in \N} \sum_{|l - j|\leq 2} 2^{2j\delta} \|(k^2 - \Delta)^{-s/2}(\chi_l f)\|^2_{L^2(D_j)} \lesssim \sum_{l \in \N} 2^{2l\delta} \|(k^2 - \Delta)^{-s/2}(\chi_l f)\|^2_{L^2}
\end{equation*}
Using now Plancherel's identity, the right hand side of the previous inequality can be bounded by
\[ k^{-2s} \sum_{l \in \N} 2^{2l\delta} \int_{\R^n} |\chi_l(x) f(x)|^2 \, dx \lesssim k^{-2s} \| f \|_{L^{2,\delta}}^2.\]
Consider now the case where the support of $\chi_l$ and $D_j$ and far from each other. The norm $\| (k^2 - \Delta)^{-s/2}(\chi_l f) \|_{L^2(D_j)}$ will be estimated by duality:
\begin{equation}
\frac{1}{(2 \pi)^n} \int_{\R^n} \int_{\R^n} \int_{\R^n} e^{i(x - y) \cdot \xi} (k^2 + |\xi|^2)^{-s/2}f_l (y) \, dy \, d\xi \, g(x) \, dx
\label{term:L2duality}
\end{equation}
where $g$ is any smooth function compactly supported in $D_j$ and $f_l$ denotes for simplicity $ \chi_l f$. Let $\phi$ be a real-valued smooth function defined on $[0, \infty)$ with compact $\{ 0 \leq t \leq 1/4 \}$ and such that $\phi (t) = 1$ for $ 0 \leq t \leq 1/8$. Then, \eqref{term:L2duality} is equal to
\[\frac{1}{(2 \pi)^n} \int_{\R^n} \int_{\R^n} \int_{\R^n} e^{i(x - y) \cdot \xi} \Big[ 1 - \phi\Big( \frac{|x - y|}{|2^j - 2^l|} \Big)\Big] (k^2 + |\xi|^2)^{-s/2}f_l (y) \, dy \, d\xi \, g(x) \, dx.\]
This holds because if $x \in D_j$ and $y \in \supp \chi_l$ with $|j - l|>2$, then $|x - y| \geq 2^{\max(j,l)} / 4 > |2^j - 2^l| / 4$. Furthermore, using the identity
\[- \frac{\Delta_\xi e^{i(x - y)\cdot \xi}}{|x - y|^2} = e^{i(x - y)\cdot \xi}, \qquad x \neq y\]
($N$ times) and integrating by parts ($2N$ times) in $\xi$, we see that \eqref{term:L2duality} equals
\[-\frac{1}{(2 \pi)^n} \int_{\R^n} \int_{\R^n} \int_{\R^n} e^{i(x - y) \cdot \xi} \frac{1}{|x - y|^{2N}} \Big[ 1 - \phi\Big( \frac{|x - y|}{|2^j - 2^l|} \Big)\Big] \Delta_\xi^N (k^2 + |\xi|^2)^{-s/2}f_l (y) \, dy \, d\xi \, g(x) \, dx.\]
For $N > n / 2$, we have that
\[\Big| \int_{\R^n} e^{i(x - y) \cdot \xi} \Delta_\xi^N (k^2 + |\xi|^2)^{-s/2} \, d\xi \Big| \lesssim k^{-s}, \]
so we can apply Fubini to integrate first in $\xi$ and then apply the Cauchy--Schwarz. Consequently, we have
\begin{equation*}
\Big| \int_{\R^n} \int_{\R^n} \int_{\R^n} e^{i(x - y) \cdot \xi} (k^2 + |\xi|^2)^{-s/2}f_l (y) \, dy \, d\xi \, g(x) \, dx \Big| \lesssim k^{-s} \| \Phi \ast f_l\|_{L^2} \| g \|_{L^2}
\end{equation*}
with
\[\Phi (x) = \frac{1}{|x|^{2N}} \Big[ 1 - \phi\Big( \frac{|x|}{|2^j - 2^l|} \Big)\Big]. \]
This implies, by duality and Young's inequality, that
\begin{equation}
\sum_{j \in \N} \Big( \sum_{|l - j|> 2} 2^{j\delta} \| (k^2 - \Delta)^{-s/2}(\chi_l f) \|_{L^2(D_j)} \Big)^2 \lesssim k^{-2s} \sum_{j \in \N} \Big( \sum_{|l - j|> 2} 2^{j\delta} \| \Phi \|_{L^1} \| \chi_l f \|_{L^2} \Big)^2 
\label{term:withSUM}
\end{equation}
where
\[\| \Phi \|_{L^1} = \frac{1}{|2^j - 2^l|^{2N - n}} \int_{\R^n} \frac{1}{|x|^{2N}} [ 1 - \phi( |x| )] \,dx. \]
The right-hand side of \eqref{term:withSUM} can be bounded as follows
\begin{equation*}
k^{-2s} \sum_{j \in \N} \Big( \sum_{|l - j|> 2}  \frac{2^{(j - l)\delta}}{|2^j - 2^l|^{2N - n}} 2^{l\delta} \| \chi_l f \|_{L^2} \Big)^2 \lesssim k^{-2s}  \sum_{j \in \N} \Big( \sum_{|l - j|> 2}  \frac{1}{2^{|j - l|(2N - n - |\delta|)}} 2^{l\delta} \| \chi_l f \|_{L^2} \Big)^2.
\end{equation*}
To check the previous inequality, it may be convenient to notice that $|2^j - 2^l| > 2^{\max(j,l) - 1}$ for $|l - j|>0$. By Young's inequality for convolutions, the right-hand side can be bounded by
\[k^{-2s} \Big( \sum_{j \in \N} 2^{-j(2N - n - |\delta|)} \Big)^2 \sum_{l\in \N} 2^{2l\delta} \| \chi_l f \|^2_{L^2}. \]
Choosing $N > (n + |\delta|)/2 $ we have that the right-hand side of \eqref{term:withSUM} is bounded by
\[k^{-2s} \sum_{l \in \N} 2^{2l\delta} \| \chi_l f \|^2_{L^2} \lesssim k^{-2s} \| f \|^2_{L^{2,\delta}}.\]
This ends the proof of this lemma.
\end{proof}

We turn our attention to estimates \eqref{es:re_part} and \eqref{es:im_part}:
\paragraph*{\emph{Proof of the inequality \eqref{es:re_part}}} Let $D_j$ with $j \in \N$ be as in the proof of Proposition \ref{prop:st} and bound
\begin{equation}
\| \mathcal{P}_k f \|_{H^{1, -\delta}_k} \lesssim \Big( \sum_{j \in \N} 2^{-(2\delta - 1)j} \Big)^{1/2} \sup_{j \in \N} \big( 2^{-j/2} \|(k^2 - \Delta)^{1/2} \mathcal{P}_k f\|_{L^2(D_j)} \big).
\label{es:HdeltaFBesov}
\end{equation}
As we see below, the inequality \eqref{es:re_part} will be a consequence of the following lemma.

\begin{lemma}\label{lem:KPV} \sl Let $\lambda $ be a positive constant and $\mathcal{P}$ be given by
\[\mathcal{P} f(x) = \mathrm{p.v.} \int_{\R^n} \frac{e^{i x\cdot \xi}\widehat{f}(\xi)}{1 - |\xi|^2} \, d\xi.\]
Then,
\[ \int_{|x| < R} |(1 - \Delta)^{1/2} \mathcal{P} f(x)|^2 \, dx \lesssim R \lambda \| f \|^2_{L^2} \]
for all $R > 0$ and $f \in \mathcal{S}(\R^n)$ such that $\supp f \subset \{ \lambda / 2 \leq |x| \leq 2 \lambda \}$.
\end{lemma}

Lemma \ref{lem:KPV} is a slight modification of Lemma 2.4 in \cite{MR1230709} due to Kenig, Ponce and Vega.

Let us show that Lemma \ref{lem:KPV} implies \eqref{es:re_part}. From the inequality \eqref{es:HdeltaFBesov} we have
\begin{align}
\| \mathcal{P}_k f \|_{H^{1, -\delta}_k} &\lesssim \sup_{R > 0} \big( R^{-1/2} \|(k^2 - \Delta)^{1/2} \mathcal{P}_k f\|_{L^2(|x| \leq R)} \big)
\label{es:thegoodnorm} \\
&= k^{-n/2-1/2} \sup_{R > 0} \big( R^{-1/2} \|(1 - \Delta)^{1/2} \mathcal{P} S_k f\|_{L^2(|x| \leq R)} \big),
\label{es:withoutk}
\end{align}
where $S_k f(x) = f(x/k)$. Let $\chi$ be a smooth function with values in $[0, 1]$ such that $\supp \chi \subset D_0$ and $\chi(x) = 1$ for all $|x| \leq 1/2$. Let $\chi_j$ with $j \in \Z $ denote $\chi_j (x) = \chi(x/2^{j+1}) - \chi(x/2^{j})$ and note that $\supp \chi_j \subset \{  2^{j - 1} \leq |x| \leq 2^{j + 1}\}$. By construction, $\sum_{j \in \Z} \chi_j(x) = 1$ for every $x \in \R^n \setminus \{ 0 \}$. Then, by Lemma \ref{lem:KPV}, we have that
\begin{align*}
\| \mathcal{P}_k f \|_{H^{1, -\delta}_k} &\lesssim k^{-n/2-1/2} \sum_{j \in \Z} \sup_{R > 0} \big( R^{-1/2} \|(1 - \Delta)^{1/2} \mathcal{P} S_k (\chi_j f)\|_{L^2(|x| < R)} \big) \\
&\lesssim k^{-n/2-1/2} \sum_{j \in \Z} k^{1/2} 2^{j/2} \| S_k(\chi_j f) \|_{L^2}\\
&= \sum_{j \in \Z} 2^{j/2} \| \chi_j f \|_{L^2}.
\end{align*}
This last term can be manipulate to obtain inequality \eqref{es:re_part}:
\begin{align*}
\sum_{j \in \Z} 2^{j/2} \| \chi_j f \|_{L^2} &\leq \sum_{j < 0} 2^{j/2} \| f \|_{L^2(D_0)} + \sum_{j \in \N} 2^{j(1-2\delta)/2} 2^{j\delta} \| \chi_j f \|_{L^2} \\
& \lesssim \| f \|_{L^2(D_0)} + \Big( \sum_{j \in \N} 2^{j(1-2\delta)} \Big)^{1/2} \Big( \sum_{j\in \N} 2^{2j\delta} \| \chi_j f \|^2_{L^2} \Big)^{1/2} \\
& \lesssim \| f \|_{L^{2,\delta}}.
\end{align*}

In order to prove \eqref{es:re_part}, the only ingredient to be checked is Lemma \ref{lem:KPV}.
\begin{proof}[Proof of Lemma \ref{lem:KPV}]
Start by writing
\[(\rm I - \Delta)^{1/2} \mathcal{P} f = \mathcal{K}f + \mathcal{L}f, \]
where the operators $\mathcal{K}$ and $\mathcal{L}$ are given by
\begin{align*}
\mathcal{K}f(x) &= \mathrm{p.v.} \int_{\R^n} e^{i x\cdot \xi} \frac{\phi(1 - |\xi|^2) (1 + |\xi|^2)^{1/2}}{1 - |\xi|^2} \widehat{f}(\xi) \, d\xi, \\
\mathcal{L}f(x) &= (2\pi)^{n/2} \mathcal{J}_{1} m(D) f,
\end{align*}
where $\phi$ is a smooth function defined on $\R$ and taking values in $[0,1]$ such that $\supp \phi \subset [- 1/(8n), 1/(8n)]$ and $\phi(t) = 1$ for all $|t|\leq 1/(16n)$, $\mathcal{J}_{1}$ is the Bessel potential defined as in the section \ref{sec:regularity}, and $m(D)$ is the multiplier with symbol
\[m(\xi) = \frac{[1 - \phi(1 - |\xi|^2)] (1 + |\xi|^2)}{1 - |\xi|^2}.\]
By H\"older's inequality with $1/2 = 1/p + 1/(2n)$ we have
\[\| \mathcal{L}f \|^2_{L^2(|x| < R)} \lesssim R\, \| \mathcal{L}f \|^2_{L^p}.\]
By Sobolev embeddings\footnote{This is nothing but Lemma 2 in \S 3.2 of chapter V in \cite{MR0290095} with $\alpha = 1$, Young's inequality for functions convolved with finite measures and Theorem 1 in \S 1.2 of chapter V in \cite{MR0290095} with $\alpha = 1$}, the right-hand side is bounded by a multiple constant of
\begin{equation}
R\, \| m(D)f \|^2_{L^{p'}}
\label{term:multiplier}
\end{equation}
with $1/p + 1/p' = 1$. By the Mikhlin--H\"ormander multiplier theorem and H\"older's inequality with $1/p' = 1/2 + 1/(2n)$, we have that \eqref{term:multiplier} is bounded by a constant multiple of
\[R\, \| f \|^2_{L^{p'}} \lesssim R \lambda \, \| f \|^2_{L^2}. \]
Therefore, we have
\[\| \mathcal{L}f \|^2_{L^2(|x| < R)} \lesssim R \lambda \, \| f \|^2_{L^2}. \]
To finish the proof of Lemma \ref{lem:KPV}, it only remains to prove the corresponding  estimate for the singular part $\mathcal{K}$. To this end, we introduce a partition of unity $\{ \psi_1, \dots, \psi_{2n} \}$ of $\Gamma = \{ \xi \in \R^n : 1/2 < |\xi| < 3/2 \}$ subordinated to $\{ \Gamma_1, \dots, \Gamma_{2n} \}$, where
\begin{equation*}
\Gamma_{2l-1} = \{ \xi \in \Gamma: \xi_l > 1/(2\sqrt{n}) \}, \qquad \Gamma_{2l} = \{ \xi \in \Gamma : \xi_l < - 1/(2\sqrt{n}) \}.
\end{equation*}
This partition of unity can be assumed to satisfy:
\begin{itemize}
\item[(i)] $\psi_2(x) = \psi_1 (I_1 x)$ with $I_1$ the reflection given by the matrix $[ - e_1 | e_2 | \dots | e_n]$ and $ e_1, \dots, e_n $ the elements of canonical base of $\R^n$;
\item[(ii)] for $l = 2, \dots, n$,
\[\psi_{2l-1} (x) = \psi_1 (I_l x), \qquad \psi_{2l} (x) = \psi_2 (I_l x) \]
with $I_l$ the rotation given the matrix $[- e_l | e_2 | \dots |e_{l - 1} | e_1 | e_{l + 1} | \dots | e_n]$.
\end{itemize}
Thus, $\mathcal{K}f = \sum_{j = 1}^{2n} \mathcal{K}_j f $ with
\[\mathcal{K}_j f(x) = \mathrm{p.v.} \int_{\R^n} e^{i x\cdot \xi} \frac{\phi(1 - |\xi|^2) (1 + |\xi|^2)^{1/2}}{1 - |\xi|^2} \psi_j(\xi) \widehat{f}(\xi) \, d\xi,\]
and, in order to prove the lemma, it is enough to prove that
\begin{equation}
\| \mathcal{K}_1 f \|^2_{L^2(|x| < R)} \lesssim R \lambda \, \| f \|^2_{L^2}
\label{es:S1}
\end{equation}
for all $f \in \mathcal{S}(\R^n)$ with $\supp f \subset \{ \lambda / 2 \leq |x| \leq 2 \lambda \}$.

In order to prove \eqref{es:S1}, we first write
\[\Phi (\xi) = \frac{\phi(1 - |\xi|^2) (1 + |\xi|^2)^{1/2}}{1 - |\xi|^2} \psi_1(\xi). \]
Note that $\phi$ has been chosen in such a way that it satisfies
\[\supp \Phi \subset \{ \xi = (\xi_1, \xi') \in \R \times \R^{n - 1} : 1 - |\xi'|^2 > 1/(8n) \}\]
and for what follows we write
\[\Phi(\xi) = \frac{\Psi(\xi)}{(1 - |\xi'|^2)^{1/2} - \xi_1}  \quad \text{with} \quad \Psi(\xi) = \frac{\phi(1 - |\xi|^2) (1 + |\xi|^2)^{1/2}}{(1 - |\xi'|^2)^{1/2} + \xi_1} \psi_1(\xi). \]
This also allows us to use the expression
\begin{equation*}
\mathcal{K}_1 f(x) = \lim_{\epsilon \rightarrow 0} \int_{0 < \epsilon < |(1 - |\xi'|^2)^{1/2} - \xi_1|}  e^{i x\cdot \xi} \Phi(\xi) \widehat{f}(\xi) \, d\xi.
\end{equation*}
Let $\mathcal{F}'$ denote the Fourier transform on the variable $x'$ with dual variable $\xi'$. It follows that
\[ \mathcal{K}_1 f(x) = \int_{\R^{n - 1}}  e^{i x'\cdot \xi'} \int_\R  a(x_1, y_1, \xi') \mathcal{F}'f(y_1, \xi') \,dy_1 \,d\xi'\]
with
\[ a(x_1, y_1, \xi') = \mathrm{p.v.} \int_\R   e^{i (x_1 - y_1) \xi_1}  \frac{\Psi(\xi)}{(1 - |\xi'|^2)^{1/2} - \xi_1} \, d\xi_1\]
where the principal value is understood as $\lim_{\epsilon \rightarrow 0} \int_{0 < \epsilon < |(1 - |\xi'|^2)^{1/2} - \xi_1|}$.
By Plancherel's identity we obtain
\begin{align*}
\| \mathcal{K}_1 f \|^2_{L^2(|x| < R)} &\leq \int_{|x_1| < R} \int_{\R^{n - 1}} |\mathcal{K}_1 f(x_1, x')|^2 \,dx' \,dx_1 \\
&\sim \int_{|x_1| < R} \int_{\R^{n - 1}} \Big| \int_\R  a(x_1, y_1, \xi') \mathcal{F}'f(y_1, \xi') \,dy_1 \Big|^2 \,d\xi' \,dx_1.
\end{align*}
Furthermore, since $\supp f \subset \{ \lambda / 2 \leq |x| \leq 2 \lambda \}$, we have by applying the Cauchy--Schwarz inequality, and then Plancharel's identity, that
\begin{align*}
\| \mathcal{K}_1 f \|^2_{L^2(|x| < R)} &\lesssim \lambda \int_{|x_1| < R} \int_{\R^{n - 1}} \sup_{\lambda / 2 \leq |y_1| \leq 2 \lambda} |a(x_1, y_1, \xi')|^2 \int_\R |\mathcal{F}'f(y_1, \xi')|^2 \,dy_1  \,d\xi' \,dx_1 \\
&\lesssim \lambda R \sup_{|x_1| < R} \sup_{\lambda / 2 \leq |y_1| \leq 2 \lambda} |a(x_1, y_1, \xi')|^2 \| f \|^2_{L^2}.
\end{align*}
Therefore, to conclude the proof of the inequality \eqref{es:S1}, it is enough to show that
\begin{equation}
\sup_{|x_1| < R} \sup_{\lambda / 2 \leq |y_1| \leq 2 \lambda} |a(x_1, y_1, \xi')| \lesssim 1.
\label{es:ax_1y_1xi'}
\end{equation}
Now recall the following identity
\[ \frac{1}{(2\pi)^{1/2}} \mathrm{p.v.} \int_\R \frac{e^{-ist}}{t} \varphi(t) \, dt = -i \frac{1}{2}  \int_\R \sign (s - t) \widehat{\varphi} (t) \,dt. \]
Consequently, $a(x_1, y_1, \xi')$ is a multiple constant of
\[e^{i(x_1 - y_1)(1 - |\xi'|^2)^{1/2}} \int_\R \sign(x_1 - y_1 - t) \int_\R e^{-its} \Psi((1 - |\xi'|^2)^{1/2} - s, \xi') \, ds \,dt,\]
which can be written as a convolution changing variables according to $\xi_1 = (1 - |\xi'|^2)^{1/2} - s $:
\[\int_\R e^{i(x_1 - y_1 - t)(1 - |\xi'|^2)^{1/2}} \sign(x_1 - y_1 - t) \int_\R e^{it\xi_1} \Psi(\xi) \, d\xi_1 \,dt.\]
Finally, by Young's inequality and then the non-stationary phase principle, we have
\[|a(x_1, y_1, \xi')| \lesssim \int_\R \Big| \int_\R e^{it\xi_1} \Psi(\xi) \, d\xi_1 \Big| \, dt \lesssim 1. \]
This proves \eqref{es:ax_1y_1xi'}, so that \eqref{es:S1} holds and the proof of Lemma \ref{lem:KPV} is over.
\end{proof}

\paragraph*{\emph{Proof of the inequality \eqref{es:im_part}}} In the same way we obtained \eqref{es:thegoodnorm}, we see that
\begin{equation*}
\| \mathcal{Q}_k f \|_{H^{1, -\delta}_k} \lesssim \sup_{R > 0} \big( R^{-1/2} \|(k^2 - \Delta)^{1/2} \mathcal{Q}_k f\|_{L^2(|x| < R)} \big).
\end{equation*}
Note that $(k^2 - \Delta)^{1/2} \mathcal{Q}_k f(x)$ is a constant multiple of
\[ \int_{ |\xi| = k } e^{i x\cdot \xi} \widehat{f}(\xi) \, d\sigma_k(\xi) = \frac{1}{k} \int_{ |\eta| = 1 } e^{i kx\cdot \eta} \widehat{S_kf}(\eta) \, d\sigma(\eta),\]
where $d\sigma$ is the volume form on $\{ |\eta| = 1 \}$ with $S_k f$ as in \eqref{es:withoutk}. Note that the previous expression is essentially the Fourier transform of the measure $\widehat{S_kf}(\eta) \, d\sigma(\eta)$ supported on $\{ |\eta| = 1 \}$. Because of the same considerations as in \eqref{es:withoutk}, we have that
\[ \| \mathcal{Q}_k f \|_{H^{1, -\delta}_k} \lesssim k^{-n/2-1/2} \sup_{R > 0} \Big( R^{-1/2} \big|\int_{ |\eta| = 1 } e^{i x\cdot \eta} \widehat{S_kf}(\eta) \, d\sigma(\eta) \big|_{L^2(|x| < R)} \Big). \]
By Theorem 7.1.26 in \cite{MR1065993}, we can bound the right hand side of the previous inequality and obtain
\[\| \mathcal{Q}_k f \|_{H^{1, -\delta}_k} \lesssim k^{-n/2-1/2} \Big( \int_{ |\eta| = 1 } | \widehat{S_kf}(\eta) |^2 \, d\sigma(\eta) \Big)^{1/2}. \]
By an appropriate duality argument (in the spirit of Theorem 14.1.1 in \cite{MR2108588}) applied to Theorem 7.1.26 in \cite{MR1065993} we have that
\[\Big( \int_{ |\eta| = 1 } | \widehat{S_kf}(\eta) |^2 \, d\sigma(\eta) \Big)^{1/2} \lesssim \sum_{j\in \Z} k^{1/2} 2^{j/2} \Big( \int_{k 2^{j-1} < |x| \leq k 2^j}  | S_k f (x) |^2 \,dx \Big)^{1/2}. \]
By performing a rescaling in $k$, we have
\[ \| \mathcal{Q}_k f \|_{H^{1, -\delta}_k} \lesssim k^{-n/2-1/2} k^{1/2} k^{n/2} \sum_{j\in \Z} 2^{j/2} \Big( \int_{2^{j-1} < |x| \leq 2^j}  | f (x) |^2 \,dx \Big)^{1/2}. \]
The right-hand side can be obviously manipulated as follows to obtain inequality \eqref{es:im_part}
\begin{align*}
\sum_{j\in \Z} 2^{j/2} \Big( \int_{2^{j-1} < |x| \leq 2^j}  | f (x) |^2 \,dx \Big)^{1/2} &\lesssim \sum_{j \leq 0} 2^{j/2} \Big( \int_{|x| \leq 1}  | f (x) |^2 \,dx \Big)^{1/2} \\
& \quad + \Big( \sum_{j\geq 1} 2^{j(1 - 2\delta)} \Big)^{1/2} \Big( \sum_{j\geq 1} 2^{2j\delta} \int_{D_j}  | f (x) |^2 \,dx \Big)^{1/2}\\
&\lesssim \| f \|_{L^{2,\delta}}
\end{align*}
with $D_j$ as in the proof of Proposition \ref{prop:st}.

\subsection{Boundedness of the multiplication operator}\label{sec:multiplication}
Consider any two $f,g \in \mathcal{S}(\R^n)$, the multiplication by $V$ is defined by
\[\langle V f, g \rangle = \langle V, f g \rangle, \]
where the brackets denote the corresponding dualities. Let $\chi_D$ be a compactly supported smooth function with values in $[0,1]$ such that $\chi(x) = 1$ for all $x \in D$. Since $V \in L^p_{-s}(\R^n)$ with support in $D$, there exists $W \in L^p(\R^n)$ such that $V = (\rm I - \Delta)^{s/2} W$, and consequently
\[\langle V f, g \rangle = \langle W, (\id - \Delta)^{s/2} (f_D g_D) \rangle\]
with $f_D = \chi_D f$ and $g_D = \chi_D g$.
Let $\phi_\varepsilon$ be as in \eqref{id:def_of_fepsilon} and set $W^\sharp = \phi_\varepsilon \ast W$ and $W^\flat = W - W^\sharp$. For $s - n/p < t < s$ we have that
\[\langle V f, g \rangle = \langle (\id - \Delta)^{t/2} W^\sharp , (\id - \Delta)^{(s - t)/2} (f_D g_D) \rangle + \langle W^\flat, (\id - \Delta)^{s/2} (f_D g_D) \rangle \]
and, by H\"older's inequality,
\begin{equation}
\label{eq:multiplication_op_holder_est}
\big| \langle V f, g \rangle \big| \leq \| (\id - \Delta)^{t/2} W^\sharp \|_{L^q}  \| (\id - \Delta)^{(s - t)/2} (f_D g_D) \|_{L^{q'}} + \| W^\flat \|_{L^p} \| (\id- \Delta)^{s/2} (f_D g_D) \|_{L^{p'}}
\end{equation}
with $q = n/(s-t)$ and $q'$ and $p'$ the dual exponents of $q$ and $p$, respectively. Since $q > p$, we have by Young's inequality that
\[\| (\id - \Delta)^{t/2} W^\sharp \|_{L^q} \lesssim \varepsilon^{-t+n/q-n/p} \| W \|_{L^p}.\]
Furthermore, we will prove in Lemma \ref{lem:KP} that
\begin{align}
& \| (\id - \Delta)^{(s - t)/2} (fg) \|_{L^{q'}} \lesssim \| f_D \|_{H^{s - t}} \| g_D \|_{H^{s - t}} \label{es:KP_s-t} \\
& \| (\id - \Delta)^{s/2} (fg) \|_{L^{p'}} \lesssim \| f_D \|_{H^s} \| g_D \|_{H^s}, \label{es:KP_s}
\end{align}
hence, for $k \geq 1$, we have
\begin{align}
\big| \langle V f, g \rangle \big| &\lesssim \varepsilon^{-t+n/q-n/p} \| W \|_{L^p}  \| f_D \|_{H^{s - t}_k} \| g_D \|_{H^{s - t}_k} + \| W^\flat \|_{L^p} \| f_D \|_{H^s_k} \| g_D \|_{H^s_k} \nonumber \\
& \lesssim \varepsilon^{-t+n/q-n/p} \| W \|_{L^p}  \| f \|_{H^{s - t,-\delta}_k} \| g \|_{H^{s - t,-\delta}_k} + \| W^\flat \|_{L^p} \| f \|_{H^{s,-\delta}_k} \| g \|_{H^{s,-\delta}_k}. \label{es:compact}
\end{align}
The estimate \eqref{es:compact} will be justified by Lemma \ref{lem:compact}. Finally, by Proposition \ref{prop:st}, we have
\[\big| \langle V f, g \rangle \big| \lesssim \big( \varepsilon^{-t+n/q-n/p} k^{-2t} \| W \|_{L^p} + \| W^\flat \|_{L^p} \big)  \| f \|_{H^{s,-\delta}_k} \| g \|_{H^{s,-\delta}_k}. \] 
Choosing $\varepsilon = k^{-1/2}$, a simple duality argument show that the estimate stated in Proposition \ref{prop:multiplication} holds.

In order to end the proof of this proposition, we need to show inequalities \eqref{es:KP_s-t}, \eqref{es:KP_s} and \eqref{es:compact}. Note that they follow from the next lemmas:
\begin{lemma}\label{lem:KP}\sl Let $s > 0$ and $p \in (1, \infty)$ satisfy $p \geq n/s$. Then,
\[\| fg \|_{L^{p'}_s} \lesssim \| f \|_{H^s} \| g \|_{H^s}\]
for all $f,g \in \mathcal{S}(\R^n)$ with $1/p + 1/p' = 1$.
\end{lemma}
\begin{proof}
The Kato--Ponce inequality (see \cite{MR3200091} and the original reference \cite{MR951744}) ensures that
\[\| fg \|_{L^{p'}_s} \lesssim \| f \|_{H^s} \| g \|_{L^r} + \| f \|_{L^r} \| g \|_{H^s} \]
for $1/p' = 1/2 + 1/r$, which is equivalent to $1/2 = 1/p + 1/r$. By the Sobolev embeddings, we have that
\[\| fg \|_{L^{p'}_s} \lesssim \| f \|_{H^s} \| g \|_{H^t} + \| f \|_{H^t} \| g \|_{H^s} \]
with $t - n/2 = - n/r$. Since $t = n/2 - n/r = n/p \leq s$, the estimate claimed in the statement follows immediately from the last one.
\end{proof}

\begin{lemma}\label{lem:compact}\sl Let $s$ and $\delta$ be positive constants and $\phi \in \mathcal{S}(\R^n)$. Then,
\[\|\phi f\|_{H^s_k} \lesssim \| f \|_{H^{s, -\delta}_k}\]
for all $f \in \mathcal{S}(\R^n)$.
\end{lemma}
\begin{proof}
Note that it is enough to prove the lemma for $\delta = 2N$ with $N \in \N \setminus \{ 0 \}$. In this case, we just need to show that
\begin{equation}
\|T_N g\|_{H^s_k} \lesssim \| g \|_{L^2}
\label{es:test_ine}
\end{equation}
for all $g \in \mathcal{S}(\R^n)$, where
\[T_N g (x) = \frac{1}{(2\pi)^{n/2}} \int_{\R^n} e^{ix\cdot \xi} a_N(x, \xi) \widehat{g}(\xi) \,d\xi\]
and
\[a_N (x, \xi) = \phi(x) e^{-ix \cdot \xi} (\id - \Delta_\xi)^N \big(e^{i x\cdot \xi} (k^2 + |\xi|^2)^{-s/2} \big).\]
In order to check this claim, it suffices to test the inequality \eqref{es:test_ine} for the functions
\[g = (1 + |\centerdot|^2)^{-N} (k^2 - \Delta)^{s/2} f\]
with $f$ any function in $\mathcal{S}(\R^n)$, and note that
\[ T_N \big( (1 + |\centerdot|^2)^{-N} (k^2 - \Delta)^{s/2} f \big) = \phi f.\]

The inequality \eqref{es:test_ine} follows from two general results for pseudodifferential operators. To apply them, we first observe that $a_N$ is a smooth function in $\R^n\times \R^n$ and
\[|\partial^\beta_x \partial^\alpha_\xi a_N(x,\xi)| \leq A_{\alpha, \beta, N} (k + |\xi|)^{-s - |\alpha|} \]
for all multi-indices $\alpha$ and $\beta$. Then, by symbolic calculus for pseudodifferential operators (for example Theorem 2 of \S 3 in chapter VI of \cite{Stein}) we see that
\begin{equation}
\|T_N g\|_{H^s_k} \leq \| S_0 g \|_{L^2} + \| S_{-1} g \|_{L^2},
\label{es:symbolicCAL}
\end{equation}
where
\[S_j g (x) = \frac{1}{(2\pi)^{n/2}} \int_{\R^n} e^{ix\cdot \xi} b_j(x, \xi) \widehat{g}(\xi) \,d\xi\]
with
\[b_0 (x, \xi) = \phi(x) e^{-ix \cdot \xi} (k^2 + |\xi|^2)^{s/2} (1 - \Delta_\xi)^N \big(e^{i x\cdot \xi} (k^2 + |\xi|^2)^{-s/2} \big),\]
and $|\partial^\beta_x \partial^\alpha_\xi b_{-1}(x,\xi)| \leq A_{\alpha, \beta, N} (k + |\xi|)^{-1 - |\alpha|} $. By the $L^2$ boundedness of pseudodifferential operators(for example Theorem 1 of \S 3 in chapter VI of \cite{Stein}), we have that the right-hand side of \eqref{es:symbolicCAL} can be bounded as
\[\| S_0 g \|_{L^2} + \| S_{-1}\circ (k^2-\Delta)^{1/2} ((k^2-\Delta)^{-1/2}  g) \|_{L^2} \lesssim \| g \|_{L^2} + \| g \|_{H^{-1}_k} \lesssim \| g \|_{L^2}. \]
This ends the proof of this lemma.
\end{proof}

\section{Reconstruction of the local strength}
\label{sec:inverse_problem}

Let us recall that our aim is to reconstruct $\mu$, the local strength of $q$, from one single realization of the measurement data
\[M(\tau,\theta) = \lim_{K \rightarrow \infty} \frac{1}{K} \int_K^{2K} k^m u^\infty(k, \theta, -\theta) \overline{u^\infty(k + \tau, \theta, -\theta)}  \, dk\]
given for some set of $\tau \geq 0$ and $\theta\in \mathbb{S}^{n-1}$.
Recall that, by Proposition \ref{prop:SobolevRegularity}, we have $q \in L^p_s (\R^n)$ with $1 < p < \infty$ and $s < (m - n)/2$ almost surely. In the section \ref{sec:direct-scattering}, we have studied the direct scattering theory for a potential under slightly more general regularity assumptions, and shown that the backscattering far-field pattern
\begin{equation}
u^\infty(k, \theta, -\theta) = c_n \int_{\R^n} e^{ik\theta \cdot y} q(y) (e^{ik\theta \cdot y} + u_{\rm sc}(y)) \,dy\label{term:FULLfar-filed}
\end{equation}
can be expressed with the Born series \eqref{id:neumann_series} of the scattered wave 
for $k \geq k_0(\omega)$ almost surely, where $k_0$ is the threshold wavelength in Remark \ref{re:k0}. Here $c_n$ is a constant depending only on $n$. 

Below, we give proofs to Theorems \ref{thm:main_theorem_rigorous} and \ref{thm:main_theorem_born_approx}. 
Notice that for the full non-linear inverse scattering problem in Theorem \ref{thm:main_theorem_rigorous}, we restrict to the case $n = m = 3$. This condition will be needed in sections \ref{subsec:2nd order} and \ref{subsec:higher_order}.

In a nutshell, the idea behind the reconstruction of $\mu$ consists of writing $u^\infty$ as the Born series
\[u^\infty (k, \theta, -\theta) = \sum_{j \in \N \setminus \{ 0 \}} u^\infty_j(k, \theta, -\theta) \]
for $k\geq k_0(\omega)$ and $\theta\in \mathbb{S}^{n-1}$
in such a way that the higher order terms
\begin{equation}
\frac{1}{K} \int_K^{2K} k^m u^\infty_j(k, \theta, -\theta) \overline{u^\infty_l(k + \tau, \theta, -\theta)}  \, dk
\label{term:remainder}
\end{equation}
with $j + l \geq 3$ are negligible in comparison with the realization of 
\begin{equation}
\frac{1}{K} \int_K^{2K} k^m u^\infty_1(k, \theta, -\theta) \overline{u^\infty_1(k + \tau, \theta, -\theta)}  \, dk.
\label{term:leading}
\end{equation}
Next, one proceeds by proving that the limit of \eqref{term:leading} as $K\to\infty$ provides enough information to reconstruct $\mu$, when given at multiple values of $\tau$ and $\theta$. Recall that according to \eqref{id:neumann_series} we have
\begin{equation}
u^\infty_j(k, \theta, -\theta) = c_n \int_{\R^n} e^{ik\theta \cdot y} q(y) (\mathcal{R}^+_k \circ q)^{j - 1} u_0(y) \,dy \label{term:INDIfar-field}
\end{equation}
where $u_0(x) = e^{ik\theta \cdot x} $ and $c_n$ is the same as in \eqref{term:FULLfar-filed}. With the expression \eqref{term:INDIfar-field} the connection of measurement data and the statistics of $q$ becomes apparent.

For the sake of clarity, we first reconstruct the local strength assuming $\mathbb{E} q=0$, and then in the section \ref{sec:non-zero-mean} we consider the more general case where $\mathbb{E}q$ is smooth and supported in $D$.

\subsection{Single backscattering} \label{subsec:single}

In order to prove that the limit of \eqref{term:leading} coincides almost surely with a deterministic function, we need to apply suitable ergodicity arguments.
The following theorem (available in \cite[p. 94]{CramerLeadbetter}) provides a useful condition:

\begin{theorem}\sl
\label{thm:aux_ergodicity}
Let $X_t$ with $t\geq 0$ be a real-valued stochastic process with continuous paths and zero-mean $\expec X_t = 0$. Assume that for some positive constants $c,\epsilon$ the condition
\begin{equation*}
	|\expec (X_t X_{t+r})| \leq c(1+r)^{-\epsilon}
\end{equation*}
holds for all $t \geq 0 $ and $r \geq 0$. Then,
\begin{equation*}
	\lim_{T\to\infty} \frac 1 {T} \int_{0}^T X_t \, dt = 0 \quad \text{almost surely}.
\end{equation*}
\end{theorem}
Note that under the same assumptions of this theorem,
\begin{equation*}
	\lim_{T\to\infty} \frac{1}{T} \int_T^{2T} X_t \, dt = 0 \quad \text{almost surely}.
\end{equation*}
The single backscattering $u_1^ \infty(k, \theta, -\theta)$ in the formula \eqref{term:leading} is described by a constant multiple of
\[\langle q, e^{i2k\theta \cdot y} \rangle = U_k + i V_k, \]
where $U_k$ and $V_k$ denote the real and imaginary parts. We can rewrite the product in \eqref{term:leading} as a constant multiple of
\begin{equation}
\label{id:decompo}
\begin{aligned}
2 \langle q, e^{i2k\theta \cdot y} \rangle \overline{\langle q, e^{i2(k+\tau)\theta \cdot y} \rangle} &= (1 + i) ( U_k^2 + U_{k+\tau}^2 + V_k^2 + V_{k+\tau}^2 ) - (U_k - U_{k+\tau})^2 \\
&\quad - (V_k - V_{k+\tau})^2 - i(U_k + V_{k+\tau})^2 - i(V_k - U_{k+\tau})^2.
\end{aligned}
\end{equation}
Let $W_k$ denote any of random variables
\begin{equation}
\label{eq:list_of_random_variables}
U_k,\, U_{k+\tau},\, V_k,\, V_{k+\tau},\, U_k - U_{k+\tau},\, V_k - V_{k+\tau},\, U_k + V_{k+\tau},\, V_k - U_{k+\tau}.
\end{equation}
Using Theorem \ref{thm:aux_ergodicity}, we will prove that
\begin{equation}
\label{id:lim}
\lim_{K \to \infty} \frac{1}{K} \int_K^{2K} k^m (W_k^2 - \mathbb{E}W_k^2) \,dk = 0
\end{equation}
almost surely, and hence we deduce by \eqref{id:decompo} that
\begin{equation}
\begin{aligned} 
\lim_{K \rightarrow \infty} \frac{1}{K} \int_K^{2K} & k^m u^\infty_1(k, \theta, -\theta) \overline{u^\infty_1(k + \tau, \theta, -\theta)}  \, dk \\ &= \lim_{K \rightarrow \infty} \frac{1}{K} \int_K^{2K} k^m \mathbb{E}( u^\infty_1(k, \theta, -\theta) \overline{u^\infty_1(k + \tau, \theta, -\theta)} )  \, dk
\end{aligned}
\label{id:ergodicity}
\end{equation}
almost surely.

According to Theorem \ref{thm:aux_ergodicity}, identity \eqref{id:lim} holds if 
there exists an $\varepsilon > 0$ such that
\begin{equation}\label{es:sufCOND}
\big|\mathbb{E} \big(k^m (W_k^2 - \mathbb{E}W_k^2 ) (k + r)^m (W_{k + r}^2 - \mathbb{E}W_{k + r}^2 ) \big)\big| \lesssim (1 + r)^{-\varepsilon}
\end{equation}
for all $k \geq 1$ and $r \geq 0$. In order to verify condition \eqref{es:sufCOND}, we observe that $(W_k, W_{k + r})$ is always a centred Gaussian random vector (which holds because of \eqref{term:gaussianVECTOR} and $\mathbb{E}q = 0$), and we use the following lemma.
\begin{lemma} 
\label{lemma:gaussian_pair}
\sl Let $X$ and $Y$ be two random variables such that the pair $(X, Y)$ is a Gaussian random vector. If $\mathbb{E} X = \mathbb{E} Y = 0$, then
\[\mathbb{E} ((X^2 - \mathbb{E}X^2) (Y^2 - \mathbb{E}Y^2)) = 2 (\mathbb{E}(XY))^2. \]
\end{lemma}
\begin{proof}
Note that one can assume that the variances be $\mathbb{E}X^2 = \mathbb{E}Y^2 = 1$. Thus, one is reduced to prove
\begin{equation}
\label{id:variance1}
\mathbb{E} ((X^2 - 1) (Y^2 - 1)) = 2 (\mathbb{E}(XY))^2
\end{equation}
for $X$ and $Y$ satisfying $\mathbb{E}X^2 = \mathbb{E}Y^2 = 1$. In order to prove \eqref{id:variance1}, it is enough to show that it holds for a Gaussian vector $(X, Y')$ having the same probability law as $(X, Y)$. Indeed,
\begin{align*}
\mathbb{E} ((X^2 - 1) (Y^2 - 1)) &= \int_{\R^2} (x^2 - 1) (y^2 - 1) \, d\mathbb{P}_{(X,Y)}(x,y) \\
& = \int_{\R^2} (x^2 - 1) (y^2 - 1) \, d\mathbb{P}_{(X,Y')}(x,y) = \mathbb{E} ((X^2 - 1) ((Y')^2 - 1))
\end{align*}
and
\[\mathbb{E}(XY) = \int_{\R^2} xy \, d\mathbb{P}_{(X,Y)}(x,y) = \int_{\R^2} xy \, d\mathbb{P}_{(X,Y')}(x,y) = \mathbb{E}(XY'). \]
Let $X'$ be a Gaussian random variable with mean $0$, variance $1$ and independent of $X$. Consider $Y' = \cos \alpha X + \sin \alpha X' $ with $\cos \alpha = \mathbb{E} (XY)$, which is possible since $|\mathbb{E} (XY)| \leq 1$. Note that $Y'$ is a Gaussian random variable and the pair $(X, Y')$ is a Gaussian random vector. Moreover, since
\[ \left( \begin{array}{c}
\mathbb{E} X \\ \mathbb{E} Y 
\end{array} \right) = \left( \begin{array}{c}
\mathbb{E} X \\ \mathbb{E} Y' 
\end{array} \right), \qquad \left( \begin{array}{c c}
\mathbb{E} X^2 & \mathbb{E} (XY)  \\ \mathbb{E} (XY) &  \mathbb{E} Y^2 
\end{array} \right) = \left( \begin{array}{c c}
\mathbb{E} X^2 & \mathbb{E} (XY')  \\ \mathbb{E} (XY') &  \mathbb{E} (Y')^2 
\end{array} \right), \]
the Gaussian vector $(X, Y)$ and $(X, Y')$ are equally distributed. Therefore, it only remains to show that \eqref{id:variance1} holds for $(X, Y')$, but this is a simple computation that can be verified using that $\mathbb{E}(X Y') = \cos \alpha$,
\begin{align*}
\mathbb{E} ((X^2 - \mathbb{E}X^2) ((Y')^2 - \mathbb{E}(Y')^2)) &= (\cos \alpha)^2 \mathbb{E}X^4 - (\cos \alpha)^2 + (\sin \alpha)^2 \mathbb{E}(X^2 (X')^2) - (\sin \alpha)^2 \\
& \quad + 2 \cos \alpha \sin \alpha \, \mathbb{E}(X^3 X')
\end{align*}
and
\[\mathbb{E}X^4 = 3, \qquad \mathbb{E}(X^2 (X')^2) = \mathbb{E}X^2 \mathbb{E}(X')^2 = 1, \qquad \mathbb{E}(X^3 X') = \mathbb{E}X^3 \mathbb{E}X' = 0. \]
This concludes the proof.
\end{proof}

Using Lemma \ref{lemma:gaussian_pair}, the identity \eqref{id:lim} holds if there exist constants $c > 0$, which may depend on $\tau$, and $\varepsilon > 0$ such that
\begin{equation}
	\label{eq:decay_of_Wk}
	|\mathbb{E} (k^{m/2}(k+r)^{m/2}W_k W_{k + r}) | \le c (1 + r)^{-\varepsilon}.
\end{equation}
Note that the inequality \eqref{eq:decay_of_Wk} is a consequence of the following proposition and therefore \eqref{id:lim} holds and consequently \eqref{id:ergodicity}.

\begin{proposition} \label{prop:first_order_decay}\sl Suppose the potential satisfies $\expec q = 0$. Then, there exists a known constant $c_{n,m}$, depending on $n$ and $m$, such that
\begin{equation}
\mathbb{E}(u_1^ \infty(k, \theta, -\theta) \overline{u_1^ \infty(k + \tau, \theta, -\theta)}) = c_{n,m}  k^{-m} \widehat{\mu} (2\tau\theta) + \mathcal{O}(k^{-m-1})
\label{id:approxCORRELATION}
\end{equation}
for $k \geq 1/2$ and $\tau \geq 0$. Moreover, for all $k_1, k_2 > 0$, we have that
\begin{align}
\label{es:correlREAL}
|\mathbb{E}(U_{k_1} U_{k_2})| &\lesssim k_1^{-m} (1 + |k_1 - k_2|)^{-N}\\
\label{es:correlIMAGINARY}
|\mathbb{E}(V_{k_1} V_{k_2})| &\lesssim k_1^{-m}(1 + |k_1 - k_2|)^{-N}\\
\label{es:correlREAL-IMAGI}
|\mathbb{E}(U_{k_1} V_{k_2})| &\lesssim k_1^{-m}(1 + |k_1 - k_2|)^{-N}
\end{align}
for all $N \in \N$.
\end{proposition}
\begin{proof}
Note that
\begin{equation}
\mathbb{E}(\langle q, e^{i2k_1\theta \cdot y} \rangle \overline{\langle q, e^{i2k_2\theta \cdot x} \rangle}) = \int_{\R^n} \Big( \int_{\R^n} K_q(x,y) e^{-i2 k_1 \theta \cdot(x-y)} \,dy \Big)  e^{-i2 (k_2 - k_1) \theta \cdot x} \, dx.
\label{id:startingPOINT}
\end{equation}
By \eqref{id:expansion}, we have that
\[\begin{aligned}
\mathbb{E}(\langle q, e^{i2k_1\theta \cdot y} \rangle & \overline{\langle q, e^{i2k_2\theta \cdot x} \rangle}) \\ &= \int_{\R^n} \mu(x) |2 k_1|^{-m} e^{-i2 (k_2 - k_1) \theta \cdot x} \,dx + \int_D a(x, 2 k_1 \theta) e^{-i2 (k_2 - k_1) \theta \cdot x} \,dx
\end{aligned}\]
for $k_1 \geq 1/2$, which implies \eqref{id:approxCORRELATION}.

On the other hand, by \eqref{id:startingPOINT} and \eqref{id:pre-expansion}, we have
\[\mathbb{E}(\langle q, e^{i2k_1\theta \cdot y} \rangle \overline{\langle q, e^{i2k_2\theta \cdot x} \rangle}) = \int_{\R^n} c_q(x, 2 k_1 \theta) \chi(x)  e^{-i2 (k_2 - k_1) \theta \cdot x} \, dx\]
with $\chi \in C^\infty_0(\R^n)$ such that $\chi(x) = 1$ for all $x \in D$, which implies, by the non-stationary phase principle, that
\begin{equation}
|\mathbb{E}(\langle q, e^{i2k_1\theta \cdot y} \rangle \overline{\langle q, e^{i2k_2\theta \cdot x} \rangle})| \lesssim (1 + k_1)^{-m} (1 + |k_1 - k_2|)^{-N}
\label{es:primera}
\end{equation}
for all $N \in \N$. By the same kind of considerations, one can proves that
\[\mathbb{E}(\langle q, e^{i2k_1\theta \cdot y} \rangle \langle q, e^{i2k_2\theta \cdot x} \rangle) = \int_{\R^n} c_q(x, 2 k_1 \theta) \chi(x)  e^{i2 (k_1 + k_2) \theta \cdot x} \, dx\]
with $\chi \in C^\infty_0(\R^n)$ such that $\chi(x) = 1$ for all $x \in D$, which implies, again by the non-stationary phase principle, that
\begin{equation}
|\mathbb{E}(\langle q, e^{i2k_1\theta \cdot y} \rangle \langle q, e^{i2k_2\theta \cdot x} \rangle)| \lesssim (1 + k_1)^{-m} (1 + k_1 + k_2)^{-N}
\label{es:segunda}
\end{equation}
for all $N \in \mathbb{N}$.  Finally, the estimates \eqref{es:correlREAL}, \eqref{es:correlIMAGINARY} and \eqref{es:correlREAL-IMAGI} follow from \eqref{es:primera} and \eqref{es:segunda}.

\end{proof}

As a consequence of the identities \eqref{id:ergodicity} and \eqref{id:approxCORRELATION}, the information of the local strength of $q$ provided by the single backscattering can be recorded as follows.
\begin{corollary}\label{cor:first_order_data} \sl Suppose $\expec q = 0$. The Fourier transform of the local strength of $q$ can be recovered from the single backscattering as follows:
\begin{equation}
\lim_{K \rightarrow \infty} \frac{1}{K} \int_K^{2K} k^m u^\infty_1(k, \theta, -\theta) \overline{u^\infty_1(k + \tau, \theta, -\theta)}  \, dk = c_{n,m} \widehat{\mu} (2\tau\theta) \label{eq:first_order_integral}
\end{equation}
almost surely for any fixed $\tau \geq 0$ and $\theta \in \mathbb{S}^{n - 1}$.
\end{corollary}

\begin{proof}[Proof of Theorem \ref{thm:main_theorem_born_approx} in the case $\mathbb{E}q=0$.]
Corollary \ref{cor:first_order_data} connects the measurement data to the Fourier transform of $\mu$ at a point $2\tau\theta$. 
We can now proceed by repeating the same measurement at a countable dense set $\{(\tau_j,\theta_j)\}_{j\in\N} \subset \R_+\times {\mathbb S}^2$. The almost sure convergence takes place simultaneously is this countable dense set. Since $\mu$ is smooth with compact support, it follows that $\widehat{\mu}$ is in the Schwartz class, in particular is continuous. Then, we can recover $\widehat{\mu}$ everywhere from the countable dense set extending by continuity. Finally, the inverse formula of the Fourier transform provides $\mu$.
\end{proof}

The following consequence of identity \eqref{id:decompo} and Proposition \ref{prop:first_order_decay} is not needed for Corollary \ref{cor:first_order_data}. However, it will allow us to study the convergence speed of integral \eqref{eq:first_order_integral} in appendix \ref{sec:scaling_regimes}.

\begin{corollary}
\label{cor:covariances_in_full_first_order_integral}\sl
Suppose $\expec q = 0$ and
let $Z : \Omega \times \R_+ \to \mathbb{C}$ be the random process defined by
\begin{equation*}
	Z(k) = u^\infty_1(k, \theta, -\theta) \overline{u^\infty_1(k + \tau, \theta, -\theta)}.
\end{equation*}
It follows that 
\begin{align}
	\big| \expec \big[ (Z(k)-\expec Z(k)) (Z(k+r) - \expec Z(k+r)) \big] \big|
	 & \lesssim  (1 + \min(r, |r-\tau|))^{-N} \label{eq:2nd orderd_first_order_cov_1}\\
		\big| \expec \big[ (Z(k)-\expec Z(k)) \overline{(Z(k+r) - \expec Z(k+r))} \big] \big|
	& \lesssim  (1 + \min(r, |r-\tau|))^{-N} \label{eq:2nd orderd_first_order_cov_2}
\end{align}
for any $N$.
\end{corollary}
\begin{proof}
We observe that using identity
\eqref{id:decompo} one can write the expectation in \eqref{eq:2nd orderd_first_order_cov_2} (as well as \eqref{eq:2nd orderd_first_order_cov_1}) as a product of sums of random variables appearing in formula \eqref{eq:list_of_random_variables}. Multiplying the terms and applying triangle inequality yields a bound to the right hand side of \eqref{eq:2nd orderd_first_order_cov_2} given as a sum of terms
\begin{equation*}
	\big|\expec\big[(W_k^2-\expec W_k^2)((W'_{k+r})^2-\expec (W'_{k+r})^2)\big]\big|
	= 2 (\expec W_k W'_{k+r})^2,
\end{equation*}
where $W_k$ and $W'_{k+r}$ represent random variables given in \eqref{eq:list_of_random_variables} and we used Lemma \ref{lemma:gaussian_pair} for the identity. Finally, the result is obtained by applying  Proposition \ref{prop:first_order_decay} and the fact that the cross-covariance is computed at all pairs of frequencies $k_1 \in \{k,k+\tau\}$ and $k_2\in\{k+r,k+r+\tau\}$. Same reasoning applies to inequality \eqref{eq:2nd orderd_first_order_cov_2}.
\end{proof}

\subsection{2nd order backscattering} \label{subsec:2nd order} We now consider the interactions between the single and 2nd order backscattering for the case $n = m = 3$ and show that they are negligible. Recall that in this particular case $q \in L^p_{-s} (\R^3)$ for $1 < p < \infty$ and $s>0$ almost surely. According to \eqref{term:INDIfar-field}, the 2nd order backscattering $u^\infty_2(k,\theta, -\theta)$ is described by a constant multiple of
\begin{equation}
\int_{\R^3} \int_{\R^3} e^{ik \theta \cdot (x + y)} q(x) q(y) \Phi^+_k (x - y) \, dx \, dy,
\label{term:2nd order}
\end{equation}
where $ \Phi^+_k $ is the fundamental solution which satisfies the outgoing Sommerfeld radiation condition. Note that the integrals in \eqref{term:2nd order} have to be understood as a distributional pairing. The interaction to be considered now correspond to the terms \eqref{term:remainder}
with $j + l = 3$ and $j = l = 2$ and the goal is to show that they are negligible, more precisely to prove the following statement:
\begin{proposition} \label{prop:2nd orderscattering} \sl Let us assume that $\expec q = 0$. We have that, for every $\tau \geq 0$ and every $\theta \in \mathbb{S}^2$,
\[ \lim_{K \rightarrow \infty} \frac{1}{K} \int_K^{2K} k^3 u^\infty_j(k, \theta, -\theta) \overline{u^\infty_l(k + \tau, \theta, -\theta)}  \, dk = 0\]
almost surely whenever $j + l = 3$ or $j = l = 2$.
\end{proposition}
Let us prove this proposition.
By the Cauchy--Schwarz inequality and changing variables, the modulus of \eqref{term:remainder} can be bounded by
\[ \bigg( \frac{1}{K} \int_K^{2K} k^3 |u^\infty_j(k, \theta, -\theta)|^2 \,dk  \bigg)^{1/2} \bigg( \frac{1}{K} \int_{K + \tau}^{2K +\tau} k^3 |u^\infty_l(k, \theta, -\theta)|^2  \,dk \bigg)^{1/2}. \]
Thus, in order to study the interactions between the single and 2nd order scattering is enough to consider
\[ \frac{1}{K} \int_K^{2K} k^3 |u^\infty_j(k, \theta, -\theta)|^2 \,dk \]
for any $K \geq 1$ with $j = 1, 2$. After identity \eqref{id:ergodicity} and \eqref{es:primera}, we know that
\begin{equation}
\lim_{K \rightarrow \infty} \frac{1}{K} \int_K^{2K} k^3 |u^\infty_1(k, \theta, -\theta)|^2 \,dk < \infty
\label{lim:|single|}
\end{equation}
almost surely. Therefore, in order to prove Proposition \ref{prop:2nd orderscattering}, it is enough to check that
\begin{equation}
\lim_{K \rightarrow \infty} \frac{1}{K} \int_K^{2K} k^3 |u^\infty_2(k, \theta, -\theta)|^2 \,dk = 0
\label{lim:2nd order}
\end{equation}
almost surely or alternatively,
\begin{equation}
\lim_{K \rightarrow \infty} \frac{1}{K - 1} \int_1^{K} k^3 |u^\infty_2(k, \theta, -\theta)|^2 \,dk = 0
\label{lim:2nd order'}
\end{equation}
almost surely.
Note that
\begin{equation*}
\frac{1}{K - 1} \int_1^K k^3 |u^\infty_2(k, \theta, -\theta)|^2 \,dk  =  \int_1^\infty \frac{k \mathbf{1}_{[1, K]} (k)}{K - 1} k^2 |u^\infty_2(k, \theta, -\theta)|^2 \,dk,
\end{equation*}
where $\mathbf{1}_{[1, K]}$ denotes the characteristic function of the interval $[1, K]$. Since $k \mathbf{1}_{[1, K]} (k)/(K - 1)$ converges point-wise to zero as $K$ goes to infinity, we have by the dominate convergence theorem that, if
\begin{equation*}
\int_{1}^\infty k^2 |u^\infty_2(k, \theta, -\theta)|^2 \,dk < \infty,
\end{equation*}
then \eqref{lim:2nd order'} holds. Obviously, by the continuity of the function $\theta \in \mathbb{S}^2 \mapsto u^\infty_2(k,\theta, -\theta) $, it will be enough to show that
\begin{equation}
\mathbb{E} \int_{\mathbb{S}^2} \int_1^\infty k^2 |u^\infty_2(k,\theta, -\theta)|^2 \, dk \, d\sigma(\theta) < \infty,
\label{es:boundedness}
\end{equation}
where $d\sigma$ denotes the volume form on $\mathbb{S}^2$. In order to prove \eqref{es:boundedness}, we will show that
\begin{equation}
\limsup_{\varepsilon \rightarrow 0}\mathbb{E} \int_{\mathbb{S}^2} \int_{1}^\infty k^2 |v_\varepsilon(k,\theta)|^2 \, dk \, d\sigma(\theta) < \infty
\label{es:boundednessSUP}
\end{equation}
with
\[v_\varepsilon(k, \theta) = \int_{\R^3} \int_{\R^3} e^{ik \theta \cdot (x + y)} q_\varepsilon(x) q_\varepsilon(y) \Phi^+_k (x - y) \, dx \, dy, \]
where $q_\varepsilon (\omega, x) = \langle q(\omega), \varphi_\varepsilon (x - \centerdot) \rangle$ and $\varphi_\varepsilon $ is as in \eqref{id:def_of_fepsilon}. Then, as a consequence of Fatou's lemma, we see that \eqref{es:boundedness} holds.

Let us prove \eqref{es:boundednessSUP}. Since in dimension $n = 3$, $ \Phi^+_k (x) $ is given by a constant multiple of $e^{ i k |x|}/|x|$, the approximation of the 2nd order backscattering $v_\varepsilon (k,\theta)$ is
\begin{equation}
\int_{\R^3} \int_{\R^3} e^{ik [\theta \cdot (x + y) + |x - y|]}  \frac{q_\varepsilon(x)q_\varepsilon(y)}{|x - y|} \, dx \, dy.
\label{term:2nd order-asymp}
\end{equation}
Changing variables, \eqref{term:2nd order-asymp} becomes
\[\int_{\R^3} \int_{\mathbb{S}^2} \int_0^\infty e^{ik (\theta \cdot z + \rho)}  q_\varepsilon \big( \frac{z + \rho \omega}{2} \big) q_\varepsilon \big( \frac{z - \rho \omega}{2} \big) \rho \, d\rho \,d\sigma(\omega) \, dz. \]
Denoting
\begin{equation}
f_\varepsilon (z, \rho) = \rho \mathbf{1}_{[0,\infty)}(\rho) \int_{\mathbb{S}^2} q_\varepsilon \big( \frac{z + \rho \omega}{2} \big) q_\varepsilon \big( \frac{z - \rho \omega}{2} \big) \,d\sigma(\omega),
\label{def:f_eps}
\end{equation}
we have that
\[v_\varepsilon(k,\theta) = \int_{\R^3} \int_\R e^{ik (\theta \cdot z + \rho)} f_\varepsilon(z, \rho) \, d\rho \, dz. \]
Note that that there exists an $R$ which depends on $D$ such that
\begin{equation}
|z| + |\rho| = \bigg| \frac{z + \rho \omega}{2} + \frac{z - \rho \omega}{2} \bigg| + \bigg| \frac{z + \rho \omega}{2} - \frac{z - \rho \omega}{2} \bigg| \leq  | z + \rho \omega | + | z - \rho \omega | \leq R
\label{es:support_q}
\end{equation}
whenever
\[ \frac{z + \rho \omega}{2},\, \frac{z - \rho \omega}{2} \in \bigcup_{\varepsilon \in (0,1]} \supp q_\varepsilon.  \]
Note that
\begin{equation}
\int_{\mathbb{S}^2} k^2 |v_\varepsilon(k,\theta)|^2 \, d\sigma(\theta) \sim \int_{\mathbb{S}^2} \Big| \int_{\R^3} e^{i k \theta \cdot z} T f_\varepsilon(z, k)  \, dz \Big|^2  k^2 \, d\sigma(\theta)
\label{equi:L2Sn-1}
\end{equation}
with \[ T f_\varepsilon(z, k) = \int_\R e^{ i k \rho} f_\varepsilon(z, \rho)  \, d\rho.\]

\begin{lemma}\label{lem:fourier_restriction}\sl The right hand-side of \eqref{equi:L2Sn-1} can be bounded by above as follows:
\[\int_{\mathbb{S}^2} \Big| \int_{\R^3} e^{i k \theta \cdot z} T f_\varepsilon(z, k)  \, dz \Big|^2  k^2 \, d\sigma(\theta) \lesssim \int_{\R^3} | T f_\varepsilon(z, k) |^2 \, dz \]
almost surely, where the implicit constant depends on $D$.
\end{lemma}
\begin{proof}
Start by noting  that the term to be estimated by above can be rewritten as
\begin{equation}
\int_{S_k} \Big| \int_{\R^3} e^{i \theta \cdot z} T f_\varepsilon(z, k)  \, dz \Big|^2 \, d\sigma_k(\theta)
\label{term:L2Sk}
\end{equation}
with $S_k = \{ x \in \R^3 : |x| = k  \}$ and $d\sigma_k$ denoting its volume form. The term \eqref{term:L2Sk} is equivalent to the square of the $L^2$ norm of the Fourier transform of $T f_\varepsilon(\centerdot, k)$ restricted to $S_k$. We will estimate it by duality: let $g$ be a smooth function on $S_k$, then
\[\int_{S_k} g(\theta) \int_{\R^3} e^{i \theta \cdot z} T f_\varepsilon(z, k)  \, dz \, d\sigma_k(\theta) = \int_{\R^3}  \int_{S_k}  e^{i \theta \cdot z} g(\theta)  \, d\sigma_k(\theta)\, T f_\varepsilon(z, k)  \, dz. \]
Since $\supp T f_\varepsilon(\centerdot, k) \subset \{ |z| < R \}$ with $R$ as in \eqref{es:support_q}, we have by Cauchy-Schwarz that
\begin{equation}
\begin{aligned}
\Big| \int_{S_k} g(\theta) \int_{\R^3} &e^{i \theta \cdot z} T f_\varepsilon(z, k)  \, dz \, d\sigma_k(\theta) \Big| \\
& \leq \Big\| \int_{S_k}  e^{i \theta \cdot z} g(\theta)  \, d\sigma_k(\theta) \Big\|_{L^2(|z|<R)} \| T f_\varepsilon(\centerdot, k) \|_{L^2}
\end{aligned}
\label{in:duality}
\end{equation}
almost surely. Consider $\{ \chi_j : j = 1, \dots, 6 \}$ a partition of unity of $S_k$ subordinated to the sets
\begin{align*}
\Gamma_{2l} = \{ x \in S_k :  2 \sqrt{3} x_{l} > k \}, & &
\Gamma_{2l - 1} = \{ x \in S_k :  2 \sqrt{3} x_{l} < -k \}
\end{align*}
for $l = 1, 2, 3$. Then,
\begin{equation}
\int_{S_k}  e^{i \theta \cdot z} g(\theta)  \, d\sigma_k(\theta) = \sum_{j = 1}^6 \int_{\R^2}  e^{i \vartheta_j(y) \cdot z} g_j(\vartheta_j(y)) \frac{k}{\sqrt{k^2 - |y|^2}} \, dy,
\label{id:partition}
\end{equation}
where $g_j = \chi_j g$ and
\begin{align*}
& \vartheta_1(y) = (- \sqrt{k^2 - |y|^2}, y_1, y_2), & & \vartheta_2(y) = (\sqrt{k^2 - |y|^2}, y_1, y_2), \\
 & \vartheta_3(y) = (y_1, - \sqrt{k^2 - |y|^2}, y_2),& & \vartheta_4(y) = (y_1, \sqrt{k^2 - |y|^2}, y_2), \\
 & \vartheta_5(y) = (y_1, y_2, - \sqrt{k^2 - |y|^2}), & & \vartheta_6(y) = (y_1, y_2, \sqrt{k^2 - |y|^2}).
\end{align*}
Note that every term on the sum of \eqref{id:partition} is a multiple of the two-dimensional Fourier transform of
\[y \longmapsto e^{i (-1)^{j(l)} z_l \sqrt{k^2 - |y|^2}} g_{j(l)}(\vartheta_{j(l)}(y)) \frac{k}{\sqrt{k^2 - |y|^2}}\]
evaluated at $\widehat{z_l}$, where $\widehat{z_1} = (z_2, z_3), \widehat{z_2} = (z_1, z_3) $ and $\widehat{z_3} = (z_1, z_2)$, and $j(l)$ stand for $2l-1$ or $2l$ with $l = 1,2,3$.

Letting $d\widehat{z_1}, d\widehat{z_2}$ and $d\widehat{z_3}$ denote $dz_2 dz_3, dz_1 dz_3 $ and $dz_1 dz_2$ respectively, we have, by the Plancherel identity in $\R^2$, that
\begin{equation}
\begin{aligned}
\int_{\R^2} \Big| \int_{\R^2}  e^{i \vartheta_{j(l)}(y) \cdot z} g_{j(l)}(\vartheta_{j(l)}(y)) \frac{k}{\sqrt{k^2 - |y|^2}} \, dy \Big|^2 \, d\widehat{z_l} & \lesssim \int_{\R^2}  |g_{j(l)}(\vartheta_{j(l)}(y))|^2 \frac{k}{\sqrt{k^2 - |y|^2}} \, dy \\
& \leq \int_{S_k}  | g(\theta) |^2  \, d\sigma_k(\theta),
 \end{aligned}
\label{es:plancherel}
\end{equation}
since
\[\frac{k}{\sqrt{k^2 - |y|^2}} < 2 \sqrt{3} \qquad \text{for}\, y \in \supp g_{j(l)} \circ \vartheta_{j(l)}. \]
Therefore, from \eqref{id:partition} and \eqref{es:plancherel}, we conclude that
\begin{equation}
\int_{|z| < R} \Big| \int_{S_k}  e^{i \theta \cdot z} g(\theta)  \, d\sigma_k(\theta) \Big|^2 \, dz \lesssim R \int_{S_k}  | g(\theta) |^2  \, d\sigma_k(\theta).
\label{in:extension}
\end{equation}
Finally, by duality, we can ensure that \eqref{term:L2Sk} is almost surely bounded by
\[ \| T f_\varepsilon(\centerdot, k) \|_{L^2}^2 \]
with a constant which depends on $R$. Therefore, the lemma is proven.
\end{proof}

After \eqref{equi:L2Sn-1} and Lemma \ref{lem:fourier_restriction} we obtain
\begin{equation*}
\int_{\mathbb{S}^2} \int_1^\infty k^2 |v_\varepsilon(k,\theta)|^2 \,dk \, d\sigma(\theta) \lesssim \int_{\R^3} \int_1^\infty | T f_\varepsilon(z, k) |^2  \, dk\, dz
\end{equation*}
almost surely 
with an implicit constant depending on the domain $D \subset \R^3$. Note that $T f_\varepsilon$ is a constant multiple of the inverse Fourier transform of $f_\varepsilon$ in the variable $\rho$, so we have that
\begin{equation}
\int_{\mathbb{S}^2} \int_1^\infty k^2 |v_\varepsilon(k,\theta)|^2 \,dk \, d\sigma(\theta) \lesssim \int_{\R^3} \int_\R |f_\varepsilon(z, \rho) |^2  \, d\rho\, dz
\label{es:after_restriction}
\end{equation}
almost surely 
by the Plancherel identity. By \eqref{es:after_restriction} and \eqref{es:support_q}, we can conclude that
\begin{equation}
\limsup_{\varepsilon \rightarrow 0} \mathbb{E} \int_{\mathbb{S}^2} \int_1^\infty k^2 |v_\varepsilon(k,\theta)|^2 \, dk \, d\sigma(\theta) \lesssim \lim_{\varepsilon \rightarrow 0} \int_{|z| < R} \int_{|\rho| < R} \mathbb{E} |f_\varepsilon(z, \rho) |^2 \, d\rho\, dz,
\label{es:limsup-lim}
\end{equation}
provided that the limit on the right-hand side exists. We now show that the limit exists and this equals
\begin{align*}
\int_{|z| < R} \int_{0 < \rho < R} \rho^2 \Bigg[& \int_{\mathbb{S}^2} \int_{\mathbb{S}^2} K_q \bigg( \frac{z + \rho \omega}{2} , \frac{z - \rho \omega}{2} \bigg) K_q \bigg( \frac{z + \rho \theta}{2}, \frac{z - \rho \theta}{2} \bigg) \,d\sigma(\omega) \,d\sigma(\theta)\\
& +\int_{\mathbb{S}^2} \int_{\mathbb{S}^2} K_q \bigg( \frac{z + \rho \omega}{2}, \frac{z + \rho \theta}{2} \bigg) K_q \bigg( \frac{z - \rho \omega}{2}, \frac{z - \rho \theta}{2} \bigg) \,d\sigma(\omega) \,d\sigma(\theta)\\
& +\int_{\mathbb{S}^2} \int_{\mathbb{S}^2} K_q \bigg( \frac{z + \rho \omega}{2}, \frac{z - \rho \theta}{2} \bigg) K_q \bigg( \frac{z - \rho \omega}{2}, \frac{z + \rho \theta}{2} \bigg) \,d\sigma(\omega) \,d\sigma(\theta) 
\Bigg]\, d\rho\, dz.
\end{align*}
Before proving this claim, note that this already ensures that \eqref{es:boundednessSUP} holds, since $K_q(x, y) = c_{n,m}\, \mu(x) \log|x - y| + F_\alpha(x, y)$, according to the Proposition \ref{prop:correlation_function}, and hence $K_q$ is integrable over $\mathbb{S}^2 \times \mathbb{S}^2$.

Finally, we show that the limit of the right-hand side of \eqref{es:limsup-lim} exists and we compute it. Start by noting that \eqref{def:f_eps} makes $\mathbb{E} |f_\varepsilon (z, \rho)|^2$ be equal to
\begin{equation*}
\rho^2 \mathbf{1}_{[0,\infty)}(\rho) \int_{\mathbb{S}^2} \int_{\mathbb{S}^2} \mathbb{E} \big[ q_\varepsilon \big( \frac{z + \rho \omega}{2} \big) q_\varepsilon \big( \frac{z - \rho \omega}{2} \big) q_\varepsilon \big( \frac{z + \rho \theta}{2} \big) q_\varepsilon \big( \frac{z - \rho \theta}{2} \big) \big] \,d\sigma(\omega) \,d\sigma(\theta).
\end{equation*}
By the Isserlis' theorem, this equals
\begin{align*}
\rho^2 \mathbf{1}_{[0,\infty)}(\rho) &\Bigg[ \int_{\mathbb{S}^2} \int_{\mathbb{S}^2} \mathbb{E} \bigg[ q_\varepsilon \bigg( \frac{z + \rho \omega}{2} \bigg) q_\varepsilon \bigg( \frac{z - \rho \omega}{2} \bigg) \bigg] \mathbb{E} \bigg[q_\varepsilon \bigg( \frac{z + \rho \theta}{2} \bigg) q_\varepsilon \bigg( \frac{z - \rho \theta}{2} \bigg) \bigg] \,d\sigma(\omega) \,d\sigma(\theta)\\
& +\int_{\mathbb{S}^2} \int_{\mathbb{S}^2} \mathbb{E} \bigg[ q_\varepsilon \big( \frac{z + \rho \omega}{2} \bigg) q_\varepsilon \bigg( \frac{z + \rho \theta}{2} \bigg) \bigg] \mathbb{E} \bigg[q_\varepsilon \bigg( \frac{z - \rho \omega}{2} \bigg) q_\varepsilon \bigg( \frac{z - \rho \theta}{2} \bigg) \bigg] \,d\sigma(\omega) \,d\sigma(\theta)\\
& +\int_{\mathbb{S}^2} \int_{\mathbb{S}^2} \mathbb{E} \bigg[ q_\varepsilon \bigg( \frac{z + \rho \omega}{2} \bigg) q_\varepsilon \bigg( \frac{z - \rho \theta}{2} \bigg) \bigg] \mathbb{E} \bigg[q_\varepsilon \bigg( \frac{z - \rho \omega}{2} \bigg) q_\varepsilon \bigg( \frac{z + \rho \theta}{2} \bigg) \bigg] \,d\sigma(\omega) \,d\sigma(\theta) 
\Bigg].
\end{align*}
Recalling \eqref{id:kernel_cov}, it is a simple observation to note that $ \mathbb{E} [ q_\varepsilon (x) q_\varepsilon (y) ] $ converges to $ K_q(x, y) $ point-wise as $\varepsilon$ vanishes. Moreover, using again that $K_q(x, y) = c_{n,m}\, \mu(x) \log|x - y| + F_\alpha(x, y)$, we can check that
\[\big| \mathbb{E} [ q_\varepsilon (x) q_\varepsilon (y) ] \big| \lesssim \big|\log |x - y|\big| \mathbf{1}_{\{ |x - y| < 1 \}} (x, y) + 1 \]
assuming $\supp \varphi \subset \{ |x| \leq 1/4 \}$, which can always be assumed. Hence, by applying the dominated convergence theorem to the integral we have that the limit of the right hand-side of \eqref{es:limsup-lim} exists and is the one claimed above.

To sum up, we have shown that \eqref{es:boundednessSUP} holds and consequently,
\begin{equation*}
\lim_{K \rightarrow \infty} \frac{1}{K} \int_K^{2K} k^3 |u^\infty_2(k, \theta, -\theta)|^2 \,dk = 0
\end{equation*}
almost surely for every $\theta \in \mathbb{S}^2$. This ends the proof of Proposition \ref{prop:2nd orderscattering}.

\subsection{Multiple backscattering}\label{subsec:higher_order}

Again, we only consider the effects of multiple scattering under the assumptions $m = n = 3$. In that case, the realizations of $q$ are almost surely in $L^p_{-s}(\R^3)$ for $1 < p < \infty$ and $s>0$. According to the previous decomposition on single and 2nd order backscattering, we can write
\begin{align*}
\frac{1}{K} & \int_K^{2K} k^3 u^\infty(k, \theta, -\theta) \overline{u^\infty(k + \tau, \theta, -\theta)}  \, dk \\
& = \sum_{1 \leq j \leq 2} \sum_{1 \leq l \leq 2} \frac{1}{K} \int_K^{2K} k^3 u^\infty_j(k, \theta, -\theta) \overline{u^\infty_l(k + \tau, \theta, -\theta)}  \, dk \\
& \quad + \frac{1}{K} \int_K^{2K} k^3 \Big(u^\infty(k, \theta, -\theta) - \sum_{1 \leq j \leq 2} u^\infty_j(k, \theta, -\theta) \Big) \overline{u^\infty(k + \tau, \theta, -\theta)}  \, dk \\
& \quad + \sum_{1 \leq j \leq 2} \frac{1}{K} \int_K^{2K} k^3 u^\infty_j(k, \theta, -\theta) \Big( \overline{u^\infty(k + \tau, \theta, -\theta)} - \sum_{1 \leq l \leq 2} \overline{u^\infty_l(k + \tau, \theta, -\theta)} \Big) \, dk.
\end{align*}
The first term on the right-hand side of the previous inequality corresponds to the single and 2nd order backscattering terms studied in the sections \ref{subsec:single} and \ref{subsec:2nd order}. The other two terms describe the multiple backscattering and will be shown here that they are negligible, that is, they vanish as $K$ grows. Applying the Cauchy--Schwarz inequality and changing variable as in the proof of Proposition \ref{prop:2nd orderscattering}, we only need to check that
\begin{align}
& \lim_{K \rightarrow \infty} \frac{1}{K} \int_K^{2K} k^3 |u^\infty(k , \theta, -\theta)|^2  \, dk < \infty, \label{term:full_scattering} \\
& \lim_{K \rightarrow \infty} \frac{1}{K} \int_K^{2K} k^3 \Big| u^\infty(k, \theta, -\theta) - \sum_{1 \leq l \leq 2} u^\infty_l(k, \theta, -\theta) \Big|^2 \, dk = 0 \label{term:multiple_scattering}
\end{align}
hold almost surely.
Note that \eqref{term:full_scattering} follows from \eqref{term:multiple_scattering}, \eqref{lim:|single|} and \eqref{lim:2nd order}. Thus, it is enough to show that \eqref{term:multiple_scattering} holds. This is a straight consequence of the next lemma:
\begin{lemma}
\label{lem:residual_decay}
\sl Suppose that $\expec q = 0$. We have that
\[ \sup_{\theta \in \mathbb{S}^2} \Big| u^\infty(k, \theta, -\theta) - \sum_{1 \leq l \leq 2} u^\infty_l(k, \theta, -\theta) \Big| = o(k^{-2(1 - 3s)})\]
almost surely.
\end{lemma}
\begin{proof}
By \eqref{term:FULLfar-filed} and \eqref{term:INDIfar-field}, and then Lemma \ref{lem:KP}, we have that
\begin{align*}
\Big| u^\infty(k, \theta, -\theta) - \sum_{1 \leq l \leq 2} u^\infty_l(k, \theta, -\theta) \Big| & \lesssim \Big| \int_{\R^3} e^{ik\theta \cdot y} q(y) (u_{\rm sc}(y) - u_1(y)) \, dy \Big| \\
& \lesssim \| q \|_{L^p_{-s}} \| e^{i k \theta \cdot y}  \chi \|_{H^s} \| \chi (u_{\rm sc} - u_1) \|_{H^s},
\end{align*}
where $\chi$ is a smooth function with compact support such that $\chi(x) = 1$ for every $x$ in a ball containing $D$. A direct computation shows that $ \| e^{i k \theta \cdot y}  \chi \|_{H^s} = \mathcal{O}(k^s)$. On the other hand,
\[\| \chi (u_{\rm sc} - u_1) \|_{H^s} \leq \| \chi (u_{\rm sc} - u_1) \|_{H^s_k} \lesssim \| u_{\rm sc} - u_1 \|_{H^{s, - \delta}_k}\]
by Lemma \ref{lem:compact}. By Theorem \ref{thm:relosvent} and Proposition \ref{prop:multiplication}, we know that the operator $\mathcal{R}^+_k \circ q$ maps $H^{s,-\delta}_k(\R^3)$ into itself, for $k \geq k_0(\omega)$ almost surely, with a norm 
\[\|\mathcal{R}^+_k \circ q\|_{H^{s,-\delta}_k \rightarrow H^{s,-\delta}_k} = o(k^{-(1-2s)}).\]
Therefore, $u_{\rm sc} - u_1 = \sum_{j > 1} (\mathcal{R}^+_k \circ q)^j u_0$ can be bounded as follows:
\[\| u_{\rm sc} - u_1 \|_{H^{s, - \delta}_k} \leq \sum_{j > 1} \|\mathcal{R}^+_k \circ q\|^j_{H^{s,-\delta}_k \rightarrow H^{s,-\delta}_k} \| \chi u_0 \|_{H^{s,-\delta}_k} = o( k^s k^{-2(1 - 2s)} ) \]
almost surely.
This concludes the proof of this lemma.
\end{proof}

\begin{proof}[Proof of Theorem \ref{thm:main_theorem_rigorous} in the case $\mathbb{E}q=0$.]
Following the discussion in the beginning of this section we notice that we finally possess all necessary tools to prove the main theorem in the case $\mathbb{E}q = 0$. Having established the well-posedness of the forward problem and the measurement in the sections \ref{sec:random_potential} and \ref{sec:direct-scattering}, we have shown in Corollary \ref{cor:first_order_data} that the first order contribution $M_1(\tau,\theta)$ coincides almost surely (and up to a multiplicative constant) with $\widehat{\mu}(2\tau \theta)$ at fixed $\tau\geq 0$ and $\theta\in{\mathbb S}^2$. In Proposition \ref{prop:2nd orderscattering} and this section, we have proven that in $\R^3$ the contribution from the second and higher order scattering vanishes in $M(\tau,\theta,-\theta)$. We can now proceed similarly to the proof of Theorem \ref{thm:main_theorem_born_approx} and repeat the measurement at a countable dense set
$\{(\tau_j,\theta_j)\}_{j\in\N} \subset \R_+ \times \mathbb{S}^2$. Finally, continuation from a dense set yields the result.
\end{proof}

\subsection{Non-zero-mean potentials}\label{sec:non-zero-mean}
This section is devoted to extend the proof in sections \ref{subsec:single}, \ref{subsec:2nd order} and \ref{subsec:higher_order} to the case of non-zero-mean potentials. We proceed pointing out the places where some changes have to be made.

\subsubsection{Single backscattering} As in the section \ref{subsec:single},  the first goal is to show that the identity \eqref{id:ergodicity} holds when $\mathbb{E} q \neq 0$. As in the zero-mean case, this will be a consequence of the fact that
\begin{equation}
\lim_{K \to \infty} \frac{1}{K} \int_K^{2K} k^m (W_k^2 - \mathbb{E}W_k^2) \,dk = 0
\label{lim:vanishing}
\end{equation}
almost surely for $W_k$ as in \eqref{eq:list_of_random_variables}. The difference now is that $\mathbb{E}W_k \neq 0$, so we write
\[W_k^2 - \mathbb{E}W_k^2 = Z_k^2 - \mathbb{E}Z_k^2 + 2\mathbb{E}W_k Z_k \]
with $Z_k = W_k - \mathbb{E} W_k$. We will prove that \eqref{lim:vanishing} holds showing that
\begin{align}
\lim_{K \to \infty} &\frac{1}{K} \int_K^{2K} k^m (Z_k^2 - \mathbb{E}Z_k^2) \,dk = 0 \quad {\rm and}
\label{lim:vanishingZk} \\
\lim_{K \to \infty} &\frac{1}{K} \int_K^{2K} k^m \mathbb{E}W_k Z_k \,dk = 0
\label{lim:crossed}
\end{align}
almost surely. In order to check \eqref{lim:vanishingZk}, we use Theorem \ref{thm:aux_ergodicity} verifying that
there exists an $\varepsilon > 0$ such that
\begin{equation}\label{es:sufCONDZ_k}
\big|\mathbb{E} \big(k^m (Z_k^2 - \mathbb{E}Z_k^2 ) (k + r)^m (Z_{k + r}^2 - \mathbb{E}Z_{k + r}^2 ) \big)\big| \lesssim (1 + r)^{-\varepsilon}
\end{equation}
for all $k \geq 1$ and $r \geq 0$. To do so, we observe that $(Z_k, Z_{k + r})$ is a centred Gaussian random vector and use Lemma \ref{lemma:gaussian_pair}. Thus, instead of condition \eqref{es:sufCONDZ_k} we just need to see that
there exists $\varepsilon > 0$ such that
\begin{equation}
	\label{eq:decay_of_Zk}
	|\mathbb{E} (k^{m/2}(k+r)^{m/2}Z_k Z_{k + r}) | \lesssim (1 + r)^{-\varepsilon}.
\end{equation}
The inequality \eqref{eq:decay_of_Zk} follows from \eqref{es:correlUk1Uk2}, \eqref{es:correlVk1Vk2} and \eqref{es:correlUk1Vk2} in the next proposition, and therefore \eqref{lim:vanishingZk} holds.

\begin{proposition} \sl Let $q$ be the potential given by Definition \ref{def:ml_iso}. There exists a known constant $c_{n,m}$, depending on $n$ and $m$, such that
\begin{equation}
\mathbb{E}(u_1^ \infty(k, \theta, -\theta) \overline{u_1^ \infty(k + \tau, \theta, -\theta)}) = c_{n,m}  k^{-m} \widehat{\mu} (2\tau\theta) + \mathcal{O}(k^{-m-1})
\label{id:approxCORRELATIONnonzero}
\end{equation}
for $k \geq 1/2$ and $\tau \geq 0$. Moreover, for all $k_1, k_2 > 0$, we have that
\begin{align}
\label{es:correlUk1Uk2}
\big|\mathbb{E}\big( (U_{k_1} - \mathbb{E}U_{k_1}) (U_{k_2} - \mathbb{E}U_{k_2}) \big) \big| &\lesssim k_1^{-m} (1 + |k_1 - k_2|)^{-N}\\
\label{es:correlVk1Vk2}
\big|\mathbb{E}\big( (V_{k_1} - \mathbb{E}V_{k_1}) (V_{k_2} - \mathbb{E}V_{k_2}) \big) \big| &\lesssim k_1^{-m}(1 + |k_1 - k_2|)^{-N}\\
\label{es:correlUk1Vk2}
\big|\mathbb{E}\big( (U_{k_1} - \mathbb{E}U_{k_1}) (V_{k_2} - \mathbb{E}V_{k_2}) \big) \big| &\lesssim k_1^{-m}(1 + |k_1 - k_2|)^{-N}
\end{align}
for all $N \in \N$.
\end{proposition}
\begin{proof}
Note that
\begin{equation}
\begin{aligned}
\mathbb{E}(\langle q - \mathbb{E}q, e^{i2k_1\theta \cdot y} \rangle & \overline{\langle q - \mathbb{E}q, e^{i2k_2\theta \cdot x} \rangle}) \\
& \quad = \int_{\R^n} \Big( \int_{\R^n} K_q(x,y) e^{-i2 k_1 \theta \cdot(x-y)} \,dy \Big)  e^{-i2 (k_2 - k_1) \theta \cdot x} \, dx.
\end{aligned}
\label{id:startingPOINTnonzero}
\end{equation}
By \eqref{id:expansion}, we have that
\begin{equation}
\begin{aligned}
\mathbb{E}(\langle q - \mathbb{E}q, e^{i2k_1\theta \cdot y} \rangle & \overline{\langle q - \mathbb{E}q, e^{i2k_2\theta \cdot x} \rangle}) \\ &= \int_{\R^n} \mu(x) |2 k_1|^{-m} e^{-i2 (k_2 - k_1) \theta \cdot x} \,dx + \int_D a(x, 2 k_1 \theta) e^{-i2 (k_2 - k_1) \theta \cdot x} \,dx
\end{aligned}\label{eq:principalSYM}
\end{equation}
for $k_1 \geq 1/2$. Furthermore, 
\begin{equation}
\begin{aligned}
\mathbb{E}(\langle q, e^{i2k_1\theta \cdot y} \rangle &  \overline{\langle q, e^{i2k_2\theta \cdot x} \rangle}) \\
&= \mathbb{E}(\langle q - \mathbb{E}q, e^{i2k_1\theta \cdot y} \rangle  \overline{\langle q - \mathbb{E}q, e^{i2k_2\theta \cdot x} \rangle}) + \langle \mathbb{E}q, e^{i2k_1\theta \cdot y} \rangle  \overline{\langle \mathbb{E}q, e^{i2k_2\theta \cdot x} \rangle}
\end{aligned}
\label{id:q=q-Eq+Eq}
\end{equation}
with $\mathbb{E}q$ smooth and compactly supported. By \eqref{eq:principalSYM} and the non-stationary phase principle, we have that \eqref{id:approxCORRELATIONnonzero} holds.

On the other hand, by \eqref{id:startingPOINTnonzero} and \eqref{id:pre-expansion}, we have
\[\mathbb{E}(\langle q - \mathbb{E}q, e^{i2k_1\theta \cdot y} \rangle \overline{\langle q - \mathbb{E}q, e^{i2k_2\theta \cdot x} \rangle}) = \int_{\R^n} c_q(x, 2 k_1 \theta) \chi(x)  e^{-i2 (k_2 - k_1) \theta \cdot x} \, dx\]
with $\chi \in C^\infty_0(\R^n)$ such that $\chi(x) = 1$ for all $x \in D$, which implies, by the non-stationary phase principle, that
\begin{equation}
|\mathbb{E}(\langle q - \mathbb{E}q, e^{i2k_1\theta \cdot y} \rangle \overline{\langle q - \mathbb{E}q, e^{i2k_2\theta \cdot x} \rangle})| \lesssim (1 + k_1)^{-m} (1 + |k_1 - k_2|)^{-N}
\label{es:primeranonzero}
\end{equation}
for all $N \in \N$. By the same kind of considerations, one can proves that
\[\mathbb{E}(\langle q - \mathbb{E}q, e^{i2k_1\theta \cdot y} \rangle \langle q - \mathbb{E}q, e^{i2k_2\theta \cdot x} \rangle) = \int_{\R^n} c_q(x, 2 k_1 \theta) \chi(x)  e^{i2 (k_1 + k_2) \theta \cdot x} \, dx\]
with $\chi \in C^\infty_0(\R^n)$ such that $\chi(x) = 1$ for all $x \in D$, which implies, again by the non-stationary phase principle, that
\begin{equation}
|\mathbb{E}(\langle q - \mathbb{E}q, e^{i2k_1\theta \cdot y} \rangle \langle q - \mathbb{E}q, e^{i2k_2\theta \cdot x} \rangle)| \lesssim (1 + k_1)^{-m} (1 + k_1 + k_2)^{-N}
\label{es:segundanonzero}
\end{equation}
for all $N \in \mathbb{N}$.  Finally, the estimates \eqref{es:correlUk1Uk2}, \eqref{es:correlVk1Vk2} and \eqref{es:correlUk1Vk2} follow from \eqref{es:primeranonzero} and \eqref{es:segundanonzero}.
\end{proof}

Eventually, we prove that \eqref{lim:crossed} holds. To do so, we check that the quasi-orthogonality condition in Theorem \ref{thm:aux_ergodicity} holds:
there exists $\varepsilon > 0$ such that
\begin{equation}
	|\mathbb{E} (k^m \mathbb{E} W_k Z_k (k+r)^m \mathbb{E} W_{k + r} Z_{k + r}) | \lesssim (1 + r)^{-\varepsilon}.
	\label{es:quasi-ortho}
\end{equation}
Indeed, by \eqref{eq:decay_of_Zk} we have
\begin{equation*}
	|\mathbb{E} (k^m \mathbb{E} W_k Z_k (k+r)^m \mathbb{E} W_{k + r} Z_{k + r}) | \lesssim (1 + r)^{-\varepsilon} k^{m/2} (k+r)^{m/2} \mathbb{E} W_k \mathbb{E} W_{k + r}.
\end{equation*}
Since $\mathbb{E}q$ is smooth and compactly supported, we have by the non-stationary phase principle that
\[ k^{m/2} (k+r)^{m/2} \mathbb{E} W_k \mathbb{E} W_{k + r} \lesssim 1,\]
and consequently \eqref{es:quasi-ortho}. Then, by Theorem \ref{thm:aux_ergodicity}, we can ensure that \eqref{lim:crossed} holds.

Summarizing, \eqref{lim:vanishingZk} and \eqref{lim:crossed} imply that \eqref{lim:vanishing} holds, and so does \eqref{id:ergodicity}. On the other hand, by \eqref{id:approxCORRELATIONnonzero} we can conclude that Corollary \ref{cor:first_order_data} also holds in the case that $\mathbb{E}q$ is smooth and compactly supported in $D$. Finally, Theorem \ref{thm:main_theorem_born_approx} follows in the case $\mathbb{E}q \neq 0$ by the same density argument performed in the section \ref{subsec:single}.

\subsubsection{2nd order backscattering} The goal of this section is to prove that Proposition \ref{prop:2nd orderscattering} holds when $\mathbb{E}q$ is smooth and has support in $D$. As in the case $\mathbb{E}q = 0$, it is enough to show that
\begin{align}
& \lim_{K \rightarrow \infty} \frac{1}{K} \int_K^{2K} k^3 |u^\infty_1(k, \theta, -\theta)|^2 \,dk < \infty \quad {\rm and}
\label{lim:|single|nonzero} \\
& \lim_{K \rightarrow \infty} \frac{1}{K} \int_K^{2K} k^3 |u^\infty_2(k, \theta, -\theta)|^2 \,dk = 0
\label{lim:2nd order_nonzero}
\end{align}
hold almost surely. The finiteness of the limit in \eqref{lim:|single|nonzero} is a consequence of \eqref{id:q=q-Eq+Eq}, \eqref{es:primeranonzero} and the non-stationary phase principle. However, showing that \eqref{lim:2nd order_nonzero}, or alternatively
\begin{equation*}
\lim_{K \rightarrow \infty} \frac{1}{K - 1} \int_1^{K} k^3 |u^\infty_2(k, \theta, -\theta)|^2 \,dk = 0,
\end{equation*}
holds almost surely requires a more subtle argument. Fortunately, this is exactly the same as in the case $\mathbb{E}q = 0$, and it reduces to prove that 
\begin{equation}
\lim_{\varepsilon \rightarrow 0} \int_{|z| < R} \int_{|\rho| < R} \mathbb{E} |f_\varepsilon(z, \rho) |^2 \, d\rho\, dz
\label{lim:f_eps}
\end{equation}
exists and is finite, with $f_\varepsilon(z, \rho)$ as in \eqref{def:f_eps}. It is convenient to write
\begin{align*}
f_\varepsilon (z, \rho) =& \rho \mathbf{1}_{[0,\infty)}(\rho)
 \int_{\mathbb{S}^2} (q_\varepsilon - \mathbb{E} q_\varepsilon) \big( 
\frac{z + \rho \omega}{2} \big) (q_\varepsilon - \mathbb{E}
q_\varepsilon) \big( \frac{z - \rho \omega}{2} \big) \,d\sigma(\omega) \\
&+ 2 \rho \mathbf{1}_{[0,\infty)}(\rho) \int_{\mathbb{S}^2} 
(q_\varepsilon - \mathbb{E} q_\varepsilon) \big( \frac{z + \rho \omega}
{2} \big) \mathbb{E} q_\varepsilon \big( \frac{z - \rho \omega}{2} \big) 
\,d\sigma(\omega) \\
&+ \rho \mathbf{1}_{[0,\infty)}(\rho) \int_{\mathbb{S}^2} \mathbb{E} 
q_\varepsilon \big( \frac{z + \rho \omega}{2} \big) \mathbb{E} 
q_\varepsilon \big( \frac{z - \rho \omega}{2} \big) \,d\sigma(\omega).
\end{align*}
Thus,
\[\mathbb{E} |f_\varepsilon(z, \rho) |^2 = \rho^2 \mathbf{1}_{[0,\infty)}(\rho)
 \int_{\mathbb{S}^2} \int_{\mathbb{S}^2} \sum_{j=1}^4 I_j (z, \rho, \omega, \theta) \,d\sigma(\omega) \,d\sigma(\theta) \]
 with
\begin{align*}
I_4 =& \mathbb{E} \Big [ (q_\varepsilon - \mathbb{E} q_\varepsilon) \big( 
\frac{z + \rho \omega}{2} \big) (q_\varepsilon - \mathbb{E}
q_\varepsilon) \big( \frac{z - \rho \omega}{2} \big) (q_\varepsilon - \mathbb{E} q_\varepsilon) \big( 
\frac{z + \rho \theta}{2} \big) (q_\varepsilon - \mathbb{E}
q_\varepsilon) \big( \frac{z - \rho \theta}{2} \big)  \Big],\\
I_3 =& 4 \mathbb{E} \Big [ (q_\varepsilon - \mathbb{E} q_\varepsilon) \big( 
\frac{z + \rho \omega}{2} \big) (q_\varepsilon - \mathbb{E}
q_\varepsilon) \big( \frac{z - \rho \omega}{2} \big) (q_\varepsilon - \mathbb{E} q_\varepsilon) \big( 
\frac{z + \rho \theta}{2} \big)  \Big] \mathbb{E}
q_\varepsilon \big( \frac{z - \rho \theta}{2} \big),\\
I_2 =& 2 \mathbb{E} \Big [ (q_\varepsilon - \mathbb{E} q_\varepsilon) \big( 
\frac{z + \rho \omega}{2} \big) (q_\varepsilon - \mathbb{E}
q_\varepsilon) \big( \frac{z - \rho \omega}{2} \big)  \Big] \mathbb{E} q_\varepsilon \big( \frac{z + \rho \theta}{2} \big) \mathbb{E}
q_\varepsilon \big( \frac{z - \rho \theta}{2} \big)\\
& + 4 \mathbb{E} \Big [ (q_\varepsilon - \mathbb{E} q_\varepsilon) \big( 
\frac{z + \rho \omega}{2} \big) (q_\varepsilon - \mathbb{E}
q_\varepsilon) \big( \frac{z + \rho \theta}{2} \big)  \Big] \mathbb{E} q_\varepsilon \big( \frac{z - \rho \omega}{2} \big) \mathbb{E}
q_\varepsilon \big( \frac{z - \rho \theta}{2} \big),\\
I_1 =& \mathbb{E} q_\varepsilon \big( 
\frac{z + \rho \omega}{2} \big) \mathbb{E} q_\varepsilon \big( \frac{z - \rho \omega}{2} \big) \mathbb{E} q_\varepsilon \big( \frac{z + \rho \theta}{2} \big) \mathbb{E}
q_\varepsilon \big( \frac{z - \rho \theta}{2} \big).
\end{align*}
Applying Isserlis' theorem on the terms $I_3$ and $I_4$, we see that $I_3 = 0$ and $I_4$ can be studied in the same way as we did at the end of the section \ref{subsec:2nd order}. Finally, the terms with $I_2$ and $I_1$ give no problem when showing the existence and finiteness of \eqref{lim:f_eps}  since $\mathbb{E} [ (q_\varepsilon - \mathbb{E} q_\varepsilon) (x) (q_\varepsilon - \mathbb{E}
q_\varepsilon) (y)]$ converges to $K_q (x,y)$ pointwise as $\varepsilon$ vanishes and
\[\big| \mathbb{E} [ (q_\varepsilon - \mathbb{E} q_\varepsilon) (x) (q_\varepsilon - \mathbb{E}
q_\varepsilon) (y)] \big| \lesssim 1 + \big|\log |x - y|\big| \mathbf{1}_{\{ |x - y| < 1 \}} (x, y). \]

\subsubsection{Multiple backscattering} As we argued in the section \ref{subsec:higher_order}, we just needed to show that \eqref{term:full_scattering} and \eqref{term:multiple_scattering} hold almost surely. Note that \eqref{term:full_scattering} follows from \eqref{term:multiple_scattering}, \eqref{lim:|single|nonzero} and \eqref{lim:2nd order_nonzero}. In turn, \eqref{term:multiple_scattering} follows from Lemma \ref{lem:residual_decay}. As discussed in the proof of Theorem \ref{thm:main_theorem_rigorous} for the case $\mathbb{E}q = 0$, we can conclude now this theorem for non-zero-mean potentials.

\appendix
\section{Scaling regimes}
\label{sec:scaling_regimes}

\subsection{Two-scale model of a non-smooth scatterer}
In this paper we have set up a theoretical framework for the inversion of the probabilistic backscattering problem that hopefully helps to inspire stable imaging algorithms. Here, we briefly consider two practical quantities ---the correlation length and the characteristic size of the potential--- and how their multi-scale analysis can be used to estimate the {statistics of error in the reconstruction in terms of these
quantities.} Unlike in the previous sections, 
\emph{our analysis is not rigorous in this appendix} as we apply several approximations often used in multi-scale analysis in physics literature. {In this appendix our aim is to formulate our previous
results in the terminology used in multi-scale analysis. Also, we will also show how our microlocal techniques can be applied also in much more general setting than is done in the main text of the paper when
one adds several scales of orders in the scattering model and makes
certain approximations. Indeed, using those approximations, we can consider 
the multi-scale analysis in arbitrary dimension $n$ and for general potential models.}

{Before we formulate the results, we make a few remarks.}
First, in any practical measurement device the maximum frequency $K$ in the data is bounded, i.e., we can never have infinite precision. In order to estimate how wide frequency band is needed to achieve reasonable accuracy we need to understand effective scales of frequency correlations in the backscattered far field. 
Second, in physical applications one often has some prior information about the correlations within the random scatterer. This information is typically described via so-called correlation length of the random media.
Third, if the potential $q$ has a form $q(x)=\epsilon Q(x)$ and the random field $Q$ is roughly of a constant order, we say that the parameter $\epsilon>0$ is the \emph{characteristic size} of $q$.
The effect of higher order scattering in the Schr\"{o}dinger
equation
\begin{equation*}
	(\Delta + k^2 + \epsilon Q(x)) u(x) = 0
\end{equation*}
becomes negligible (problem is approximately linear) if the characteristic size $\epsilon$ is small compared to other parameters appearing in the system. 

In what follows, we employ scaling regimes that are characterized by three quantities described above: the (effective) length of the \emph{frequency band} $K$, the \emph{correlation length} $\ell$ and the \emph{characteristic size} of the potential $\epsilon$.
Moreover, we consider the following scaling regime with respect to the relative sizes of $K, \ell $ and $\epsilon$: for the reference  wavelength $K_0$, distance of propagation $L_0$ and correlation $\ell_0$ scales we have
\begin{equation}
	\label{eq:scaling_regimes}
	K \gg K_0 \max\left(\left(\frac{\ell}{\ell_0}\right)^{-\beta_1},\left(\frac{L}{L_0}\right)^n\right) \quad {\rm and} \quad  \epsilon \ll  \left(\frac{K}{K_0}\right)^{-\beta_2},
\end{equation}
with some $\beta_1>1$ and $\beta_2> m/2-1$,
where $L = {\rm diam}(D)$ {is the diameter of the domain where the potential $q$ is supported} and $m$ is the order of the covariance operator of $Q$. 
For convenience, we assume below that the reference scales $K_0$ and $\ell_0$ are of constant order can be neglected from the error estimate analysis.

In the scaling regime given in \eqref{eq:scaling_regimes} we show below that the measurement data with
a  {finite frequency band},
\begin{equation}
	\label{eq:approx_measurement_data}
	M_K(\theta, \tau, \omega) =\frac 1 K\int_K^{2K} k^m u^\infty(k,\theta,-\theta) \overline{u^\infty(k+\tau,\theta,-\theta)} dk,
\end{equation}
{is close to}  the ideal data $M(\theta,\tau,\omega)$ in \eqref{eq:intro_measurement_data}. Therefore, the local strength $\mu(x)$ can be {approximately estimated
from $M_K(\theta, \tau, \omega)$.}
 
At this point, notice carefully that in the specific case $n=m=3$ the higher order scattering can be analysed rigorously (as in the section \ref{sec:inverse_problem}) and we need no prior assumptions regarding the characteristic size $\epsilon$, {that is, the latter inequality in (\ref{eq:scaling_regimes}) is not needed.} 
In the same spirit, the estimates regarding the second order scattering (and therefore for the full non-linear scattering) can be improved by techniques used in the section \ref{subsec:2nd order}. 
{We emphasize that in the case  $n=m=3$} the main results of the paper apply for general potentials and do not require the form assumed below.

Let us now define the two-scale model that we analyse in detail.

\begin{definition}[Two-scale model]\sl
\label{def:one_scale}
Consider a microlocally isotropic random field 
\begin{equation}
	\label{eq:one_scale_model}
	q(x) = \epsilon \sqrt{\mu(x)} Q\left(\frac x\ell\right),
\end{equation}
{where $\mu \in C_0^\infty(D)$ is the local strength} function and $Q$ is a stationary zero-mean microlocally isotropic field of order $m$ such that
\begin{equation*}
	\sigma_p(\xi) = |\xi|^{-m} + a(\xi),
\end{equation*}
where $a\in {\mathcal S}^{-m-1}(\R^n)$, the correlation function $K_Q(z) = \expec \left(Q(x)Q(x-z)\right)$ satisfies $K_Q\in L^1(\R^n)$ and $Q$ has a correlation length of constant order.
\end{definition}
Notice that realizations of the random field $Q$ are not compactly supported whereas the realizations of $q$ are due to the compact support of $\mu$. For orders $m>n-1$, {such that $m\neq n$}, we know according to Proposition \ref{prop:correlation_function} and since $Q$ stationary that
\begin{equation*}
	\expec(Q(x)Q(y)) = K_Q(x-y) = c|x-y|^{m-n} + F_\alpha(x-y),
\end{equation*}
where $c\in \R$ and $F_\alpha$ is smooth (for case $m=n$, the leading term is logarithmic). An example of a Gaussian random process with such asymptotics is given by covariance function
\begin{equation*}
	K_Q(x-y) = \exp\left(-|x-y|^{m-n}\right)
\end{equation*}
for $m>n$. 
For more discussion of random processes of type \eqref{eq:one_scale_model}, consider Example \ref{example:model} and examples given in \cite{HLP,LPS}.

\begin{remark}\sl
In the case $m>n$ the correlation length can be
defined by
\begin{equation}
	\label{eq:correlation_length}
	L_Q(x)=\bigg|\frac {\int_{\R^n} \expec(Q(x)Q(x+y)) dy}{\expec(Q(x)^2)}\bigg|^{1/n}.
\end{equation}
for a random field $Q$ such that $\expec(Q(x)Q(x+\cdot)) \in L^1(\R^n)$.

Since we have
\begin{equation*}
	K_{q}(x,y) = \epsilon^2 \sqrt{\mu(x)\mu(y)} K_Q\left(\frac{x-y}{\ell}\right)
\end{equation*}
and $K_{q}(x,x) = \epsilon^2 \mu(x) K_Q(0)$, we find that
\begin{equation*}
	L_q(x) = \left|\frac{\int_{\R^n} \sqrt{\mu(x+y)} K_Q\left(\frac y\ell\right) dy}{\sqrt{\mu(x)} K_Q(0)} \right|^{\frac 1n}
	= \ell \left|\frac{\sqrt{\mu(x)} \int_{\R^n} K_Q\left(y\right) dy + {\mathcal O}\left(\ell\right)}{\sqrt{\mu(x)} K_Q(0)} \right|^{\frac 1n} = \ell L_Q(x) + {\mathcal O}(\ell^2).
\end{equation*}
Therefore, the correlation length of process $Q_\ell$ is of order $\ell$.
\end{remark}

Consider the covariance operator $C_{Q_\ell}$ of random field $Q(x/\ell)$.
Its symbol $c_{Q_\ell}$ satisfies (in the sense of generalized functions)
\begin{eqnarray*}
	c_{Q_\ell}(x,\xi) & = & \int_{\R^n} K_Q\left(\frac x\ell ,\frac y\ell\right) \exp(-i\xi\cdot(x-y)) dy \\
	& = & \ell^n \int_{\R^n} K_Q\left(\frac x \ell, y'\right) \exp\left(-i\ell\xi \cdot \left(\frac x\ell - y'\right)\right)d y' \\
	& = & \ell^n c_Q(\ell \xi),
\end{eqnarray*}
since $c_Q$ is independent on the spatial variable. Consequently, due to the additional multiplication by $\epsilon \sqrt{\mu(x)}$ the symbol of $C_q$ is of the form
\begin{equation}
	\label{eq:one_scale_model_symbol}
	c_{q}(x,\xi) = \epsilon^2\ell^{n}\left(\mu(x) |\ell\xi|^{-m} + \tilde a(x,\ell\xi)\right),
\end{equation}
where $\tilde a \in {\mathcal S}^{-m-1}(\R^n\times \R^n)$ and $\tilde a$ has a compact support with respect to $x$.

%
%
%

\subsection{Inverse scattering with the two-scale model}

{Next we analyze the statistical properties of 
 the measurement data with
a finite frequency band, $M_K(\theta, \tau, \omega)$, given in (\ref{eq:approx_measurement_data}).
Let us start by considering} the correlations for first order scattering from potential model described in Definition \ref{def:one_scale}. By utilizing the asymptotic of the symbol $c_q$ in \eqref{eq:one_scale_model_symbol}, we obtain
\begin{eqnarray}\nonumber 
\mathbb{E}(u_1^ \infty(k, \theta, -\theta) \overline{u_1^ \infty(k+\tau, \theta, -\theta)})
 &=& \epsilon^2\ell^{n}\int_{\R^n}    \left(\mu(x) |2k\ell|^{-m} + \tilde a(x,2k\ell\theta)\right)e^{-i2 \tau \theta \cdot x} \, dx 
 \\ 
 &=& \epsilon^2 \ell^{n-m}(2k)^{-m} \left(  \hat \mu(2 \tau \theta) + L{\mathcal O}\left(\frac 1{k\ell}\right)\right),
 \label{correlation}
\end{eqnarray}
where the second term appears {due to the fact that  $\tilde a$ is supported in $D\times \R^2$.}
This identity corresponds to the formula \eqref{id:approxCORRELATION} and is valid in our scaling regime
in equation \eqref{eq:scaling_regimes}.

Using similar technique as in Proposition \ref{prop:first_order_decay} and utilizing non-stationary phase principle (more precisely, see equations \eqref{es:primera} and \eqref{es:segunda}), we obtain inequality
\begin{equation}
	\label{eq:scaling_decay_rates_1}
	\left|\mathbb{E}(u_1^ \infty(k, \theta, -\theta) \overline{u_1^ \infty(k+\tau, \theta, -\theta)})	\right|
	\lesssim \epsilon^2 \ell^{n-m} L^n (1+k)^{-m}(1+\tau)^{-N},
\end{equation}
where $L={\rm diam}(D)$ and $N>0$ is arbitrary and the implicit constant depends on $N$. The bounding constant in non-stationary phase principle also depends on the domain (and the phase, which is here trivial), which we have explicitly included on the right-hand side. Similar deduction also yields the bound
\begin{equation}
	\label{eq:scaling_decay_rates_2}
	\left|\mathbb{E}(u_1^ \infty(k, \theta, -\theta) u_1^ \infty(k+\tau, \theta, -\theta))	\right|
	\lesssim \epsilon^2 \ell^{n-m} L^n (1+k)^{-m}(1+k+\tau)^{-N}.
\end{equation}


Now let us turn our attention to the convergence of the data (\ref{eq:approx_measurement_data}), i.e., the mean integral over the frequency band and, for a brief moment, consider a real-valued random process $X(k)$ as in Theorem \ref{thm:aux_ergodicity}. It is well-known  that the variance of the mean integral
\begin{equation*}
	Y(K) = \frac{1}{K}\int_K^{2K} (X(k)-\expec X(k)) dk
\end{equation*}
decays at rate ${\mathcal O}(1/K)$, if the covariance of $X$ decays at a polynomial rate
\begin{equation}
	\label{eq:scaling_regimes_aux1}
	\expec (X(k)-\expec X(k)) (X(k+\tau)-\expec X(k+\tau)) \lesssim (1+\tau)^{-N}
\end{equation}
for $N>2$.
Similar arguments carry over to complex valued random processes, when the modulus of $Y(K)$ and decay of the complex covariance are considered.

Notice that due to Corollary \ref{cor:covariances_in_full_first_order_integral} the decay rates in equations \eqref{eq:scaling_decay_rates_1} and \eqref{eq:scaling_decay_rates_2} imply similar decay of covariance for the product
$$Z(k) = k^m u^\infty_1(k,\theta,-\theta)\overline{u^\infty_1(k + \tau, \theta, -\theta)}.$$
The first order backscattered term $u^\infty_1$ is a linear mapping of the potential and {we find that 
\begin{equation}\label{1st order sc.}
	M_K^{(1)}(\theta,\tau,\omega) := \frac{1}{K} \int_{K}^{2K} Z(k)  \, dk  =c_{n,m}\epsilon^2\left( \widehat{\mu} (2\tau\theta) + N_1(\omega)\right),
\end{equation}
where $c_{n,m}\in \R$, $\mu$ is  the local strength of the random field $q$} and the error term $N_1(\omega)$ can be factorized in  random  and deterministic parts
\begin{equation*}
	N_1(\omega) = E(\omega) + F.
\end{equation*}
The error term $E: \Omega \to {\mathbb C}$ represents the random deviation from expected value
\begin{equation*}
	E(\omega) = \frac{1}{\epsilon^2}( M_K - \expec M_K)
\end{equation*}
and is a zero-mean random variable satisfying 
\begin{equation*}
	\expec |E|^2 \lesssim \frac{\ell^{n-m} L^n}{K}.
\end{equation*}
The deterministic part $F$ represents approximation
\begin{equation}\label{F error}
	|F| = \frac{1}{\epsilon^2}\Big| \expec M_K - M(\theta,\tau,\omega)\Big| \lesssim \ell^{n-m}L^n \cdot \frac{\log K}{K\ell}
\end{equation}
due to the Proposition \ref{prop:first_order_decay}. In the particular case of $n=m$, the error in the first order backscattered data is roughly of order (standard deviation of the random error and deterministic error)
\begin{equation}
	\label{eq:scaling_first_order_error_std}
	\expec |N_1|^2 \lesssim \expec |E|^2 + F^2 \lesssim  \frac{L^n}{K} + L^{2n}\left(\frac{\log K}{K\ell}\right)^2.
\end{equation}
{Summarizing, the inequality (\ref {eq:scaling_first_order_error_std}) 
gives an estimate the how well  the first order scattering term $
	M_K^{(1)}(\theta,\tau,\omega)$ in (\ref{1st order sc.}) estimates the local strength $\mu$ of
	the random field $q$. Next we consider the effect of the higher order scattering.}

\subsection{Effects from the higher order scattering}

Our analysis in the section \ref{sec:inverse_problem} regarding the full non-linear backscattering assumes the specific case $n=m=3$. 
Therefore, under general conditions and, in particular, general dimension the data $M_K(\theta,\tau,\omega)$ from full backscattered field does not necessarily converge when $K$ increases. Below we show that in the scaling regime \eqref{eq:scaling_regimes} the smallness of the potential reduces the effect of the higher-order scattering, which is rather straightforward to quantify given the techniques in the proof of Lemma \ref{lem:residual_decay}. However, in addition, we wish to quantify the effect of correlation length, which requires more care.

Let us first consider the effect of correlation length in the model \ref{def:one_scale} and state the following well-known fact: the dilation operator $D_\ell : f \mapsto f\left(\frac \cdot \ell\right)$ is bounded in $L^p_{-s}(\R^n)$ and its norm scales according to
\begin{equation}
	\label{eq:one_scale_dilation_ineq}
	\norm{D_\ell f}_{L^p_{-s}} \lesssim \ell^{\frac np}\norm{f}_{L^p_{-s}} 
\end{equation}
for $\ell\leq 1$ and $s\geq 0$. Same inequality holds as well for $\ell \geq 1$ and $s\leq 0$. The inequality can be shown to hold by first noting that the standard norm of $L^p_{s}(\R^n)$ and $\norm{\cdot}_{L^p(\R^n)} + \norm{(-\Delta)^{s}\cdot}_{L^p(\R^n)}$ are equivalent for $s\geq 0$. Thereafter, one has directly
$(-\Delta)^s (D_\ell u) = \ell^{-2s} D_\ell(\Delta u)$ for any $s\in\N$ and by utilizing interpolation and duality argument, the result follows.

Let $\Psi \in C_0^\infty(\R^n)$ be a smooth function such that $\supp (\Psi) \subset \left[-1 ,1\right]^n$. Moreover, we assume that functions $\Psi_{\vj} := \Psi(\cdot - \vj)$ define a partition of unity, i.e.
\begin{equation}
	\label{eq:one_scale_part_unity_in_D}
	\sum_{\vj\in\Z^n} \Psi_{\vj}(x) = 1 \quad \textrm{for all}\; x\in \R^n.
\end{equation}
Next, let $R(\ell) \in \N$ be the (smallest) number of translations in one dimension that are needed to cover the dilated domain $\frac 1 \ell \cdot D$, i.e.
we assume
\begin{equation}
	\label{eq:one_scale_vj_decomp}
	\sum_{|\vj|\leq R(\ell)} \Psi_{\vj}(x) = 1 \quad \textrm{for all}\; x \; \textrm{such that} \; \ell x \in D.
\end{equation}
Clearly, we have that $R(1) \lesssim {\rm diam}(D)$ and $R(\ell) \lesssim \frac{{\rm diam}(D)} \ell$.

In order to quantify the effect of $\ell$ in the Born series of the data, we need a scaling estimate of $q$ in the norm of $L^p_{-s}(\R^n)$ as well as an operator norm of $f\mapsto qf$.
Since the local strength $\mu\in C_0^\infty(D)$ is bounded, it follows by \eqref{eq:one_scale_dilation_ineq} that
\begin{multline}
	\label{eq:one_scale_norm}
	\norm{\sqrt{\mu}Q\left(\frac \cdot\ell\right)}_{L^p_{-s}} \lesssim \ell^{\frac n p} 
	\norm{\sqrt{\mu}(\ell \cdot )Q}_{L^p_{-s}}
	\lesssim \ell^{\frac n p} \sum_{|\vj|\leq R(\ell)} \norm{\Psi_{\vj}\, Q}_{L^p_{-s}} \\
	\lesssim \ell^{\frac n p-n} \frac{L^n}{R(\ell)^n} \sum_{|\vj|\leq R(\ell)} \norm{\Psi_{\vj}\, Q}_{L^p_{-s}} = G_1(\omega)L \ell^{\left(\frac 1 p-1\right)n}.
\end{multline}
Above, {the random bound} $$G_1(\omega) = \frac 1{R(\ell)^n} \sum_{|\vj|\leq R(\ell)} \norm{\Psi_{\vj}\, Q}_{L^p_{-s}}$$
is bounded almost surely, since the expectation $\expec G_1$ is bounded. Furthermore, $\expec G_1$ is independent of $\ell$, since random variables $\norm{\Psi_{\vj}\, Q}_{L^p_{-s}}$ are identically distributed (although not independent).

Considering the multiplication operator $f\mapsto qf$, we reproduce the ideas of the section \ref{sec:multiplication} by replacing $V$ with $ \sqrt{\mu}Q\left(\frac \cdot\ell\right)$.
By dilation, boundedness of $\mu$ and identity \eqref{eq:one_scale_part_unity_in_D}, we can directly record the following identity
\begin{equation}
	\label{eq:one_scale_aux_dil1}
	\left|\langle \sqrt{\mu}Q\left(\frac \cdot\ell\right) f, g\rangle\right| = \ell^n \left|\langle \sqrt{\mu}(\ell \cdot) Q, f(\ell\cdot) g(\ell \cdot)\rangle\right| \lesssim \ell^n \left|\left\langle \sum_{|\vj|\leq R(\ell)}\Psi_\vj\, Q, f(\ell\cdot) g(\ell \cdot)\right\rangle\right|.
\end{equation}
We write
\begin{equation*}
	 W_\vj = (\id - \Delta)^{-s/2} \left(\Psi_\vj\, Q\right) \quad {\rm and} \quad 
	W = \sum_{|\vj|\leq R(\ell)} W_\vj,	 
\end{equation*}
and define $W^\sharp$ and $W^\flat$ according to the section \ref{sec:multiplication}.
Following the idea in inequality \eqref{eq:multiplication_op_holder_est} and applying the estimate \eqref{eq:one_scale_aux_dil1}, decomposition \eqref{eq:one_scale_vj_decomp} and dilation scaling \eqref{eq:one_scale_dilation_ineq}, we obtain
\begin{multline*}
\left|\langle \sqrt{\mu}Q\left(\frac \cdot\ell\right) f, g\rangle\right| \\
\leq \ell^n\left( \| (\id - \Delta)^{t/2} W^\sharp \|_{L^r}  \| (\id - \Delta)^{(s - t)/2} (f(\ell\cdot) g(\ell \cdot)) \|_{L^{r'}} + \| W^\flat \|_{L^p} \| (\id- \Delta)^{s/2} (f(\ell\cdot) g(\ell \cdot)) \|_{L^{p'}}\right) \\
\lesssim \ell^n\left( \ell^{\frac{n}{r'}} \| (\id - \Delta)^{t/2} W^\sharp \|_{L^r}  \| (\id - \Delta)^{(s - t)/2} (f g) \|_{L^{r'}} + \ell^{\frac{n}{p'}}\| W^\flat \|_{L^p} \| (\id- \Delta)^{s/2} (f g) \|_{L^{p'}}\right),
\end{multline*}
where $s-\frac n p < t < s$, $r = \frac{n}{s-t}$ and $p$ and $r$ are H\"older conjugates of $p'$ and $r'$, respectively. 
Since $\frac n{p'} < \frac n{r'}$, we have that for small $\ell<1$ the term $\ell^{n/p'}$ dominates and
\begin{eqnarray*}
	\left|\langle \sqrt{\mu}Q\left(\frac \cdot\ell\right) f, g\rangle\right| & \lesssim & 
	\ell^{n(2-\frac 1p)} \sum_{|\vj|\leq R(\ell)} \big( C(k) \| W_\vj\|_{L^p} + \| W_\vj^\flat \|_{L^p} \big)  \| f \|_{H^{s,-\delta}_k} \| g \|_{H^{s,-\delta}_k} \\
	& \lesssim & L^n\ell^{n(1-\frac 1p)} \frac 1{R(\ell)^n} \sum_{|\vj|\leq R(\ell)} \big( C(k) \| W_\vj\|_{L^p} + \| W_\vj^\flat \|_{L^p} \big)  \| f \|_{H^{s,-\delta}_k} \| g \|_{H^{s,-\delta}_k} \\
	& \simeq & L^n\ell^{n(1-\frac 1p)} G_2(\omega,\ell)  \| f \|_{H^{s,-\delta}_k} \| g \|_{H^{s,-\delta}_k},
\end{eqnarray*}
where 
\begin{equation*}
	G_2(\omega,\ell) = \frac 1{R(\ell)^n} \sum_{|\vj|\leq R(\ell)} \big( C(k) \| W_\vj\|_{L^p} + \| W_\vj^\flat \|_{L^p} \big)
\end{equation*}
is almost surely finite, since $\expec G_2 < \infty$ and $C(k)$ decays to zero as $k$ increases.
In consequence, we obtain
\begin{equation}
	\label{eq:one_scale_multiplication_estimate}
	\norm{\sqrt{\mu}Q\left(\frac \cdot\ell\right)f}_{H^{-s, \delta}_k} 
	= L^n \ell^{n(1-\frac 1p)} G_2(\omega,\ell)\norm{f}_{H^{s, -\delta}_k}.
\end{equation}

We are finally ready to analyse the effect from higher order scattering in the measurement data.
Let us write $$u^\infty_R(k,\theta,-\theta) = \sum_{j=2}^\infty u_j^\infty(k,\theta,-\theta).$$ Similarly to the proof of Lemma \ref{lem:residual_decay} we can apply the norm bounds obtained in equations \eqref{eq:one_scale_norm} and \eqref{eq:one_scale_multiplication_estimate} and further in the section \ref{sec:direct-scattering} to arrive at
\begin{eqnarray}
	\label{eq:scaling_regimes_aux3}
	\sup_{\theta\in\mathbb{S}^2} \left| u^\infty_R(k,\theta,-\theta)\right| & \leq & \nonumber
	\epsilon \norm{ \sqrt{\mu}Q\left(\frac \cdot\ell\right)}_{L^p_{-s}} \| \exp(i k \theta \cdot y)  \chi \|_{H^s} \|\chi u_0 \|_{H^{s,-\delta}_k} \\
	& & \quad \times \sum_{j > 0} \epsilon^{j}\norm{\mathcal{R}^+_k \circ \left(\sqrt{\mu}Q\left(\frac \cdot\ell\right)\right)}^j_{H^{s,-\delta}_k \rightarrow H^{s,-\delta}_k} \nonumber \\
	& \lesssim & G_1(\omega) L^n \epsilon \ell^{-(1-\frac 1p)n} k^{2s} \sum_{j>0} \left( G_2(\omega,\ell) L^n \epsilon \ell^{(1-\frac 1p)n} k^{-(1-2s)}\right)^j \nonumber \\
	& \leq & G_3(\omega,\ell) L^{2n} \epsilon^2 \frac{k^{-1+4s}}{1-G_2(\omega,\ell)L^n \epsilon \ell^{(1-\frac 1p)n} k^{-1+2s}} \nonumber\\
	& \leq & G_4(\omega,\ell) L^{2n}\epsilon^2 k^{-1+4s},
\end{eqnarray}
where expectations of random coefficients $G_j$, $j=1,...,4,$ are uniformly bounded by a constant independent of $\ell$ and $s>0$ is arbitrarily small. Therefore, for $k$ large enough so that $1-G_2(\omega,\ell)L^n\epsilon \ell^{(1-\frac 1p)n} k^{-1+2s}>\frac 12$, we have an upper bound
\begin{eqnarray*}
	\frac{1}{K} \int_{K}^{2K} k^m \left| u^\infty_R(k, \theta, -\theta)\right|^2  \, dk  & \lesssim & \frac{G_4(\omega,\ell)L^{4n}\epsilon^4 }{K}\int_{K}^{2K} k^{m} k^{-2+\delta} dk \\
	& \simeq & G_4(\omega,\ell)L^{4n}\epsilon^4  K^{m-2+\delta}
\end{eqnarray*}
for some small $\delta=4s>0$.

By combining the two approximations and {formulas (\ref{1st order sc.}) and
(\ref {eq:scaling_first_order_error_std}),
we see using the Cauchy--Schwarz inequality that
the finite frequency band $	M_K(\theta, \tau, \omega) $ given in (\ref{eq:approx_measurement_data}) satisfies}
\begin{equation}\label{MK estimate}
	M_K(\theta,\tau,\omega) =
	  	\epsilon^2 (\widehat{\mu} (2\tau\theta) + N_1(\omega) + N_2(\omega))
\end{equation}
for
\begin{equation*}
	N_2(\omega) = G_4(\omega,\ell)L^{4n}\epsilon^2 K^{m-2+\delta}.
\end{equation*}
In fact, since $G_4$ can be considered to be of constant order we have that the full error in $M_K(\theta,\tau,\omega)$ (when scaled by $\epsilon^2$) in the case $n=m$ is of order
\begin{equation}\label{N1N2 estimate}
	\expec |N_1 + N_2|^2 \lesssim  \frac{L^n} K + L^{2n} \left(\frac{\log K}{K\ell}\right)^2
	+ L^{4n} \epsilon^2  K^{n-2+\delta}.
\end{equation}
where $\delta>0$ is arbitrarily small. {Thus we can summarise the main result of this appendix:
Under the assumptions on scaling regimes \eqref{eq:scaling_regimes},
the finite frequency band measurement $M_K(\theta, \tau, \omega) $, given in (\ref{eq:approx_measurement_data}),
determines the Fourier transform of local strength $\mu(x)$ by the formula
\begin{equation}\label{MK estimate 2}
	\widehat{\mu} (2\tau\theta) 
		=	\epsilon^{-2}M_K(\theta,\tau,\omega)+\mathcal E,\quad \mathcal E=-\epsilon^{-2}(N_1(\omega) + N_2(\omega))
\end{equation}
where the error  $\mathcal E$ can be estimated using the formula (\ref{N1N2 estimate}). In the scaling regime \eqref{eq:scaling_regimes}
 the error term $\mathcal E$ is negligible with high probability. In consequence, we have means to approximate the Fourier transform of $\mu$ at the given frequency $2\tau$. Note that as $\mu\in C^\infty_0(D)$, the 
 Fourier transform of $\mu$  decays rapidly, and therefore the formula  (\ref{MK estimate 2}) and the error estimate
  (\ref{N1N2 estimate}) are useful for small values of $\tau$, e.g. $|\tau|\ll K(L_0/L)^n$.}

\section{Random variables with Gaussian probability laws}\label{sec:gaussians}
Let $X$ be a random variable, we say that it has a Gaussian law with mean $\mu$ and variance $\sigma^2$, if its law $\mathbb{P}_X$ satisfies:
\[\mathbb{P}_X \big( (-\infty, a] \big) = \frac{1}{\sqrt{2\pi \sigma^2}} \int_{-\infty}^a e^{-\frac{(x - \mu)^2}{2\sigma^2}} \, dx, \qquad \forall a \in \R.\]

\begin{lemma}\label{lem:evenMOMENTUM} \sl Let $X$ be a Gaussian random variable with mean $0$. Then, there exists a constant $c_k > 0$ such that
\[(\mathbb{E} X^{2k})^\frac{1}{2k} = c_k (\mathbb{E} X^2)^\frac{1}{2}, \qquad \forall k\in \N \setminus \{ 0 \}.\]
\end{lemma}
\begin{proof}
By homogeneity, one can assume that the variance be $\mathbb{E} X^2 = 1$. Integrating by parts,
\begin{equation*}
\mathbb{E} X^{2k} = \int_\R x^{2k} e^{-\frac{x^2}{2}} \, dx = - \int_\R x^{2k -1} \frac{d}{dx} \Big( e^{-\frac{x^2}{2}} \Big) \, dx = (2k - 1) \int_\R x^{2(k - 1)} e^{-\frac{x^2}{2}} \, dx.
\end{equation*}
Iterating the process we see that $\mathbb{E} X^{2k} = c'_k $ for some constant which only depends on $k$. This concludes the proof of this lemma.
\end{proof}

\begin{proposition} \sl The linear combination of two independent Gaussian random variables has a Gaussian law.
\end{proposition}

Given two random variables $X_1$ and $X_2$, the pair $X = (X_1, X_2)$ is said to be a Gaussian random vector if $r_1X_1 + r_2X_2$ has a Gaussian law for any $r_1, r_2 \in \R$. If $X$ is a Gaussian random vector, its law is determined by the vector $(\mathbb{E}X_1, \mathbb{E} X_2)$ and the matrix
\[
\left(
\begin{array}{c c}
\mathrm{Cov} (X_1, X_1) & \mathrm{Cov} (X_1, X_2) \\
\mathrm{Cov} (X_2, X_1) & \mathrm{Cov} (X_2, X_2)
\end{array}
\right).
\]

\begin{remark} \rm
\begin{itemize}
\item[(i)] If $X = (X_1, X_2)$ is a Gaussian random vector, then $X_1$ and $X_2$ are Gaussian random variables.

\item[(ii)] If $X_1$ and $X_2$ are independent Gaussian random variables, then $X=(X_1, X_2)$ is a Gaussian random vector.

\item[(iii)] The pair $X = (X_1, X_2)$ of two Gaussian random variables $X_1$ and $X_2$ is not in general a Gaussian vector. To see this, is enough to consider a Gaussian random variable $X$ with mean $0$ and variance $1$ and the random variable $X_c$ defined as $X_c(\omega) = X(\omega)$ if $X(\omega) > c$ and $X_c(\omega) = - X(\omega)$ if $X(\omega) \leq c$. One can check that $X_c$ has the same law than $X$. However, $X - X_c$ does not have a Gaussian law since $\mathbb{P}_{X - X_c} (\{ 0 \}) > 0$.

\item[(iv)] Two Gaussian random variables $X_1$ and $X_2$ may be uncorrelated (i.e. $\mathrm{Cov} (X_1, X_2) = 0$) and not be independent. This can be verified using the same $X$ and $X_c$ of (iii), which are not independent. To do so, note that the function $c \mapsto \mathbb{E}(X X_c)$ is continuous and
\[\lim_{c \to \mp \infty} \mathbb{E} (X X_c) = \pm 1.\]
Thus, there exists a $c_0$ such that $\mathbb{E} (X X_{c_0}) = 0$ and the random variables $X$ and $X_{c_0}$ are uncorrelated. 

\item[(v)] Let $X_1, X_2$ and $X_2'$ be three Gaussian random variables equally distributed and assume that $\mathbb{E}(X_1 X_2) = \mathbb{E}(X_1 X'_2)$. In general, it is not true that $(X_1, X_2)$ and $(X_1, X'_2)$ are equally distributed. We show that using the example in (iii). Let $X'$ be a Gaussian random variable with mean $0$ and variance $1$ and assume it to be independent of $X$. Then, the vectors $(X, X_{c_0})$ and $(X, X')$ do not have the same law, since $(X, X')$ is Gaussian and $(X, X_{c_0})$ is not.
\end{itemize}
\end{remark}

\bibliographystyle{abbrv}

\bibliography{citations}

\end{document}